\documentclass[11pt]{article}
\pdfoutput=1

\usepackage[margin=1in]{geometry}

\usepackage[utf8]{inputenc}
\usepackage[T1]{fontenc}
\usepackage{lmodern}
\usepackage{indentfirst}
\usepackage{tcolorbox}
\usepackage[flushleft]{threeparttable}
\usepackage{booktabs}
\usepackage{multirow}  
\usepackage{makecell} 
\usepackage{bbding}

\usepackage{amsmath, amssymb, amsthm, mathrsfs, bm}
\usepackage[dvipsnames]{xcolor}
\usepackage{graphicx}
\usepackage{float}
\usepackage{caption}
\usepackage[labelformat=simple]{subcaption} 
\usepackage{multirow, tablefootnote, colortbl, hhline} 
\usepackage{algorithm}
\usepackage{algorithmic}
\usepackage[framemethod=default]{mdframed}
\usepackage{enumitem}
\usepackage{lineno} 
\usepackage{natbib}
\usepackage[colorlinks,
            linkcolor=red,
            anchorcolor=blue,
            citecolor=blue,
            urlcolor=blue]{hyperref}

\def\vzero{\mathbf{0}}
\def\vone{\mathbf{1}}
\def\va{\mathbf{a}}
\def\vb{\mathbf{b}}
\def\vc{\mathbf{c}}
\def\vd{\mathbf{d}}
\def\ve{\mathbf{e}}
\def\vf{\mathbf{f}}
\def\vg{\mathbf{g}}
\def\vh{\mathbf{h}}
\def\vi{\mathbf{i}}
\def\vj{\mathbf{j}}
\def\vk{\mathbf{k}}
\def\vl{\mathbf{l}}
\def\vm{\mathbf{m}}
\def\vn{\mathbf{n}}
\def\vo{\mathbf{o}}
\def\vp{\mathbf{p}}
\def\vq{\mathbf{q}}
\def\vr{\mathbf{r}}
\def\vs{\mathbf{s}}
\def\vt{\mathbf{t}}
\def\vu{\mathbf{u}}
\def\vv{\mathbf{v}}
\def\vw{\mathbf{w}}
\def\vx{\mathbf{x}}
\def\vy{\mathbf{y}}
\def\vz{\mathbf{z}}

\def\vmu{\bm{\mu}}
\def\vphi{\bm{\phi}}
\def\veta{\bm{\eta}}
\def\valpha{\bm{\alpha}}
\def\vbeta{\bm{\beta}}
\def\vtheta{\bm{\theta}}
\def\vxi{\bm{\xi}}

\def\fA{{\mathcal{A}}}
\def\fB{{\mathcal{B}}}
\def\fC{{\mathcal{C}}}
\def\fD{{\mathcal{D}}}
\def\fE{{\mathcal{E}}}
\def\fF{{\mathcal{F}}}
\def\fG{{\mathcal{G}}}
\def\fH{{\mathcal{H}}}
\def\fI{{\mathcal{I}}}
\def\fJ{{\mathcal{J}}}
\def\fK{{\mathcal{K}}}
\def\fL{{\mathcal{L}}}
\def\fM{{\mathcal{M}}}
\def\fN{{\mathcal{N}}}
\def\fO{{\mathcal{O}}}
\def\fP{{\mathcal{P}}}
\def\fQ{{\mathcal{Q}}}
\def\fR{{\mathcal{R}}}
\def\fS{{\mathcal{S}}}
\def\fT{{\mathcal{T}}}
\def\fU{{\mathcal{U}}}
\def\fV{{\mathcal{V}}}
\def\fW{{\mathcal{W}}}
\def\fX{{\mathcal{X}}}
\def\fY{{\mathcal{Y}}}
\def\fZ{{\mathcal{Z}}}

\def\sA{{\mathbb{A}}}
\def\sB{{\mathbb{B}}}
\def\sC{{\mathbb{C}}}
\def\sD{{\mathbb{D}}}

\def\BE{{\mathbb{E}}}
\def\BB{{\mathbb{B}}}
\def\BF{{\mathbb{F}}}
\def\BG{{\mathbb{G}}}
\def\BH{{\mathbb{H}}}
\def\BI{{\mathbb{I}}}
\def\BJ{{\mathbb{J}}}
\def\BK{{\mathbb{K}}}
\def\BL{{\mathbb{L}}}
\def\BM{{\mathbb{M}}}
\def\BN{{\mathbb{N}}}
\def\BO{{\mathbb{O}}}
\def\BP{{\mathbb{P}}}
\def\BQ{{\mathbb{Q}}}
\def\BR{{\mathbb{R}}}
\def\BS{{\mathbb{S}}}
\def\BT{{\mathbb{T}}}
\def\BU{{\mathbb{U}}}
\def\BV{{\mathbb{V}}}
\def\BW{{\mathbb{W}}}
\def\BX{{\mathbb{X}}}
\def\BY{{\mathbb{Y}}}
\def\BZ{{\mathbb{Z}}}

\def\mA{\mathbf{A}}
\def\mB{\mathbf{B}}
\def\mC{\mathbf{C}}
\def\mD{\mathbf{D}}
\def\mE{\mathbf{E}}
\def\mF{\mathbf{F}}
\def\mG{\mathbf{G}}
\def\mH{\mathbf{H}}
\def\mI{\mathbf{I}}
\def\mJ{\mathbf{J}}
\def\mK{\mathbf{K}}
\def\mL{\mathbf{L}}
\def\mM{\mathbf{M}}
\def\mN{\mathbf{N}}
\def\mO{\mathbf{O}}
\def\mP{\mathbf{P}}
\def\mQ{\mathbf{Q}}
\def\mR{\mathbf{R}}
\def\mS{\mathbf{S}}
\def\mT{\mathbf{T}}
\def\mU{\mathbf{U}}
\def\mV{\mathbf{V}}
\def\mW{\mathbf{W}}
\def\mX{\mathbf{X}}
\def\mY{\mathbf{Y}}
\def\mZ{\mathbf{Z}}
\newcommand{\norm}[1]{\left\| #1 \right\|}

\DeclareMathOperator*{\argmin}{argmin}
\DeclareMathOperator*{\argmax}{argmax}

\newtheorem{theorem}{Theorem}
\newtheorem{lemma}{Lemma}
\newtheorem{proposition}{Proposition}
\newtheorem{definition}{Definition}
\newtheorem{corollary}[theorem]{Corollary}
\newtheorem{remark}{Remark}

\newtheorem{assumption}{Assumption}

\begin{document}
\title{Nonsmooth Nonconvex-Concave Minimax Optimization: Convergence Criteria and Algorithms
}

\author{
Jinyang Shi
\thanks{School of Mathematical Science, Fudan University; email: jinyangshi22@m.fudan.edu.cn}
\and
Luo Luo
\thanks{School of Data Science, Fudan University; email: luoluo@fudan.edu.cn}
}

\maketitle

\begin{abstract}
This paper considers constrained stochastic nonsmooth minimax optimization problem of the form $\min_{\mathbf{x}\in\mathcal{X}}\max_{\mathbf{y}\in\mathcal{Y}}f\left(\mathbf{x},\mathbf{y}\right)=\mathbb{E}[F(\mathbf{x},\mathbf{y};\mathbf{\xi})]$, 
where the objective $f(\mathbf{x},\mathbf{y})$ is concave in~$\mathbf{y}$ but possibly nonconvex in $\mathbf{x}$, the stochastic component $F(\mathbf{x},\mathbf{y};\mathbf{\xi})$ indexed by random variable~$\mathbf{\xi}$ is mean-squared Lipschitz continuous, and the feasible sets $\fX$ and $\fY$ are convex and compact.
We introduce the notion of $(\eta_x,\eta_y,\delta,\epsilon)$-Goldstein saddle stationary point (GSSP) to characterize the convergence for solving constrained nonsmooth minimax problems. 
We then develop projected gradient-free descent ascent methods for finding $(\eta_x,\eta_y,\delta,\epsilon)$-GSSPs of the objective function $f(\mathbf{x},\mathbf{y})$ with non-asymptotic convergence rates.
We further propose nested-loop projected gradient-free descent ascent methods to establish the non-asymptotic convergence for finding $(\eta,\delta,\epsilon)$-generalized Goldstein stationary points (GGSP) \citep{liu2024zeroth} of the primal function $\Phi(\mathbf{x})\triangleq\max_{\mathbf{y}\in\mathcal{Y}}{f}\left(\mathbf{x},\mathbf{y}\right)$.
It is worth noting that our algorithm designs and theoretical analyses do not require additional assumptions such as the weak convexity used in prior works on nonsmooth minimax optimization \citep{lin2025two,boct2023alternating}.
\end{abstract}

\section{Introduction}

We consider the following constrained stochastic nonsmooth minimax optimization problem
\begin{equation}\label{prob:constrained main}
\min_{\vx\in\mathcal{X}}\max_{\vy\in\mathcal{Y}}f(\vx,\vy)\triangleq\mathbb{E}[F(\vx,\vy;\vxi)],
\end{equation}
where the objective $f(\vx,\vy)$ is concave in $\vy$ but possibly nonconvex in $\vx$, the stochastic component $F(\vx,\vy;\vxi)$ is indexed by random variable $\vxi\sim\fD$, and feasible sets $\fX\subseteq\BR^{d_x}$ and~$\fY\subseteq\BR^{d_y}$ are convex and compact. 
The nonconvex-concave minimax optimization includes broad applications in machine learning, such as generative adversarial networks \citep{goodfellow2014generative,arjovsky2017wasserstein}, distributional robust optimization \citep{rahimian2019distributionally,jin2021non,sinha2017certifying,levy2020large,gao2023distributionally,biggio2012poisoning}, AUC maximization \citep{guo2023fast,liu2019stochastic,yuan2021federated,zhang2023federated}, and reinforcement learning \citep{jin2020efficiently,qiu2020single,wai2018multi,zhang2019policy}.

Most existing studies on nonconvex-concave minimax optimization focus on the problem with smooth objective function \citep{lin2020gradient,jin2020local,thekumparampil2019efficient,zhang2021complexity,lin2020near,lu2020hybrid,chen2021cubic,rafique2022weakly,huang2022accelerated,luo2022finding,luo2020stochastic,xu2023unified,xu2024derivative,wang2024efficient,yang2023accelerating,li2021complexity,zhang2020single,mahdavinia2022tight}. 
For example, \citet{lin2020gradient} showed that gradient descent ascent (GDA) and its stochastic variants can achieve the approximate stationary points of smooth nonconvex-concave minimax problems with non-asymptotic convergence rates.  
Later, \citet{lin2020near} provided inexact proximal point methods to achieve the improved first-order oracle complexity, nearly matching the lower bound in the specific nonconvex-strongly-concave setting \citep{zhang2021complexity,li2021complexity}.
In addition, \citet{xu2024derivative,huang2022accelerated,wang2023zeroth,xu2020gradient} incorporated the two-point gradient estimate \citep{nesterov2017random} into GDA and its variants, establishing efficient zeroth-order methods for  nonconvex-concave minimax optimization.

In real-world applications, the objectives in minimax formulations are usually nonsmooth
\citep{madry2018towards, rahimian2019distributionally, jiang2023optimality, sinha2017certifying}.
One popular setting is the composite minimax optimization, i.e., the objective consists of a smooth term and a simple nonsmooth convex regularization. 
In this case, we can extend the algorithms for the solving smooth minimax problems by additionally introducing the proximal mapping, which leads to the similar analysis and convergence rates to the smooth counterparts \citep{lu2020hybrid,zhao2024primal,barazandeh2020solving,huang2021efficient,zhang2024robust}.
However, the algorithms with proximal mapping are impractical when the nonsmooth part in the objective is complicated.
Another line of research considers 
the nonsmooth nonconvex-concave minimax optimization with Lipschitz continuity.
Specifically, \citet{lin2025two,boct2023alternating} establish the convergence guarantees of GDA methods for solving nonsmooth nonconvex-concave minimax problems with Lipschitz continuous objective, 
while their analyses require that the Moreau envelope \citep{moreau1965proximite} of the primal function be well-defined, which depends on an additional assumption of weak convexity.
Unfortunately, the weak convexity does not always hold,
such as in the applications of statistical learning \citep{fan2001variable,zhang2006gene,zhang2010nearly,mazumder2011sparsenet,zhang2010analysis} and training neural networks with ReLU activations \citep{nair2010rectified,zhang2020complexity}.
Very recently, \citet{masiha2026zeroth} established results without weakly convex assumption, while their analysis does not cover the fully nonsmooth setting since it still requires the objective $f(\vx,\vy)$ be smooth in $\vy$.
This leads to a natural question: \emph{Can we establish the convergence guarantees for the general nonsmooth nonconvex-concave minimax optimization?}

In this paper, we give an affirmative answer to the above question.
Specifically, we introduce the notion of $(\eta_x,\eta_y,\delta,\epsilon)$-Goldstein saddle stationary point (GSSP) to describe the convergence of solving constrained nonsmooth minimax problem (\ref{prob:constrained main}), 
which goes beyond the existing criteria of the approximate (generalized) Goldstein stationary points for nonsmooth minimization problem \citep{liu2024zeroth,zhang2020complexity,lin2022gradient}.
We then propose single-loop zeroth-order methods for achieving $(\eta_x,\eta_y,\delta,\epsilon)$-GSSPs of the objective~$f(\vx,\vy)$ with the non-asymptotic convergence guarantees. 
In addition, we investigate the convergence with respect to the primal function $\Phi(\vx) \triangleq \max_{\vy \in \mathcal{Y}} f(\vx, \vy)$, which is guaranteed to be Lipschitz continuous when the objective function $f(\vx,\vy)$ is Lipschitz continuous and concave in $\vy$ \citep{lin2020gradient}.
Therefore, the $(\eta,\delta,\epsilon)$-generalized Goldstein stationary point (GGSP) for minimizing $\Phi(\vx)$ over the feasible set $\fX$ is well-defined \citep{liu2024zeroth}.
This motivates us to propose double-loop zeroth-order methods for achieving approximate GGSPs of the primal function $\Phi(\vx)$.
It is worth noting that the previous algorithm designs and theory for nonsmooth nonconvex-concave minimax optimization requires the additional assumption that the objective function $f(\vx,\vy)$ is weakly-convex in $\vx$ to guarantee the Moreau envelope of $\Phi(\vx)$ and the (inexact) proximal point iteration to be well-defined \citep{lin2025two,boct2023alternating}, which  
is restrictive in real applications  \citep{nair2010rectified,zhang2020complexity}.
In contrast, all of our results do not depend on the assumption of weak convexity.
We summarize our results and those from related work in Table \ref{table:comparison}.

\begin{table}[t]
\centering
\caption{We compare our results with related work for both
nonconvex-strongly-concave (NC-SC) and general nonconvex-concave (NC-C) cases,
where $L$ is the mean-squared Lipschitz parameter of stochastic components, $L_f$ is Lipschitz parameter of the objective, $\mu$ is the strongly concave parameter, $\sigma^2$ is the variance of gradient noise (only for first-order methods \citep{lin2025two}, $\Delta$ is the initial optimal primal function value gap, $D_y$ is the diameter of $\fY$, and $\rho$ is the weakly convex parameter. 
The column ``WC'' indicates whether the algorithm requires the weakly convex assumption.
}\label{table:comparison}  \vskip-0.1cm
\resizebox{\linewidth}{!}{
\begin{threeparttable}
\begin{tabular}{ccccccc}
\hline
Setting & Oracle & Method$^\S$ & Complexity$^\sharp$ & Criteria$^\ddag$ &  WC  \\
\hline\addlinespace
\multirow{3}{*}[-7.0ex]{NC-SC}
    & 1st & \makecell{TTSGDA \\ \citep{lin2025two}} & $\widetilde{\fO}\left(\dfrac{\rho^3(L_f^2+\sigma^2)^2\Delta}{\mu\epsilon^6}\right)$ 
    & $\epsilon$-SP of $\Phi_{1/(2\rho)}$
    & \Checkmark \!\!   \\ \addlinespace
    & 0th & \makecell{PGFDA \\ (Theorem \ref{zero-thm})}
&$\fO\left(\dfrac{(d_x+d_y)^{3}L^5(\Delta+\delta L)}{\mu^3 \delta^4 \epsilon^3}\right)$ 
    & $(\eta_x,\eta_y,\delta,\epsilon)$-GSSP of $f$ \!\!\!\!
    & \XSolidBrush \!\!
     \\ \addlinespace
    & 0th &\makecell{NL-PGFDA  \\(Theorem \ref{Theorem of Phi})}&$\fO\left(\dfrac{d_x^{7/2} d_y L^7(\Delta+\delta L)}{\mu^2\delta^3 \epsilon^6}\right)$ 
    & $(\eta,\delta,\epsilon)$-GGSP of $\Phi$
    & \XSolidBrush \!\!
      \\ \addlinespace
    \hline \addlinespace  
    \multirow{3}{*}[-7.3ex]{NC-C}
    & 1st & \makecell{TTSGDA  \\ \citep{lin2025two}}&$\fO\left(\dfrac{\rho^3(L_f^2+\sigma^2)^2D_y^2\Delta }{\epsilon^8}\right)$ 
    & $\epsilon$-SP of $\Phi_{1/(2\rho)}$
    & \Checkmark \!\!
      \\ \addlinespace
    & 0th & \makecell{PGFDA \\(Corollary \ref{thm: complexity of f: concave})}&
    $\fO\left(\dfrac{(L+\epsilon)^5 (d_x+d_y)^{3}D_y^3 (\Delta+\delta L+\delta\epsilon)}{ \delta^4 \epsilon^6}\right)$
     & $(\eta_x,\eta_y,\delta,\epsilon)$-GSSP of $f$  \!\!\!\!
    &  \XSolidBrush \!\!  \\ \addlinespace
    & 0th & \makecell{NL-PGFDA\\ (Corollary \ref{thm: complexity of phi: concave})} & $\fO\left(\dfrac{ d_x^{11/2} d_y \big(L+\delta\epsilon/(d_x D_y)\big)^7D_y^4 \big(\Delta+\delta L+\delta^2\epsilon/(d_x D_y)\big)}{\delta^5 \epsilon^8}\right)$ 
    &  $(\eta,\delta,\epsilon)$-GGSP of $\Phi$ 
    & \XSolidBrush \!\!
      \\ \addlinespace
\hline
\end{tabular}
\vskip-0.05cm
\begin{itemize}[leftmargin=0.6cm,topsep=-0.2cm,itemsep=-0.2cm,partopsep=-0.2cm]
\item[$\S$] \large The results of \citet{boct2023alternating,masiha2026zeroth} require the objective $f(\vx,\vy)$ to be smooth in~$\vy$, which is different with our general nonsmooth setting. Hence, we do not incorporate them into comparison. \\[-0.3cm]
\item[$\sharp$] \large See discussion on the relationship between the smoothness parameters $L$ and $L_f$ in Section~\ref{sec:disscussion}.   \\[-0.3cm]
\item[$\ddag$] \large Note that \citet{lin2025two} established their convergence guarantees with respect to $\epsilon$-stationary point (SP) of the Moreau envelope $\Phi_{1/(2\rho)}(\vx)\triangleq\min_{\vz\in\BR^d} \Phi(\vz)+\rho\norm{\vz-\vx}^2$, which requires the objective $f(\vx,\vy)$ be $\rho$-weakly convex in~$\vx$ for some~$\rho>0$.
\end{itemize}  
\end{threeparttable}}  \vskip-0.3cm
\end{table}

\textbf{Paper Organization} In Section \ref{sec: Preliminary}, we formalize our problem setting and review existing convergence criteria for nonsmooth optimization.
In Section \ref{Sec: The stationary points}, we introduce the notion of stationarity for nonsmooth minimax optimization. 
In Section \ref{sec:smoothing}, we introduce the randomized smoothing for nonconvex-concave minimax problems. 
In Sections \ref{sec: f} and~\ref{sec: Phi}, we propose stochastic zeroth-order methods for solving constrained nonsmooth nonconvex-concave minimax problems and show their non-asymptotic convergence guarantees with respect to both the objective function and the primal function.
In Section~\ref{sec:disscussion}, we compare our results with related work \citep{lin2025two}.  
In Section~\ref{Sec: Experiments}, we conduct numerical experiments to evaluate the effectiveness of proposed methods.
Finally, we conclude our work and remark the future directions in Section \ref{sec:conclusion}.

\section{Preliminaries}\label{sec: Preliminary}

In this section, we first formalize assumptions for our stochastic nonsmooth nonconvex-concave minimax optimization problem, 
then review the notions of approximate (generalized) Goldstein stationary points used for the convergence analysis of nonsmooth optimization.

\subsection{Assumptions}

For constrained stochastic minimax problem (\ref{prob:constrained main}), we introduce the augmented sets as follows
\begin{align}\label{dfn:XY-delta}
\fX_\delta = \{\vx+\vu: \vx\in\fX,\vu\in\BR^{d_x},\|\vu\|\leq\delta\}
~~\text{and}~~
\fY_\delta = \{\vy+\vv: \vy\in\fY,\vv\in\BR^{d_y},\|\vv\|\leq\delta\}
\end{align}
for some $\delta>0$.
Our assumptions for the stochastic component $F(\vx,\vy;\vxi)$ and the objective function $f(\vx,\vy)$ are with respect to the augmented sets $\fX_\delta$ and $\fY_\delta$ to align with the notation of $\delta$-Goldstein subdifferential (see Definition \ref{def:delta-goldstein}) for nonsmooth analysis.

We suppose the stochastic component $F(\vx,\vy;\vxi)$ in our problem (\ref{prob:constrained main}) is mean-squared Lipschitz continuous, which follows the standard setting in existing studies on stochastic zeroth-order nonsmooth nonconvex optimization \citep{lin2022gradient,chen2023faster,kornowski2024algorithm}.

\begin{assumption}\label{asm:Lipschitz}
We suppose that the stochastic component $F(\vx,\vy;\vxi)$ is $L(\vxi)$-Lipschitz continuous in $(\vx,\vy)$ on the augmented  set $\fX_\delta\times\fY_\delta$ for any $\vxi$, i.e. it holds 
\begin{align*}
|F(\vx_1,\vy_1;\vxi) - F(\vx_2,\vy_2;\vxi)| \le L(\vxi) \left(\norm{\vx_1 - \vx_2}^2 + \norm{\vy_1 - \vy_2}^2\right)^{1/2} 
\end{align*}
for all $\vx_1, \vx_2\in\fX_\delta$ and $\vy_1, \vy_2\in\fY_\delta$.
Moreover, we suppose $L(\vxi)$ has bounded second-order moment, i.e. there exists some constant $L>0$ such that 
\begin{align*}
\BE_{\vxi} \big[L(\vxi)^2\big] \le L^2.    
\end{align*}
\end{assumption}
Note that Assumption \ref{asm:Lipschitz} indicates the stochastic component $F(\vx,\vy;\vxi)$ is mean-squared Lipschitz continuous with parameter $L$, i.e.,
for all $\vx_1, \vx_2\in\fX_\delta$ and $\vy_1, \vy_2\in\fY_\delta$, it holds 
\begin{align*}
\mathbb{E}\left[|F(\vx_1,\vy_1;\vxi) - F(\vx_2,\vy_2;\vxi)|^2\right] \leq L^2\left(\norm{\vx_1 - \vx_2}^2 + \norm{\vy_1 - \vy_2}^2\right).    
\end{align*}
We impose the following assumption on the feasible sets $\fX\subseteq\BR^{d_x}$ and $\fY\subseteq\BR^{d_y}$.

\begin{assumption}\label{asm:convex-closed}
    We suppose the feasible sets $\fX\subseteq\BR^{d_x}$ and $\fY\subseteq\BR^{d_y}$ are convex and closed.
    Furthermore, we suppose the feasible set $\fY$ is bounded with a diameter $D_y>0$.
\end{assumption}
We suppose the objective $f(\vx,\vy)$ is concave in $\vy$.
\begin{assumption}\label{asm:concave}
    We suppose the objective function $f(\vx,\vy)$ is concave in $\vy$ on the augmented set $\fY_\delta$, i.e., 
    for all $\vx \in \fX_\delta$, $\vy_1,\vy_2 \in \fY_\delta$, and  $\lambda\in[0,1]$, it holds 
    \begin{align*}        
        f(\vx, \lambda \vy_1 + (1-\lambda)\vy_2)
        \;\ge\;
        \lambda f(\vx, \vy_1) + (1-\lambda) f(\vx, \vy_2).
    \end{align*}
\end{assumption}
We also consider the stronger assumption that the objective $f(\vx,\vy)$ is strongly concave in $\vy$.
\begin{assumption}\label{asm:strongly-concave}
    We suppose the objective function $f(\vx,\vy)$ is $\mu$-strongly concave in $\vy$ on the augmented set $\fY_\delta$, i.e., there exists some $\mu>0$ such that for all $\vx \in \fX_\delta$, $\vy_1,\vy_2 \in \fY_\delta$, and~$\lambda\in[0,1]$,  it holds 
    \begin{align*}
        f(\vx,\lambda \vy_1 + (1-\lambda)\vy_2)
        \ge
        \lambda f(\vx,\vy_1) + (1-\lambda)f(\vx,\vy_2)
        +
        \frac{\mu\lambda(1-\lambda)}{2}\|\vy_1 - \vy_2\|^2.
    \end{align*}
\end{assumption}
The minimax problem (\ref{prob:constrained main}) can be viewed as minimizing the primal function
\begin{align}\label{def: Phi}
 \Phi(\vx) \triangleq \max_{\vy \in \mathcal{Y}} f(\vx, \vy).   
\end{align}
over the domain $\fX\subseteq\BR^{d_x}$. We suppose the function $\Phi$ is lower bounded on $\fX$.

\begin{assumption}\label{asm:lower-bound}
    We suppose the primal function $\Phi^*\triangleq\inf_{\vx\in\mathcal{X}}\Phi(\vx)$ is lower bounded, i.e., there exists $\Phi^*\triangleq\inf_{\vx\in\mathcal{X}}\Phi(\vx)>-\infty$.
\end{assumption}

\subsection{Generalized Goldstein Stationary Points}

We first review the background in the analysis of nonsmooth minimization problem. 
Recall that Rademacher’s theorem says Lipschitz continuous functions are differentiable almost everywhere. Hence, we define the Clarke subdifferential as follows \citep{clarke1990optimization}.

\begin{definition}[{\citet{clarke1990optimization}}]
The Clarke subdifferential of a Lipschitz continuous \mbox{function} $h:\BR^d\to\BR$ at a point $\vz\in\BR^d$ is defined as
\begin{align*}    
\partial h(\vz) = \mathrm{conv}\left(\left\{\, \vg : \vg = \lim_{k \to \infty} \nabla h(\vz_k),\ \vz_k \to \vz \,\right\}\right),
\end{align*}
where $\mathrm{conv}(\cdot)$ is the convex hull.
\end{definition}
It is known that finding an $\epsilon$-Clarke stationary point $\vz$ such that $\min\{\norm{\vg}:\vg\in\partial h(\vz)\}\leq\epsilon$ is computational intractable \citep{zhang2020complexity}, 
so that we further consider the $\delta$-Goldstein subdifferential which is defined as follows \citep{goldstein1977optimization}.

\begin{definition}[{\citet{goldstein1977optimization}}]\label{def:delta-goldstein}
For given $\delta>0$ and a Lipschitz continuous function $h:\BR^d\to\BR$, the $\delta$-Goldstein subdifferential at a point $\vz\in\BR^d$ is defined as 
\begin{align*}
\partial_{\delta} h(\vz) := \operatorname{conv} \left(\cup_{\hat\vz \in \BB^d(\vz,\delta)} \partial h(\hat\vz) \right),
\end{align*}
where $\BB^d(\vz,\delta)=\{\vw\in\BR^d:\norm{\vw-\vz}\leq\delta\}$ is the $d$-dimensional Euclidean ball centered at point $\vz\in\BR^d$ with radius $\delta$.
\end{definition}
Consequently, we introduce the notion of the $(\delta,\epsilon)$-Goldstein stationary point (GSP), which is widely used as the convergence criteria in the analysis of unconstrained nonconvex nonsmooth minimization problems \citep{zhang2020complexity}.
\begin{definition}[{\citet{zhang2020complexity}}]\label{def:GSP}
For given a Lipschitz function $h:\BR^d\to\BR$, we say the point $\vz\in\BR^d$ is a $(\delta,\epsilon)$-Goldstein stationary point of $h$ if it holds
\begin{align*}
    \min\big\{\|\mathbf{g}\| : \mathbf{g} \in \partial_\delta h(\mathbf{z})  \big\} \leq \epsilon.
\end{align*}
\end{definition}

For the constrained nonconvex nonsmooth problem, we can generalize the standard $(\delta,\epsilon)$-Goldstein stationary point by replacing the subgradient in its definition with the subgradient mapping, 
resulting in the criteria of $(\eta,\delta,\epsilon)$-generalized Goldstein stationary point (GGSP)  \citep{liu2024zeroth}.

\begin{definition}[{\citet{liu2024zeroth}}]\label{def:GGSP}
Consider minimization problem 
$\min_{\vz\in\fZ} h(\vz)$ with the Lipschitz continuous objective $h:\BR^{d_z}\to\BR$ and the convex and compact feasible set $\fZ\subseteq\BR^{d_z}$. 
We say the point $\vz\in\fZ$ is a~$(\eta,\delta,\epsilon)$-generalized Goldstein stationary point (GGSP) of the above minimization problem if it holds
\begin{align*}
\min_{\vg \in \partial_\delta h(\vz)} \| G(\vz, \vg, \eta) \| \leq \epsilon,   
\end{align*}
where 
$G(\vz,\vg,\eta) = (\vz-\Pi_\fZ\left(\vz - \eta \vg\right))/{\eta}$
is the subgradient mapping at point $\vz\in\fZ$ with respect to $\vg\in\partial_\delta h(\vz)$ and $\eta>0$.
Here, we use the notation $\Pi_\mathcal{Z}(\vw) :=\arg\min_{\hat\vw \in \mathcal{Z}} \|\hat\vw - \vw\|^2$ to present the projection of the point~$\vw\in\BR^{d_z}$ onto the feasible set $\mathcal{Z}\subseteq\BR^d$.
\end{definition}

\begin{remark}\label{remark:subgradient-ggsp}    
It is worth noting that solving minimization problem 
$\min_{\vz\in\fZ} h(\vz)$ by projected subgradient descent  with stepsize $\eta>0$ can be written as \citep{nesterov2013introductory} 
\begin{align*}
    \vz_+ = \vz - \eta G(\vz,\vg,\eta).
\end{align*}
Hence, the term $G(\vz,\vg,\eta)$ in the definition of $(\eta,\delta,\epsilon)$-GGSP for constrained optimization corresponds to descent direction of projected subgradient descent, which aligns with the fact that term $\mathbf{g} \in \partial_\delta h(\mathbf{z})$ in the notion of $(\delta,\epsilon)$-GSP (Definition \ref{def:GSSP}) for unconstrained problem corresponds to the descent direction of subgradient descent.
\end{remark}
\begin{remark}
In fact, the Lipschitz continuity for the objective $h$ in Definition \ref{def:GGSP} is only required on the set $\fZ_\delta = \{\vz+\vw: \vz\in\fZ,\vw\in\BR^{d_z},\|\vw\|\leq\delta\}$ to align with the notion of $\delta$-Goldstein subdifferential in Definition \ref{def:delta-goldstein}.
\end{remark}

\section{Stationarity for Nonsmooth Minimax Optimization}\label{Sec: The stationary points}

We first introduce the notation of $(\eta_x,\eta_y,\delta, \epsilon)$-Goldstein saddle stationary point (GSSP) to characterize the convergence with respect to the objective $f(\vx,\vy)$ for our nonsmooth nonconvex-concave minimax problem (\ref{prob:constrained main}).

\begin{definition}\label{def:GSSP}
    We say the point $(\vx,\vy)\in\mathcal{X}\times\mathcal{Y}$ is a $(\eta_x,\eta_y,\delta, \epsilon)$-Goldstein saddle stationary point (GSSP) of the minimax optimization problem (\ref{prob:constrained main}) if it holds
    \begin{align*}
        \min_{(\vg_x, \vg_y) \in \partial_\delta f(\vx,\vy)} \| G_x(\vx,\vy, \vg_{x}, \eta_x) \| \leq \epsilon
        \quad\text{and}~~
        \quad\min_{(\vg_x, \vg_y) \in \partial_\delta f(\vx,\vy)} \| G_y(\vx,\vy, \vg_{y}, \eta_y) \| \leq \epsilon,
    \end{align*}
    where $\partial_\delta f(\vx,\vy)$ is the $\delta$-Goldstein subdifferential that follows Definition \ref{def:delta-goldstein} and
    \begin{align}\label{eq:gradient-mapping-xy}
        G_x(\vx, \vy,\vg_{x},\eta_x) = \frac{\vx-\Pi_\mathcal{X}\left(\vx - \eta_x \vg_{x}\right) }{\eta_x}
        \quad\text{and}\quad
        G_y(\vx, \vy,\vg_{y},\eta_y) = \frac{\Pi_\mathcal{Y}\left(\vy + \eta_y \vg_{y}\right) - \vy}{\eta_y}
    \end{align}
    are subgradient mapping at point $(\vx,\vy)\in\fX\times\fY$ with respect to $(\vg_x, \vg_y) \in \partial_\delta f(\vx,\vy)$ and stepsizes~$\eta_x,\eta_y>0$ for variables $\vx$ and  $\vy$, respectively.
\end{definition}
The definition of our approximate GSSP is motivated by the subgradient descent ascent for solving constrained nonsmooth convex-concave minimax problem 
\begin{align*}
    \min_{\vx\in\mathcal{X}}\max_{\vy\in\mathcal{Y}}f(\vx,\vy),
\end{align*}
which performs the iteration
\begin{align}\label{eq:subgradient-iteration}
    \begin{cases}
        \vx_+ = \Pi_\mathcal{X}\left(\vx - \eta_x\vg_{x}\right), \\[0.1cm]
        \vy_+ = \Pi_\mathcal{Y}\left(\vy + \eta_y \vg_{y}\right),
    \end{cases}
\end{align}
where $\eta_x,\eta_y>0$ are stepsizes and $(\vg_x,\vg_y)\in\partial f(\vx,\vy)$ is a subgradient of $f$ at point $(\vx,\vy)$.
Note that the iteration scheme (\ref{eq:subgradient-iteration}) can also be written as
\begin{align}\label{eq:subgradient-iteration2}
    \begin{cases}
        \vx_+ = \vx - \eta_x G_x(\vx, \vy,\vg_{x},\eta_x), \\[0.1cm]
        \vy_+ = \vy + \eta_y G_y(\vx, \vy,\vg_{y},\eta_y),
    \end{cases}
\end{align}
where $G_x(\vx, \vy,\vg_{x},\eta_x)$ and $G_y(\vx, \vy,\vg_{y},\eta_y)$ follow equation (\ref{eq:gradient-mapping-xy}).
Hence, terms $G_x(\vx, \vy,\vg_{x},\eta_x)$ and $G_y(\vx, \vy,\vg_{y},\eta_y)$ in the notion of $(\eta_x,\eta_y,\delta, \epsilon)$-GSSP (Definition \ref{def:GSSP}) corresponds to the descent and ascent directions of iteration scheme (\ref{eq:subgradient-iteration2}), 
which is similar to the intuition of GGSP for constrained nonsmooth minimization problem described (Remark \ref{remark:subgradient-ggsp}).

\begin{remark}
It is worth noting that the original notion of GGSP \citep{liu2024zeroth} is not suitable to nonsmooth optimization. 
Note that directly applying Definition \ref{def:GGSP} with $h=f$, $\vz=(\vx,\vy)$, $\fZ=\fX\times\fY$, and $\vg=(\vg_x,\vg_y)\in\partial_\delta f(\vx,\vy)$ to minimax problem (\ref{prob:constrained main}) results
\begin{align}\label{eq:wrong-extension}
    G(\vz,\vg,\eta) = \begin{bmatrix}
    \dfrac{\vx-\Pi_\mathcal{X}\left(\vx - \eta \vg_{x}\right) }{\eta} \\[0.35cm]
    \dfrac{\vy-\Pi_\mathcal{Y}\left(\vy - \eta \vg_{y}\right) }{\eta}
    \end{bmatrix}.
\end{align} The component $(\vy-\Pi_\mathcal{Y}\left(\vy - \eta \vg_{y}\right))/{\eta}$ in equation (\ref{eq:wrong-extension}) corresponds to the step of subgradient descent on the function $f(\vx,\cdot)$ for fixed $\vx$, which validates the intuition that we desire to maximize the function $f(\vx,\vy)$ over the variable $\vy$ in minimax problem.
Additionally, the roles of variables~$\vx$ and~$\vy$ in nonconvex concave minimax problem (\ref{prob:constrained main}) are asymmetric, which indicates it is more reasonable to introduce different $\eta_x>0$ and $\eta_y>0$ for these two variables in the notion of~$(\eta_x,\eta_y,\delta, \epsilon)$-GSSP.
\end{remark}

Besides the objective, we are also interested in the stationarity of the primal function
\begin{align*}
\Phi(\vx) \triangleq \max_{\vy \in \mathcal{Y}} f(\vx, \vy).
\end{align*}
Note that the primal function $\Phi(\vx)$ for the Lipschitz nonconvex-concave objective $f(\vx,\vy)$ is Lipschitz continuous \citep{lin2025two}, which indicates its $(\eta_x,\delta,\epsilon)$-GGSP is well-defined. 
\begin{proposition}[{\citet[Lemma 4.2]{lin2025two}}]\label{Lipshitz of Phi}
    Under Assumptions \ref{asm:Lipschitz} and \ref{asm:concave}, the primal function $\Phi(\vx) \triangleq \max_{\vy \in \mathcal{Y}} f(\vx, \vy)$ is $L$-Lipschitz continuous on domain $\fX_\delta$.
\end{proposition}
Therefore, we can also study the convergence of nonsmooth nonconvex-concave minimax problem \eqref{prob:constrained main} with respect to the $(\eta_x,\delta,\epsilon)$-GGSP of the Lipschitz continuous function $\Phi$ by following Definition \ref{def:GGSP} with~$h=\Phi$ and $\fZ=\fX$.

\section{Randomized Smoothing for Nonconvex-Concave Minimax Problems}\label{sec:smoothing}

We consider the smoothed surrogate for the objective $f(\vx,\vy)$ as follows
\begin{align}\label{eq:f_delta}
    f_\delta(\vx,\vy) = \BE_{(\vu,\vv)\sim {\mathrm{Unif}}(\BB ^{d}(\vzero,1))} [f(\vx+\delta\vu, \vy+\delta\vv)],
\end{align}
where $\delta>0$, $d=d_x+d_y$, $(\vu,\vv)$ is uniformly distributed on the $d$-dimensional unit ball centered at the origin, denoted as $\BB^{d}(\vzero,1)$ \citep{yousefian2012stochastic}.
The Lipschitz continuity of $f(\vx,\vy)$ indicates the function $f_\delta(\vx,\vy)$ has the following properties \citep{lin2022gradient}.

\begin{proposition}[\citet{lin2022gradient}]\label{prop:smoothing}
Under Assumptions \ref{asm:Lipschitz}, for all $\vx\in\fX$ and $\vy\in\fY$, the objective $f(\vx,\vy)$ and its smooth surrogate $f_\delta(\vx,\vy)$ defined in equation \eqref{eq:f_delta} hold
\begin{enumerate}[leftmargin=0.8cm,topsep=1.5pt,itemsep=2.5pt,partopsep=2.5pt, parsep=2.5pt]
\item[(a)] $\vert f_{\delta}(\vx,\vy) - f(\vx,\vy) \vert \le \delta L$;
\item[(b)] $f_{\delta}(\vx,\vy)$ is $L$-Lipschitz continuous;
\item[(c)] $f_{\delta}(\vx,\vy)$ is $M_{f_\delta}$-smooth with $M_{f_\delta}=cd^{1/2}L\delta^{-1}$ for some constant $c>0$; 
\item[(d)]  $\nabla f_{\delta}(\vx,\vy) \in \partial_{\delta} f(\vx,\vy)$, where $\partial_{\delta} f(\vx,\vy)$ is the $\delta$-Goldstein subdifferential.
\end{enumerate}
\end{proposition}

In addition, the smooth surrogate $f_\delta(\vx,\vy)$ preserves  the concavity in $\vy$.
\begin{proposition}\label{prop: concave of f_delta}
Under Assumptions \ref{asm:Lipschitz} and \ref{asm:concave}, the function $f_\delta(\vx,\vy)$ is concave in~$\vy$ on the domain $\fY$.
If we further suppose Assumption \ref{asm:strongly-concave} holds, the function $f_\delta(\vx,\vy)$ is $\mu$-strongly concave in $\vy$ on the domain $\fY$.
\end{proposition}

Propositions \ref{prop:smoothing} and \ref{prop: concave of f_delta} motivate us to introduce the following smooth nonconvex-concave minimax problem
\begin{align}\label{prob:f-delta}
\min_{\vx\in\mathcal{X}}\max_{\vy\in\mathcal{Y}} f_\delta(\vx,\vy).
\end{align}
We further introduce the primal function for $f_\delta(\vx,\vy)$, which is defined as
\begin{align}\label{def:Psi}
 \Psi_\delta(\vx) \triangleq \max_{\vy \in \mathcal{Y}} f_\delta(\vx, \vy). 
\end{align}
The function $\Psi_\delta$ is guaranteed to be smooth if the objective $f(\vx,\vy)$ is Lipschitz continuous and strongly concave in $\vy$ on feasible set $\fY$. 

\begin{proposition}[\citet{lin2020gradient}]\label{prop:smooth-Psi-delta}
Under Assumptions \ref{asm:Lipschitz} and \ref{asm:strongly-concave}, the function $\Psi_\delta(\vx)$ defined in equation \eqref{def:Psi} is $M_{\Psi_\delta}$-smooth on the feasible set $\fX$ with $M_{\Psi_\delta}=(1+M_{f_\delta}/\mu)M_{f_\delta}$,
where $M_{f_\delta}>0$ follows Proposition \ref{prop:smoothing}(c).
\end{proposition}

Following the setting of Proposition \ref{prop:smooth-Psi-delta},
we define the optimal solution mappings 
\begin{align}\label{eq:y-star}
    \vy^*(\vx) \triangleq \argmax_{\vy\in\fY} f(\vx,\vy)
    \qquad\text{and}\qquad
    \vy_\delta^*(\vx) \triangleq \argmax_{\vy\in\fY} f_\delta(\vx,\vy).
\end{align}
Note that the strong concavity in $\vy$ guarantees both $\vy^*(\vx)$ and $\vy_\delta^*(\vx)$ are uniquely defined for given $\vx$.
We also bound the difference between  $\vy^*(\vx)$ and $\vy_\delta^*(\vx)$ as follows.
\begin{lemma}\label{distance between y_star and y_star_delta}
    Under Assumptions \ref{asm:Lipschitz} and \ref{asm:strongly-concave}, for all $\vx\in\fX$, it holds
    \begin{align*}
        \|\vy^*(\vx)-\vy_\delta^*(\vx)\|\leq 2\sqrt{\frac{L\delta}{\mu}}.    
    \end{align*}
\end{lemma}
Based on Lemma \ref{distance between y_star and y_star_delta},
we further bound the difference between $\Phi$ and $\Psi_\delta$ as follows.
\begin{lemma}\label{distance between Phi Phi_delta}
Under Assumptions \ref{asm:Lipschitz} and \ref{asm:strongly-concave}, 
for all $\vx\in\fX$, it holds 
$|\Phi(\vx)-\Psi_\delta(\vx)|\leq \delta L$.
\end{lemma}

\section{Finding Approximate GSSPs for the Objective Function}\label{sec: f}

In this section, we first propose gradient-free algorithm for finding $(\eta_x,\eta_y,\delta, \epsilon)$-Goldstein saddle stationary points (GSSPs) with respect to the objective function $f(\vx,\vy)$, then provide the complexity analysis for our algorithm. 

\subsection{Projected Gradient-Free Descent Ascent}

In the view of surrogate problem \eqref{prob:f-delta}, we introduce zeroth-order gradient estimators as follows \citep{yousefian2012stochastic}  
\begin{align}\label{eq:gx-gy-w-xi}
\begin{split}    
&\!\!\vg_x(\vx,\vy;\vw_x,\vw_y,\vxi)     
\triangleq \frac{d(F(\vx + \delta \vw_x,\vy+\delta \vw_y;\vxi) - F(\vx - \delta \vw_x,\vy-\delta \vw_y;\vxi))}{2\delta}
\!\cdot\! \vw_x \in\BR^{d_x}, \\
&\!\!\vg_y(\vx,\vy;\vw_x,\vw_y,\vxi) 
\triangleq \frac{d(F(\vx + \delta \vw_x,\vy+\delta \vw_y;\vxi) - F(\vx - \delta \vw_x,\vy-\delta \vw_y;\vxi))}{2\delta} \!\cdot\! \vw_y \in\BR^{d_y},
\end{split}
\end{align}
where $[\vw_x;\vw_y]\in\BR^{d}$ is sampled from the uniform distribution on the unit sphere, i.e., $\BS^{d-1}=\{\vw\in\BR^d:\norm{\vw}=1\}$, where $d=d_x+d_y$. 
Note that the above random vectors $\vg_x(\vx,\vy;\vw_x,\vw_y,\vxi)$ and $\vg_y(\vx,\vy;\vw_x,\vw_y,\vxi)$ correspond to an unbiased gradient estimator for the smooth surrogate $f_\delta(\vx,\vy)$ with bounded variance \citep{lin2022gradient}.

\begin{proposition}[{\citet[Lemma D.1]{lin2022gradient}}]\label{prop:g-unbiased-variance}
Under Assumption~\ref{asm:Lipschitz}, we define
\begin{align}\label{eq:g-w-xi}
\begin{split}    
    \vg(\vx,\vy;\vw_x,\vw_y,\vxi) = \begin{bmatrix}
        \vg_x(\vx,\vy;\vw_x,\vw_y,\vxi)  \\
        \vg_y(\vx,\vy;\vw_x,\vw_y,\vxi) 
    \end{bmatrix}
\end{split}
\end{align}
by following the notation in equation \eqref{eq:gx-gy-w-xi}, 
then it holds
\[
\mathbb{E}
\big[ \vg(\vx,\vy;\vw_x,\vw_y,\vxi) \big]
= \nabla f_\delta(\vx,\vy) 
~~\text{and}~~
\mathbb{E}
\left[\big\|
\vg(\vx,\vy;\vw_x,\vw_y,\xi)
- \nabla f_\delta(\vx,\vy)
\big\|^2\right]
\le 16\sqrt{2\pi}d L^2.
\]
\end{proposition}
According to Proposition \ref{prop:g-unbiased-variance}, we further introduce mini-batch sampling to reduce the variance in zeroth-order gradient estimators \citep{chen2023faster}.
\begin{corollary}[{\citet[Corollary 1]{chen2023faster}}]\label{Variance of the estimator}
Following the settings and notations of Proposition \ref{prop:g-unbiased-variance}, we construct the gradient estimator 
\begin{align}\label{eq:g-unbiased-variance}
\vg(\vx,\vy; \fS) = \frac{1}{b} \sum_{i=1}^b \vg(\vx,\vy;\vw_{x,i},\vw_{y,i},\vxi_i) 
\quad\text{with}\quad
\fS = \{(\vw_{x,i},\vw_{y,i}\,;\vxi_i)\}_{i=1}^b,
\end{align}
where $(\vw_{x,1},\vw_{y,1}),\dots,(\vw_{x,b},\vw_{y,b})$ and $\vxi_1,\dots,\vxi_b$ are independently sampled from the uniform distribution on the unit sphere $\BS^{d-1}$ and $\fD$, respectively,
then it holds
\begin{align*}
\mathbb{E}_\fS \left[\left\| \vg(\vx,\vy; \fS) - \nabla f_\delta(\vx,\vy) \right\|^2\right] \leq \frac{16 \sqrt{2\pi} d L^2}{b}.    
\end{align*}
\end{corollary}

\begin{algorithm}[t]
\caption{Projected Gradient-Free Descent Ascent (PGFDA)}
\label{alg:f}
\begin{algorithmic}[1]    
\STATE \textbf{Input:} initial point $(\vx_0, \vy_{-1})\in\fX\times\fY$, step sizes $\eta_x, \eta_y,\tilde\eta_y>0$, iterations number $T, K_{\rm in}, K_{\rm out}\in\BN$, smoothing parameter $\delta>0$, probability $p\in(0,1]$, batch sizes $b,\tilde b\in\BN$.  \\[0.02cm]
\STATE $\vy_0=\text{Gradient-Free-Descent}(-F(\vx_0,\,\cdot\,;\,\cdot\,),\vy_{-1},\mu, \delta, K_{\rm in})$ \label{line:y0} \\[0.02cm]
\STATE \textbf{for} $t = 0,1,2,\dots,T-1$ \textbf{do} \\[0.02cm]
\STATE \qquad sample $\zeta_t \sim \mathrm{Bernoulli}(p)$  \\[0.02cm]
\STATE \qquad \textbf{if} {$t=0$ \textbf{or} $\zeta_t = 1$} \textbf{then} \\[0.02cm]
\STATE \qquad\qquad independently sample $(\vw_{x,t,i},\vw_{y,t,i})\sim{\rm Unif}(\BS^{d-1})$ and $\vxi_{t,i}\sim\fD$ for all $i\in[{\tilde b}]$ \\[0.02cm]
\STATE \qquad\qquad $\tilde\fS_t = \{(\vw_{x,t,i},\vw_{y,t,i}\,;\vxi_{t,i})\}_{i=1}^{\tilde b}$ \\[0.02cm]
\STATE \qquad\qquad $\displaystyle{ 
        \begin{bmatrix}
            \vu_t \\ \vv_t
        \end{bmatrix}
        = \begin{bmatrix}        
            \vg_x(\vx_t,\vy_t;\tilde\fS_t) \\ 
            \vg_y(\vx_t,\vy_t;\tilde\fS_t)
        \end{bmatrix}}$ \\[0.02cm]
\STATE \qquad \textbf{else} \\[0.02cm]
\STATE \qquad\qquad independently sample  $(\vw_{x,t,i},\vw_{y,t,i})\sim{\rm Unif}(\BS^{d-1})$ and $\vxi_{t,i}\sim\fD$ for all $i\in[b]$ \\[0.02cm]
\STATE \qquad\qquad $\fS_t = \{(\vw_{x,t,i},\vw_{y,t,i}\,;\vxi_{t,i})\}_{i=1}^{b}$ \\[0.02cm] 
\STATE \qquad\qquad  $\displaystyle{ 
        \begin{bmatrix}
            \vu_t \\ \vv_t
        \end{bmatrix}
        = \begin{bmatrix}
            \vu_{t-1} \\ \vv_{t-1}
        \end{bmatrix}
        + \begin{bmatrix}
            \vg_x(\vx_t,\vy_t;\fS_t) \\ 
            \vg_y(\vx_t,\vy_t;\fS_t)
        \end{bmatrix}        
        - \begin{bmatrix}
            \vg_x(\vx_{t-1},\vy_{t-1};\fS_t) \\ 
            \vg_y(\vx_{t-1},\vy_{t-1};\fS_t)
        \end{bmatrix}        
        }$ \\[0.1cm]
\STATE \qquad \textbf{end if} \\[0.02cm]
\STATE \qquad $\displaystyle{
    \begin{bmatrix} \vx_{t+1} \\ \vy_{t+1} \end{bmatrix}
    = \begin{bmatrix}
            \Pi_\mathcal{X}(\vx_t - \eta_x\vu_t) \\
            \Pi_\mathcal{Y}(\vy_t + \eta_y\vv_t)
    \end{bmatrix}}$ \\[0.02cm]
\STATE \textbf{end for} \\[0.02cm]
\STATE $j\sim{\rm Unif}(\{0,\dots,T-1\})$, \quad $\vx_{\mathrm{out}}=\vx_j$  \label{line:end-for} \\[0.02cm]
\STATE $\vy_{\mathrm{out}}=\text{Gradient-Free-Descent}(-F(\vx_j,\,\cdot\,;\,\cdot\,), \vy_j,\mu,\delta, K_{\rm out})$ \label{line:haty}  \\[0.02cm]
\STATE \textbf{return} $(\vx_{\mathrm{out}},\vy_{\mathrm{out}})$
\end{algorithmic}
\end{algorithm}

\begin{algorithm}[H]
\caption{Projected-Gradient-Free-Descent$({H(\,\cdot\,;\,\cdot\,),\tilde\vy_0, \mu, \nu, K })$}
\label{alg:minimization}
\begin{algorithmic}[1]   
\STATE \textbf{for} $k=0,1,\dots,K-1$
\STATE \qquad independently sample $\vw_{k}\sim{\rm Unif}(\BS^{d_y-1})$ and $\vxi_k\sim\fD$ \\[0.1cm]
\STATE \qquad $\displaystyle{\vv_k = \frac{d_y}{2\nu}(H(\vy_k + \nu \vw_k;\vxi_k) - H(\vy_k - \nu \vw_k;\vxi_k))\vw_k }$ \\[0.05cm] 
\STATE  \qquad $\displaystyle{\eta_k=\frac{2}{\mu(k+1)}}$\\[0.07cm] 
\STATE \qquad  $\vy_{k+1} = \Pi_\fY (\vy_k - \eta_k \vv_k)$ \\[0.05cm] 
\STATE \textbf{end for} \\[0.05cm] 
\STATE \textbf{return} ${\vy_{\mathrm{out}}=\frac{2}{K(K-1)}\sum_{k=0}^{K-1}k\vy_k}$
\end{algorithmic}  
\end{algorithm}

\begin{remark}
The techniques of two-point zeroth-order gradient estimator \citep{yousefian2012stochastic,duchi2012randomized,nesterov2017random,shamir2017optimal} and variance reduction \citep{li2021page,fang2018spider} have been widely used in nonconvex minimax optimization \citep{xu2024derivative,huang2022accelerated,wang2023zeroth,xu2020gradient,luo2020stochastic,
yang2020catalyst,yan2020optimal,zhang2022sapd+,yang2020global,chen2024efficient,chen2022faster}, while all existing work heavily depends on the smoothness for the objective or the component in composite structure. 
In contrast, this work is the first to explore these techniques in the general nonsmooth case.
\end{remark}

Consequently, we construct the variance reduced gradient estimator $(\vu_t,\vv_t)$ as follows
\begin{align*}
\begin{bmatrix}
    \vu_t \\ \vv_t
\end{bmatrix}
= \begin{bmatrix}
    \vu_{t-1} \\ \vv_{t-1}
  \end{bmatrix}
   + \begin{bmatrix}
        \vg_x(\vx_t,\vy_t;\fS_t) \\ 
        \vg_y(\vx_t,\vy_t;\fS_t)
    \end{bmatrix}        
   - \begin{bmatrix}
            \vg_x(\vx_{t-1},\vy_{t-1};\fS_t) \\ 
            \vg_y(\vx_{t-1},\vy_{t-1};\fS_t)
   \end{bmatrix},
\end{align*}
where $\fS_t$ consists of the random sample at the $t$-th iteration.
This results in our Projected Gradient-Free Descent Ascent (PGFDA) method, which is detailed in Algorithm \ref{alg:f}.

\begin{remark}
Besides iterating on directions $\vu_t\in\BR^{d_x}$ and $\vv_t\in\BR^{d_y}$, Lines \ref{line:y0} and \ref{line:haty} of proposed PGFDA (Algorithm \ref{alg:f}) also introduce the subroutine of Projected Gradient-Free Descent (Algorithm \ref{alg:minimization}) to maximize the function value over $\vy\in\BR^{d_y}$ for given $\vx_0\in\BR^{d_x}$ and~$\vx_{\rm out}\in\BR^{d_x}$ , which is necessary to guarantee the norms of subgradient mapping with respect to $\vy$ defined in equation \eqref{eq:gradient-mapping-xy} be sufficient small.        
\end{remark}

\subsection{The Complexity Analysis for PGFDA}\label{sec:convergence-PGFDA}

We first consider finding approximate GSSPs in the nonconvex-strongly-concave setting.
For the smooth surrogate $f_\delta(\vx,\vy)$ and its primal function $\Psi_\delta(\vx)=\max_{\vy\in\fY}f_\delta(\vx,\vy)$ defined in equations \eqref{eq:f_delta} and \eqref{def:Psi}, we 
denote their gradient mappings as
\begin{align}\label{eq:mapping-smooth}
\begin{split}
    & G_x(\vx, \vy,\nabla_x f_\delta(\vx,\vy),\eta_x) = \frac{\vx-\Pi_\mathcal{X}\left(\vx - \eta_x \nabla_x f_\delta(\vx,\vy)\right) }{\eta_x}, \\
    & G_y(\vx, \vy,\nabla_y f_\delta(\vx,\vy),\eta_y) = \frac{\Pi_\mathcal{Y}\left(\vy + \eta_y \nabla_y f_\delta(\vx,\vy)\right) - \vy}{\eta_y}, \\
    & G(\vx, \nabla \Psi_\delta(\vx),\eta_x) = \frac{\vx-\Pi_\mathcal{X}\left(\vx - \eta_x \nabla\Psi_\delta(\vx)\right) }{\eta_x},    
\end{split}    
\end{align}
where $\eta_x,\eta_y>0$.
We connect the above $G_x(\vx, \vy,\nabla_x f_\delta(\vx,\vy),\eta_x)$, 
$G_y(\vx, \vy,\nabla_y f_\delta(\vx,\vy),\eta_y)$, 
and $G(\vx, \nabla \Psi_\delta(\vx),\eta_x)$ as follows.

\begin{lemma}\label{lemma: gradient mapping between Psi_delta and f_delta}
Under Assumptions \ref{asm:Lipschitz} and \ref{asm:strongly-concave}, 
it holds  
\begin{align*}
\norm{G_x(\vx,\vy,\nabla_x f_\delta(\vx,\vy),\eta_x)}
\leq\norm{G(\vx,\nabla \Psi_\delta(\vx),\eta_x)} + \frac{M_{f_\delta}}{\mu}\| G_y(\vx, \vy,\nabla_y f_\delta(\vx,\vy),\tilde\eta_y) \|,
\end{align*}
for all $\vx\in\fX$, $\vy\in\fY$, $\eta_x>0$ and $\tilde\eta_y\in(0,1/M_{f_\delta}]$, 
where 
$M_{f_\delta}>0$ follows Proposition \ref{prop:smoothing}(c).

\end{lemma}

Lemma \ref{lemma: gradient mapping between Psi_delta and f_delta} indicates that the procedure of finding an $(\eta_x,\tilde\eta_y,\delta,\epsilon)$-GSSP of $f(\vx,\vy)$ 
by PGFDA (Algorithm \ref{alg:f}) can be regarded as the following two stage.
\begin{itemize}[topsep=0.025cm,parsep=0.002cm]
    \item For Lines \ref{line:y0}--\ref{line:end-for}, it minimizes the smooth function $\Psi_\delta(\cdot)$ over the feasible set $\fX\subseteq\BR^{d_x}$ to find $\vx_{\rm out}\in\fX$ such that $\norm{G(\vx_{\rm out},\nabla\Psi_\delta(\vx_{\rm out}),\eta_x)} \leq {\epsilon}/{2}$.
    \item For Line \ref{line:haty}, it maximizes the smooth function $f_\delta(\vx_{\rm out},\cdot)$ over the feasible set $\fY\subseteq\BR^{d_y}$ to find $\vy_{\rm out}\in\fY$ such that 
    $\| G_y(\vx_{\rm out}, \vy_{\rm out},\nabla_y f_\delta(\vx_{\rm out},\vy_{\rm out}),\tilde\eta_y) \| \leq \mu\epsilon/(2M_{f_\delta})$ for given~$\vx_{\rm out}\in\fX$.
\end{itemize}
We then construct the Lyapunov function 
\begin{align}\label{Def of H}
    {\mathcal L}_t \triangleq \Psi_{\delta}(\vx_t)  + \alpha \left( \Psi_{\delta}(\vx_{t}) - f_{\delta}(\vx_{t}, \vy_{t}) \right)+\frac{\alpha\eta_y}{p}\norm{\mathbf{g}_t-\nabla f_\delta(\vx_t,\vy_t)}^2,
\end{align}
to characterize the convergence of PGFDA, where $\alpha\in(0,1/4]$ and 
\begin{align*}
    \vg_t=
\begin{bmatrix}
    \vu_t\\
    \vv_t
\end{bmatrix}.
\end{align*}

For the subroutine of Projected-Gradient-Free-Descent (Algorithm
\ref{alg:minimization}), we establish the following lemma for minimizing nonsmooth strongly convex functions.

\begin{lemma}\label{thm:zeroth-order minimization}
We apply Algorithm \ref{alg:minimization} to solve the nonsmooth minimization problem
\begin{align*}
    \min_{\vy\in\fY}h(\vy) \triangleq \BE[H(\vy;\vxi)],
\end{align*}
where $h:\BR^{d_y}\to\BR$ is $\mu$-strongly convex on the augmented domain $\fY_\delta$, the feasible set $\fY\in\BR^{d_y}$ is convex and compact, and the stochastic component $H(\vy;\vxi)$ indexed by random variable~$\vxi$ is mean-squared Lipschitz continuous in $\vy$ on $\fY_\delta$, i.e., for all $\vy_1, \vy_2\in\fY_\delta$, it holds 
\begin{align*}
\mathbb{E}\left[|H(\vy_1;\vxi) - H(\vy_2;\vxi)|^2\right] \leq L^2 \norm{\vy_1 - \vy_2}^2.    
\end{align*}
By taking $K=\lceil(64\sqrt{2\pi}d_yL^2)/(\mu\tilde\epsilon)\rceil$ and
$\eta_k=2/(\mu(k+1))$ for some $\tilde\epsilon>0$,
the output $\vy_{\mathrm{out}}$ of the algorithm  satisfies
\begin{align*}
\mathbb{E}[h_\nu(\vy_{\mathrm{out}})]-\inf _{\vy\in \fY} h_\nu(\vy)\le\tilde\epsilon,
\end{align*}
where $h_\nu(\vy)= \BE_{\vw\sim {\mathrm{Unif}}(\BB^{d}(\vzero,1))} [h(\vy+\nu\vw)]$ is the smoothing surrogate of $h$.
If we further take~$\nu={\tilde\epsilon}/(4L)$, it holds
\begin{align*}
\mathbb{E}[h(\vy_{\mathrm{out}})]-\inf _{\vy\in \fY} h(\vy)\le\tilde\epsilon.
\end{align*}   
\end{lemma}

We also provide the recursion for the Lyapunov function $\fL_t$ as follows.

\begin{lemma}\label{update}
Under Assumptions \ref{asm:Lipschitz}, \ref{asm:convex-closed}, \ref{asm:strongly-concave}, we run Algorithm \ref{alg:f} by taking parameters such that
\[\eta_x=\frac{\alpha\mu^2\delta^3}{768(1+\alpha)c^3d^{3/2}L^3},\quad\eta_y=\frac{\delta}{4c\sqrt{d}L}, 
\quad\text{and}\quad
pb\ge\frac{d}{c^2}.
\]
then it holds
\begin{align} \label{ieq:bound-P}
\mathbb{E}[\mathcal L_{t+1}]\leq \mathbb{E}\left[ \mathcal L_t - \frac{ \eta_x}{8} \Vert G(\vx_t,\nabla\Psi_\delta(\vx_t),\eta_x)  \Vert^2 
+\frac{  16\sqrt{2\pi}\alpha \eta_y dL^2}{\tilde{b}}\right].
\end{align}
\end{lemma}

Combining the above lemmas, we obtain the following theorem that provides a theoretical guarantee for achieving a $(\eta_x,\tilde\eta_y,\delta,\epsilon)$-GSSP in the nonconvex-strongly-concave case.

\begin{theorem}\label{zero-thm}
Under Assumptions \ref{asm:Lipschitz}, \ref{asm:convex-closed} and \ref{asm:strongly-concave}, we run PGFDA (Algorithm \ref{alg:f}) with

\[
\eta_x=\Theta\left(\frac{\mu^2\delta^3}{d^{3/2}L^3}\right),
\quad\eta_y=\Theta\left(\frac{\delta}{d^{1/2}L}\right),
\quad\tilde\eta_y=\Theta\left(\frac{\delta}{d^{1/2}L}\right),
\]
\[
 b=\Theta\left(\frac{d^{{3}/{2}}L^2}{\mu\delta\epsilon}\right),\quad  \tilde{b}=\Theta\left(\frac{d^2L^4}{\mu^2\delta^2\epsilon^2}\right),
 \quad p=\Theta\left(\frac{\delta\mu\epsilon}{d^{1/2}L^2}\right),
\]
\[ K_{\mathrm{in}}=\Theta\left(\frac{ d_y L^2}{\mu }\right),\quad
K_{\mathrm{out}}
=
\Theta\left(\frac{d^{3/2}d_y L^5}{\mu^3\delta^3  \epsilon^2}\right),\quad
T=\Theta\left(\frac{d^{3/2}L^3 (\Delta+\delta L+1)}{ \mu^2\delta^3 \epsilon^2} \right),
\]
then it outputs a $(\eta_x,\tilde\eta_y,\delta,\epsilon)$-GSSP 
and the total stochastic zeroth-order oracle complexity is
 \[
 \fO\left(\frac{d^{3}L^5(
 \Delta+\delta L)}{\mu^3 \delta^4 \epsilon^3}\right)
 \]
 in expectation.
\end{theorem}

We then consider finding approximate GSSPs in the general nonconvex-concave setting.
In particular, we introduce the regularized minimax problem
\begin{align}\label{prob:regularized}
    \min_{\vx\in\fX}\max_{\vy\in\fY}\tilde{f}(\vx,\vy) \triangleq f(\vx,\vy)-\frac{\epsilon}{2D_y}\|\vy-\vy_0\|^2,
\end{align}
where $f(\vx,\vy)\triangleq\mathbb{E}[F(\vx,\vy;\vxi)]$, $D_y>0$ is the diameter of the feasible set $\fY$, $\epsilon>0$ is the target accuracy, and $\vy_0\in\fY$ is the initial point of the algorithm.
We also define the smooth surrogate for the regularized function $\tilde{f}(\vx,\vy)$ as follows 
\begin{align}\label{eq:tilde-f-delta} 
    \tilde f_\delta(\vx,\vy) \triangleq \BE_{(\vu,\vv)\sim {\mathrm{Unif}}(\BB ^{d}(\vzero,1))} [\tilde f(\vx+\delta\vu, \vy+\delta\vv)].
\end{align}
The following proposition provides a closed form expression for the function $\tilde f_\delta(\vx,\vy)$.

\begin{proposition}\label{prop: the smoothed function of the regularized}
    For the regularized function defined in \eqref{prob:regularized}, its smoothing surrogate $\tilde{f}_\delta(\vx,\vy)$ defined in equation \eqref{eq:tilde-f-delta} has the closed form expression
    \begin{align}
    \tilde{f}_\delta(\vx,\vy)=f_\delta(\vx,\vy)-\frac{\epsilon}{2D_y}\|\vy-\vy_0\|^2-\frac{d_y\delta^2\epsilon}{2(d+2)D_y}.
    \end{align}
\end{proposition}
It is clear that the function $\tilde{f}_\delta(\vx,\vy)$ is strongly-concave in~$\vy$ and its gradient is close to that of~$f_\delta$ for small $\epsilon$.
Hence, we can apply PGFDA (Algorithm~\ref{alg:f}) to solve the nonconvex-strongly-concave minimax problem~\eqref{prob:regularized} to achieve an desired approximate GSSP of problem~\eqref{prob:constrained main}.
Following the proof of Theorem \ref{zero-thm}, we can obtain the theoretical guarantee for the general nonconvex-concave case as follows.

\begin{corollary}\label{thm: complexity of f: concave}    
Under Assumptions \ref{asm:Lipschitz}, \ref{asm:convex-closed} and \ref{asm:concave}, we run PGFDA (Algorithm \ref{alg:f}) to solve the regularized minimax problem \eqref{prob:regularized} with the stochastic zeroth-order oracle 
\begin{align*}
\tilde F(\vx,\vy;\vxi) \triangleq F(\vx,\vy;\vxi)-\frac{\epsilon}{2D_y}\|\vy-\vy_0\|^2
\end{align*}
and appropriate parameter settings, then it outputs a $(\eta_x,\tilde\eta_y,\delta,\epsilon)$-GSSP of problem \eqref{prob:constrained main}
and the total stochastic zeroth-order oracle complexity is
\begin{align*}
\fO\left(\frac{(L+\epsilon)^5 d^{3}D_y^3 (\Delta+\delta L)}{ \delta^4 \epsilon^6}\right).     
\end{align*}
in expectation, where 
\begin{align*}    \eta_x=\fO\left(\frac{\delta^3\epsilon^2}{d^{3/2}(L+\epsilon)^3D_y^2}\right)
\qquad\text{and}\qquad
\quad\tilde\eta_y=\fO\left(\frac{\delta}{d^{1/2}(L+\epsilon)}\right).
\end{align*}
\end{corollary}

\section{Finding $(\eta_x,\delta,\epsilon)$-GGSPs for the Primal Function}\label{sec: Phi}

In this section, we first propose nested-loop zeroth-order algorithm for 
finding $(\eta,\delta,\epsilon)$-generalized Goldstein stationary points (GGSPs) with respect to the primal function $\Phi(\vx)$, 
then provide the complexity analysis for our algorithm. 

\subsection{Nested-Loop Projected Gradient-Free Descent Ascent}

Recall that Proposition \ref{Lipshitz of Phi} guarantees the Lipschitz continuity of $\Phi(\vx)$. Hence, it is natural to introduce the smooth surrogate function   
\begin{align}\label{def of Phi_delta}
 \Phi_\delta(\vx) \triangleq \BE_{\vw\sim {\mathrm{Unif}}(\BB^{d_x}(\vzero,1))}\left[\max_{\vy \in \fY} f(\vx+\delta\vw, \vy)\right]   
\end{align}
for some $\delta>0$, then consider the minimization problem
\begin{align*}
    \min_{\vx\in\fX} \Phi_\delta(\vx).
\end{align*}
Following the idea of uniform smoothing, we introduce the zeroth-order gradient estimator for the function $\Phi_\delta(\vx)$ as follows
\begin{align*}
\vg^*(\vx;\vw,\vxi) \triangleq \frac{d_x}{2\delta}(F(\vx+\delta\vw, \vy^*(\vx+\delta\vw);\vxi)-F(\vx-\delta\vw,\vy^*(\vx-\delta\vw);\vxi))\cdot\vw,
\end{align*}
where $\vw\sim{\rm Unif}(\mathbb{S}^{d_x-1})$ and $\vxi\sim\fD$.
Since the exact $\vy^*(\cdot)$ may be unavailable, we apply Gradient-Free-Descent (Algorithm \ref{alg:minimization}) to establish $\vy^+\approx\vy^*(\vx+\delta\vw)$ and $\vy^-\approx\vy^*(\vx-\delta\vw)$ for given $\vx$ and $\vw$ in practice, leading to the estimator
\begin{align*}
    \hat{\vg}(\vx,\vy^+,\vy^-;\vw,\vxi) \triangleq  \frac{d_x}{2\delta}\big(F(\vx+\delta\vw,\vy^+;\vxi) - F(\vx-\delta\vw,\vy^-;\vxi)\big)\cdot\vw.
\end{align*}
Based on above ideas, we propose Nested-Loop Projected Gradient-Free Descent Ascent (NL-PGFDA), which is detailed in Algorithm \ref{alg:Phi}.

\subsection{The Complexity Analysis for NL-PGFDA}\label{sec:convergence NL-PGFDA}

For the nonconvex-strongly-concave case,
we provide the following theoretical guarantee for achieving a $(\eta,\delta,\epsilon)$-GGSP with respect to $\Phi(\vx)$.

\begin{algorithm}[t]
\caption{Nested-Loop Projected Gradient-Free Descent Ascent (NL-PGFDA)}
\label{alg:Phi}
\begin{algorithmic}[1]   
\STATE \textbf{Input:} initial point $(\vx_0, \vy_0)\in\fX\times\fY$, step sizes $\eta, \tilde\eta_y>0$, iterations numbers $T, K\in\BN$, smoothing parameter $\delta,\tilde \delta>0$, batch size $b\in\BN$. \\[0.02cm]
\STATE $\vy^+_{-1,i}=\vy_0$ and $\vy^-_{-1,i}=\vy_0$ for all $i\in[b]$ \\[0.02cm]
\STATE \textbf{for}  $t = 0,1,2,\dots, T-1$ \\[0.02cm]
\STATE\qquad   independently sample $\vw_{t,i}\sim{\rm Unif}(\BS^{d-1})$ for all $i\in[b]$ \label{line:nl-pgfda-w} \\[0.05cm]
\STATE\qquad  \textbf{for} $i = 1,2,\dots,b$ \\[0.02cm]
\STATE\qquad \qquad $\vy^+_{t,i}=\text{Gradient-Free-Descent}(-F(\vx_t+\delta\vw_{t,i},\,\cdot\,;\,\cdot\,), \vy^+_{t-1,i}, \mu,\tilde\delta, K)$ \label{line:y+} \\[0.02cm]
\STATE\qquad \qquad $\vy^-_{t,i}=\text{Gradient-Free-Descent} (-F(\vx_t-\delta\vw_{t,i},\,\cdot\,;\,\cdot\,), \vy^-_{t-1,i},\mu,\tilde\delta, K)$ \label{line:y-} \\[0.02cm]
\STATE\qquad \textbf{end for} \\[0.02cm]
\STATE\qquad independently sample $\vxi_{t,i}\sim\fD$ for all $i\in[b]$ \label{line:nl-pgfda-xi} \\[0.02cm]
\STATE\qquad  $\displaystyle{\hat\vu_t=\frac{1}{b}\sum_{i=1}^b\hat{\vg}(\vx_t,\vy_{t,i}^+,\vy_{t,i}^-;\vw_{t,i},\vxi_{t,i})}$ \\[0.02cm]
\STATE\qquad  $\vx_{t+1}=\Pi_\mathcal{X} (\vx_t-\eta\hat\vu_t)$ \\[0.05cm]
\STATE \textbf{end for}\\[0.05cm]
\STATE $j\sim{\rm Unif}(\{0,\dots,T-1\})$, \quad $\vx_{\mathrm{out}}=\vx_j$   \\[0.05cm]
\STATE \textbf{return} $\vx_{\mathrm{out}}$ 
\end{algorithmic}
\end{algorithm}

\begin{theorem}\label{Theorem of Phi}
Under Assumptions  \ref{asm:Lipschitz}, \ref{asm:convex-closed}, \ref{asm:strongly-concave}, run Algorithm \ref{alg:Phi} (NL-PGFDA) with

\[
\eta=\Theta\left(\frac{\delta}{d_x^{1/2}L}\right),\quad \!\! T=\Theta\left(\frac{d_x^{1/2}L(\Delta+\delta L)}{\delta\epsilon^2}\right),\quad\!\!
b=\Theta\left(\frac{d_xL^2}{\epsilon^2}\right),
\quad\!\!
\text{and}\quad\!\!
K=\Theta\left(\frac{d_x^2d_yL^4}{\mu^2\delta^2\epsilon^2}\right),
\]
then it outputs a $(\eta,\delta,\epsilon)$-GGSP 
and the total stochastic zeroth-order oracle complexity is 
\[
\fO\left(\frac{ d_x^{7/2} d_y L^7 (\Delta+\delta L)}{\mu^2\delta^3 \epsilon^6}\right)
\]
in expectation.
\end{theorem}

For the general nonconvex-concave case, 
we introduce the following regularized function
\begin{align}\label{eq:tilde-Phi}
    \tilde\Phi(\vx)\triangleq\max_{\vy\in\fY} \Big(f(\vx,\vy)-\frac{\delta\epsilon}{2 d_x D_y^2}\|\vy-\vy_0\|^2\Big),
\end{align}
and its smooth surrogate 
\begin{align}\label{Phi_mu,Phi_delta,mu}
\tilde\Phi_{\delta}(\vx) \triangleq \mathbb{E}_{\vw\sim \mathrm{Unif}(\mathbb{B}^{d_x}(\vzero,1))}\big[\tilde\Phi(\vx + \delta \vw)\big].   
\end{align}
We connect the gradient mappings with respect to $\nabla\Phi_\delta(\cdot)$ and $\nabla\tilde\Phi_{\delta}(\cdot)$ as follows.

\begin{proposition}\label{gradient distance of the strongly concave approximation}
  Under Assumption \ref{asm:Lipschitz}, \ref{asm:convex-closed}, \ref{asm:concave},  it holds
\begin{align*}
\big\|G(\vx,\nabla\Phi_\delta(\vx),\eta) - G(\vx,\nabla\tilde\Phi_{\delta}(\vx),\eta)\big\|
\le \frac{\epsilon}{2}.
\end{align*}
\end{proposition}
The above proposition indicates that we can reduce the task of
finding an approximate GGSP with respect to function $\Phi(\vx)$
into finding an approximate GGSP with respect to function $\tilde\Phi(\vx)$.
It is worth noting that the objective of the regularized problem
\begin{align}\label{prob:regularized2}
\min_{\vx\in\fX}\max_{\vy\in\fY} f(\vx,\vy)-\frac{\delta\epsilon}{2 d_x D_y^2}\|\vy-\vy_0\|^2    
\end{align}
is strongly-concave in $\vy$ when $f(\vx,\vy)$ is concave in $\vy$.
Hence,  we can apply Algorithm \ref{alg:Phi} (NL-PGFDA)
to solve the nonconvex-strongly-concave minimax problem \eqref{prob:regularized2}
to achieve the desired approximate GGSP with respect to $\tilde\Phi(\vx)$, 
also the desired approximate GGSP with respect to $\Phi(\vx)$.
We formally present the theoretical guarantees as follows.

\begin{corollary}\label{thm: complexity of phi: concave}
    Under Assumptions  \ref{asm:Lipschitz}, \ref{asm:convex-closed}, and \ref{asm:concave}, we run NL-PGFDA (Algorithm \ref{alg:Phi}) to solve the regularized minimax problem \eqref{prob:regularized2} with the stochastic zeroth-order oracle 
\begin{align*}
\tilde F_\Phi(\vx,\vy;\vxi) \triangleq F(\vx,\vy;\vxi) - \frac{\delta\epsilon}{2 d_x D_y^2}\|\vy-\vy_0\|^2    
\end{align*}
and appropriate parameter settings, then it outputs a $(\eta,\delta,\epsilon)$-GGSP with respect to $\Phi(\vx)$ and the total stochastic zeroth-order oracle complexity is
\begin{align*}
\fO\left(\frac{ d_x^{11/2} d_y \big(L+\delta\epsilon/(d_x D_y)\big)^7D_y^4 \big(\Delta+\delta L+\delta^2\epsilon/(d_x D_y)\big)}{\delta^5 \epsilon^8}\right),    
\end{align*}
where $\eta=\Theta\big(\delta d_xD_y/(d_x^{3/2}D_y L+d_x^{1/2}\delta\epsilon)\big)$.
\end{corollary}

\section{Comparison with Related Work}\label{sec:disscussion}

In a recent work, \citet{lin2025two} considered the nonsmooth minimax problem 
\begin{equation}\label{prob:lin}
\min_{\vx\in\BR^{d_x}}\max_{\vy\in\mathcal{Y}}f(\vx,\vy)
\end{equation}
where the objective $f(\vx,\vy)$ is $L_f$-Lipschitz continuous and concave in $\vy$.
They provided two-timescale first-order methods to find the $\epsilon$-stationary point of the Moreau envelope of the primal function $\Phi(\vx) \triangleq \max_{\vy \in \mathcal{Y}} f(\vx, \vy)$, 
i.e., achieving the point $\vx_{\rm out}\in\BR^{d_x}$ such that 
\begin{align*}
\|\nabla \Phi_{1/2\rho}(\vx_{\rm out})\|\le\epsilon,
\end{align*}
where the Moreau envelope is defined as
\begin{align}\label{eq:ME}
\Phi_{1/2\rho}(\vx)=\min_{\vw\in\mathbb{R}^{d_x}} \left(\Phi(\vw)+\rho\|\vw-\vx\|^2\right).    
\end{align}
For achieved $\epsilon$-stationary point $\vx_{\rm out}$,
Lemma 9 of \citet{lin2025two} indicates there exists some $\hat\vx\in\BR^{d_x}$, such that
$\min_{\vg\in\partial \Phi(\hat\vx)}\|\vg\|\le\epsilon$
and
$\|\vx_{\rm out}-\hat\vx\|\le{\epsilon}/{(2\rho)}$.
Furthermore, it holds
\begin{align*}
\left\{\partial \Phi(\hat\vx):\|\vx_{\rm out}-\hat\vx\|\le\frac{\epsilon}{2\rho}\right\}\subseteq
 \operatorname{conv} \left(\bigcup_{\vx \in \BB(\vx_{\rm out},\epsilon/2\rho)} \partial \Phi(\vx) \right)
 =\partial_{\epsilon/2\rho} \Phi(\vx).
\end{align*}
Therefore, the point $\vx_{\rm out}$ is a $(\epsilon/2\rho,\epsilon)$-Goldstein stationary point of $\Phi(\vx)$.

We summarize the main difference between results of \citet{lin2025two} and ours as follows. 
\begin{enumerate}[leftmargin=0.6cm,topsep=1.5pt,itemsep=2pt,partopsep=2pt,parsep=2pt]
\item \citet{lin2025two} requires the Moreau envelope \eqref{eq:ME} be well-defined, which depends on the weakly convex assumption, i.e., the function $f(\vx,\vy)+\frac{\rho}{2}\|\vx\|^2$ is convex in~$\vx$ for some~$\rho>0$.
In contrast, our results do not depend on the weak convexity, which includes more practical applications. 
\item The results of \citet{lin2025two} with respect to the Moreau envelope only focus on the case where $\vx$ is unconstrained, while our results with respect to approximate Goldstein saddle stationary points and generalized Goldstein stationary points address the more general constrained case.
\item In the view of $(\delta,\epsilon)$-Goldstein stationary point of $\Phi(\vx)$, the result of \citet{lin2025two} restricts the value of $\delta$ to depend on the weakly convex parameter $\rho$. 
In contrast, our results (Theorem \ref{Theorem of Phi} and Corollary \ref{thm: complexity of phi: concave} for the unconstrained case) holds for any $\delta>0$.
\item \citet{lin2025two} focus on the first-order methods, which requires the objective function is $L_f$-Lipschitz continuous and the stochastic first-order oracle has bounded variance $\sigma^2$.
In contrast, our work focus on the zeroth-order methods, which requires the stronger mean-squared Lipschitz continuity (Assumption \ref{asm:Lipschitz}) but does not need the bounded variance assumption.
\end{enumerate}

\section{Numerical Experiments}\label{Sec: Experiments}

We evaluate proposed PGFDA (Algorithm \ref{alg:f}) and NL-PGFDA (Algorithm \ref{alg:Phi}) on the problem of poisoning attack \citep{biggio2012poisoning,huang2022accelerated,liu2020min}, which is formulated by the minimax problem
\begin{align*}
\min_{\mathbf{x}\in\mathcal{X}} \max_{\mathbf{y}\in\mathcal{Y}}
f(\vx,\vy)
= \frac{1}{|\fD_p|}\sum_{(a_i,b_i)\in\fD_p}\!\!\!\!l(\vx,\vy;\va_i,b_i) 
+ \frac{1}{|\fD_t|}\sum_{(a_i,b_i)\in\fD_t}\!\!\!\!l(\vx,\vzero;\va_i,b_i) 
+ \lambda R(\vx),
\end{align*}
where $\{\va_i,b_i\}_{i=1}^n=\fD_p\cup\fD_t$ is the training dataset of $n$ samples with the feature $\va_i\in\mathbb{R}^d$ and the label $b_i\in\{-1,1\}$, $\vx=[x_1,\dots,x_d]^\top\in\BR^d$ is the parameter of the model,~$\vy\in\BR^d$ is the perturbation vector for corruption, $\lambda>0$ is the regularized hyperparameter, $\fX=\BR^d$, and~$\mathcal{Y}=\{\vy:\|\vy\|_\infty\le r\}$ for some $r>0$.
Here, we take the hinge loss \citep{biggio2012poisoning}
\begin{align*}
l_i(\vx,\vy;\va_i,b_i)=\max\{1-b_i(\va_i+\vy)^\top \vx, 0\}    
\end{align*}
and the capped-$\ell_1$ regularization \citep{zhang2010analysis}
\begin{align*}
R(\vx)=\sum_{j=1}^d\min\{|x_j|,\beta\}.
\end{align*}
for some $\beta>0$.

We perform our experiments on datasets \texttt{a9a} ($n=32561$, $d=123$), \texttt{w8a} ($n=49749$, $d=300$), and \texttt{ijcnn1} ($n=49990$, $d=22$) \citep{chang2011libsvm}.
For the problem of poisoning attack,
we let the corrupted rate $|\fD_p|/(|\fD_p|+|\fD_t|)$ be 0.15 and set $r=2$, $\lambda=10^{-5}/n$, and $\beta=2$.

For proposed methods, we tune the stepsizes $\eta$, $\eta_x$, and $\eta_y$ from
$\{10^{-5}, 5\times 10^{-5}, 10^{-4}, 5\times 10^{-4}, 10^{-3}, 
5\times 10^{-3},10^{-2}, 5\times 10^{-2}, 10^{-1}\}$
and the smoothing parameters $\delta$ and~$\tilde\delta$ from
$\{10^{-3}, 5\times 10^{-3}, 10^{-2}, 5\times 10^{-2}, 10^{-1}\}$.
For PGFDA, we set 
the probability as $p=0.1$ and the batch sizes as $b=100$ and $\tilde{b}=1000$.
For NL-PGFDA, we set batch size as $b=300$ and the number of inner-loop iterations as $K=50$.
We present the empirical results in Figure \ref{fig:three_comparison} to demonstrate the number of stochastic zeroth-order oracle (SZO) calls against the primal function value $\Phi(\vx)$.
We can observe that both two proposed methods converge successfully.

\begin{figure}[ht]
    \centering
    \begin{subfigure}[b]{0.325\textwidth}
        \centering
        \includegraphics[width=\linewidth]{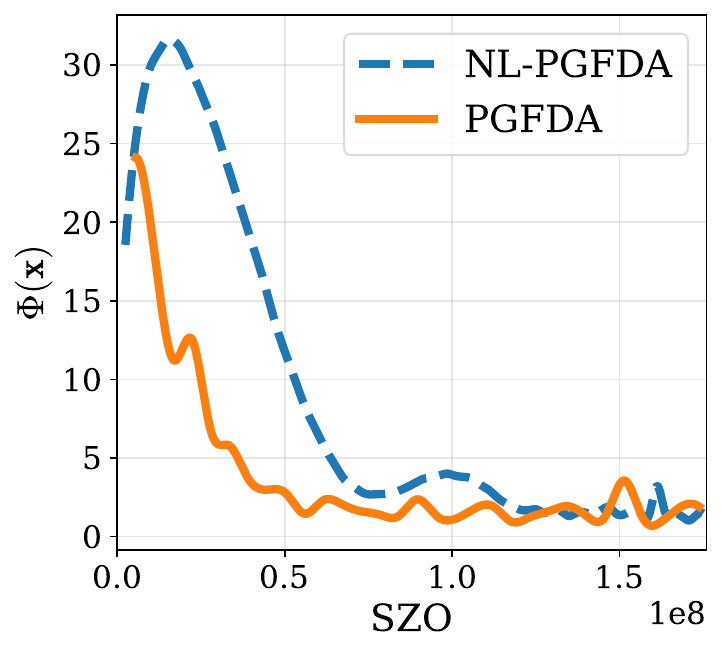}
        \caption{a9a}
    \end{subfigure}
    \hfill 
    \begin{subfigure}[b]{0.325\textwidth}
        \centering
        \includegraphics[width=\linewidth]{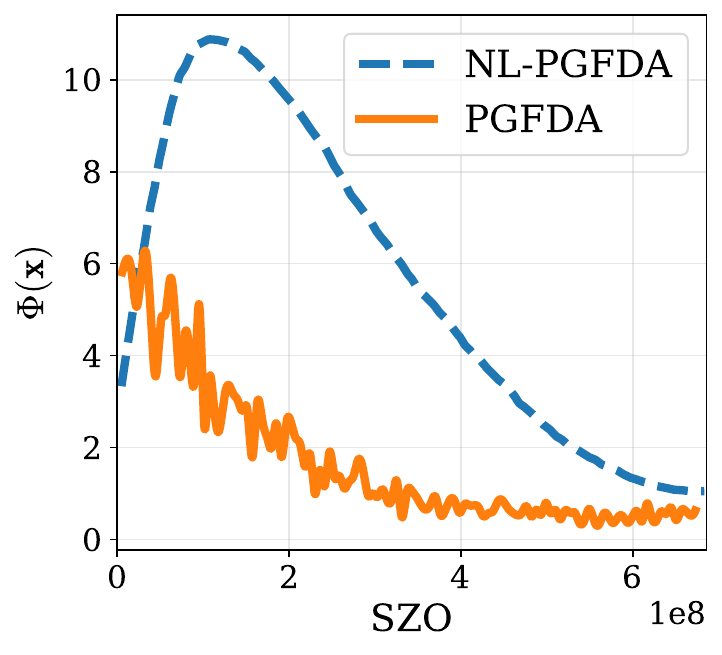}
        \caption{w8a}
    \end{subfigure} 
    \hfill
    \begin{subfigure}[b]{0.325\textwidth}
        \centering
        \includegraphics[width=\linewidth]{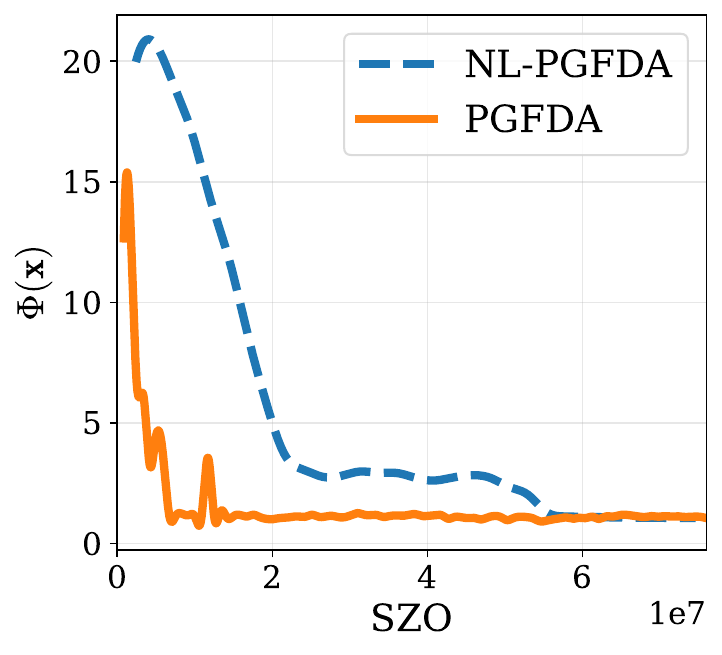}
        \caption{ijcnn1}
    \end{subfigure}
    \caption{Stochastic zeroth-order oracle (SZO) calls against the primal function value $\Phi(\vx)$}
    \label{fig:three_comparison}
\end{figure}

\section{Conclusion and Future Directions}\label{sec:conclusion}

In this work, we introduce approximate Goldstein saddle stationary points to characterize the stationarity of general constrained nonsmooth nonconvex-concave minimax problem without weak convexity. 
We then establish gradient-free descent ascent methods with non-asymptotic convergence rates with respect to this measure.
Furthermore, we propose nested-loop method to achieve  approximate generalized Goldstein stationary points of the primal function.
The convergence of proposed methods also have been verified by numerical experiments.

In future work, we would like to study first-order methods for nonsmooth nonconvex-concave minimax problem.
It is also interesting to consider the convergence of specific nonsmooth nonconvex-nonconcave minimax problem with respect to proposed approximate Goldstein saddle stationary points \citep{li2025nonsmooth,lu2025first,khan2026variance}.

\bibliographystyle{plainnat}
\bibliography{reference}

\appendix 
\makeatletter
\g@addto@macro\appendix{%
  \renewcommand\normalsize{\small}
  \renewcommand\small{\small}
  \renewcommand\footnotesize{\small}
}
\makeatother

\def\E{{\mathbb{E}}}
\def\vzero{{\bf{0}}}
\def\vone{{\bf{1}}}
\def\vmu{{\bm{\mu}}}
\def\vphi{{\bm{\phi}}}
\def\veta{{\bm{\eta}}}
\def\valpha{{\bm{\alpha}}}
\def\vbeta{{\bm{\beta}}}
\def\vtheta{{\bm{\theta}}}
\def\va{{\bf{a}}}
\def\vb{{\bf{b}}}
\def\vc{{\bf{c}}}
\def\vd{{\bf{d}}}
\def\ve{{\bf{e}}}
\def\vf{{\bf{f}}}
\def\vg{{\bf{g}}}
\def\vh{{\bf{h}}}
\def\vi{{\bf{i}}}
\def\vj{{\bf{j}}}
\def\vk{{\bf{k}}}
\def\vl{{\bf{l}}}
\def\vm{{\bf{m}}}
\def\vn{{\bf{n}}}
\def\vo{{\bf{o}}}
\def\vp{{\bf{p}}}
\def\vq{{\bf{q}}}
\def\vr{{\bf{r}}}
\def\vs{{\bf{s}}}
\def\vt{{\bf{t}}}
\def\vu{{\bf{u}}}
\def\vv{{\bf{v}}}
\def\vw{{\bf{w}}}
\def\vx{{\bf{x}}}
\def\vy{{\bf{y}}}
\def\vz{{\bf{z}}}

\def\fA{{\mathcal{A}}}
\def\fB{{\mathcal{B}}}
\def\fC{{\mathcal{C}}}
\def\fD{{\mathcal{D}}}
\def\fE{{\mathcal{E}}}
\def\fF{{\mathcal{F}}}
\def\fG{{\mathcal{G}}}
\def\fH{{\mathcal{H}}}
\def\fI{{\mathcal{I}}}
\def\fJ{{\mathcal{J}}}
\def\fK{{\mathcal{K}}}
\def\fL{{\mathcal{L}}}
\def\fM{{\mathcal{M}}}
\def\fN{{\mathcal{N}}}
\def\fO{{\mathcal{O}}}
\def\fP{{\mathcal{P}}}
\def\fQ{{\mathcal{Q}}}
\def\fR{{\mathcal{R}}}
\def\fS{{\mathcal{S}}}
\def\fT{{\mathcal{T}}}
\def\fU{{\mathcal{U}}}
\def\fV{{\mathcal{V}}}
\def\fW{{\mathcal{W}}}
\def\fX{{\mathcal{X}}}
\def\fY{{\mathcal{Y}}}
\def\fZ{{\mathcal{Z}}}

\def\sA{{\mathbb{A}}}
\def\sB{{\mathbb{B}}}
\def\sC{{\mathbb{C}}}
\def\sD{{\mathbb{D}}}
\def\BE{{\mathbb{E}}}
\def\BB{{\mathbb{B}}}
\def\BF{{\mathbb{F}}}
\def\BG{{\mathbb{G}}}
\def\BH{{\mathbb{H}}}
\def\BI{{\mathbb{I}}}
\def\BJ{{\mathbb{J}}}
\def\BK{{\mathbb{K}}}
\def\BL{{\mathbb{L}}}
\def\BM{{\mathbb{M}}}
\def\BN{{\mathbb{N}}}
\def\BO{{\mathbb{O}}}
\def\BP{{\mathbb{P}}}
\def\BQ{{\mathbb{Q}}}
\def\BR{{\mathbb{R}}}
\def\BS{{\mathbb{S}}}
\def\BT{{\mathbb{T}}}
\def\BU{{\mathbb{U}}}
\def\BV{{\mathbb{V}}}
\def\BW{{\mathbb{W}}}
\def\BX{{\mathbb{X}}}
\def\BY{{\mathbb{Y}}}
\def\BZ{{\mathbb{Z}}}

\def\mA {{\bf A}}
\def\mB {{\bf B}}
\def\mC {{\bf C}}
\def\mD {{\bf D}}
\def\mE {{\bf E}}
\def\mF {{\bf F}}
\def\mG {{\bf G}}
\def\mH {{\bf H}}
\def\mI {{\bf I}}
\def\mJ {{\bf J}}
\def\mK {{\bf K}}
\def\mL {{\bf L}}
\def\mM {{\bf M}}
\def\mN {{\bf N}}
\def\mO {{\bf O}}
\def\mP {{\bf P}}
\def\mQ {{\bf Q}}
\def\mR {{\bf R}}
\def\mS {{\bf S}}
\def\mT {{\bf T}}
\def\mU {{\bf U}}
\def\mV {{\bf V}}
\def\mW {{\bf W}}
\def\mX {{\bf X}}
\def\mY {{\bf Y}}
\def\mZ {{\bf Z}}

\def\AUC{{\rm{AUC}}}
\def\SVI{{\rm{SVI}}}
\def\VI{{\rm{VI}}}
\def\ApproxMax{{\rm{ApproxMax}}}

\section{Proofs for Results in Sections \ref{Sec: The stationary points} and \ref{sec:smoothing}}\label{appendix:stationary-smoothing}

In this section, we provide the detailed proofs of results for stationarity of nonsmooth minimax optimization and randomized smoothing. 
In the remainder of the paper, we slightly generalized the notations 
\begin{align}\label{eq:general gradient-mapping-xy}
        G_x(\vx, \vy,\vg_{x},\eta_x) = \frac{\vx-\Pi_\mathcal{X}\left(\vx - \eta_x \vg_{x}\right) }{\eta_x}
        \quad\text{and}\quad
        G_y(\vx, \vy,\vg_{y},\eta_y) = \frac{\Pi_\mathcal{Y}\left(\vy + \eta_y \vg_{y}\right) - \vy}{\eta_y}
\end{align}
defined in equation (\ref{eq:gradient-mapping-xy}) to allow that the vectors $\vg_x$ and $\vg_y$ can be any $d_x$-dimensional vector and $d_y$-dimensional vector, respectively.

\subsection{The Proof of Proposition \ref{Lipshitz of Phi}}
\begin{proof}
This results can be achieved by Lemma 4.2 of \citet{lin2025two}.
We provide the proof in our settings for the completeness.
Recall that $\vy^*(\vx)=\argmax_{\vy\in\fY}f(\vx,\vy)$, which implies
\begin{align*}
& \Phi(\vx_1) - \Phi(\vx_2) 
= f(\vx_1, \vy^*(\vx_1)) - f(\vx_2, \vy^*(\vx_2)) \\
\leq & f(\vx_1, \vy^*(\vx_1)) - f(\vx_2, \vy^*(\vx_1)) 
\leq  L\norm{\vx_1 - \vx_2}
\end{align*} 
for all $\vx_1,\vx_2\in\fX_\delta$.
By changing the roles of $\vx_1$ and $\vx_2$, it also holds
\begin{align*}
\Phi(\vx_2) - \Phi(\vx_1) \leq  L\norm{\vx_1 - \vx_2}.
\end{align*}
Combining above results, we achieve
$|\Phi(\vx_1) - \Phi(\vx_2)|\leq L\norm{\vx_1 - \vx_2}$, which finishes the proof.
\end{proof}

\subsection{The Proof of Proposition \ref{prop: concave of f_delta}}

\begin{proof}
We focus on proving the result under Assumption \ref{asm:strongly-concave}.
For all $\delta>0$, $\vx,\vu\in\fX$, $\vy_1,\vy_2,\vv\in\fY$, and $\lambda\in[0,1]$,  Assumption \ref{asm:strongly-concave} implies it holds that
\begin{align*}
& f\big(\vx+\delta \vu,\; \lambda \vy_1 + (1-\lambda)\vy_2 + \delta \vv\big) \\
\ge & \lambda\, f(\vx+\delta \vu,\, \vy_1 + \delta \vv)
+ (1-\lambda)\, f(\vx+\delta \vu,\, \vy_2 + \delta \vv)
+ \frac{\mu\lambda(1-\lambda)}{2}\|\vy_1 - \vy_2\|^2.
\end{align*}
for all $\vu\in\BR^{d_x}$ and $\vv\in\BR^{d_y}$ such that $||\vu||\leq 1$ and $||\vv||\leq 1$.
We let $(\vu,\vv)$ be random variable uniformly distributed on the unit ball $\BB^{d}(\vzero,1)$, then taking the expectation with respect to $(\vu,\vv)$ on above inequality leads to 
\begin{align*}
& f_\delta(\vx,\, \lambda \vy_1 + (1-\lambda)\vy_2) 
\ge 
\lambda\, f_\delta(\vx, \vy_1)
+ (1-\lambda)\, f_\delta(\vx, \vy_2)
+ \frac{\mu\lambda(1-\lambda)}{2}\|\vy_1 - \vy_2\|^2,
\end{align*}
where we use definition of $f_\delta$ in definition \eqref{eq:f_delta}.
This implies the function $f_\delta(\vx,\vy)$ is $\mu$-strongly concave in $\vy$.  
Taking $\mu=0$ yields the result in the general concave case.
\end{proof}

\subsection{Proof of Lemma \ref{distance between y_star and y_star_delta}}
\begin{proof}
    We prove this lemma by contradiction. 
    Suppose there exists some $\vx\in\BR^{d_x}$ such that
    \begin{align}\label{eq:contradiction-ystar}
        \|\vy^*(\vx)-\vy_\delta^*(\vx)\|> 2\sqrt{\frac{\delta L}{\mu}}.    
    \end{align}     
    Since the function $f(\vx,\vy)$ is $\mu$-strongly concave in $\vy$, it holds
    \begin{align}\label{eq:diff-fd-elta}
    \begin{split}        
    f(\vx,\vy_\delta^*(\vx))&\le f(\vx,\vy^*(\vx))+\langle \vg_y, \vy_\delta^*(\vx)-\vy^*(\vx)\rangle-\frac{\mu}{2}\|\vy_\delta^*(\vx)-\vy^*(\vx)\|^2 \\
    &< f(\vx,\vy^*(\vx))+\langle \vg_y, \vy_\delta^*(\vx)-\vy^*(\vx)\rangle-2\delta L \\
    & \leq f(\vx,\vy^*(\vx))-2\delta L
    \end{split}
    \end{align}
    for some $\vg_y\in \partial_y f(\vx,\vy^*(\vx))$, where the second inequality is based on equation \eqref{eq:contradiction-ystar} and the last inequality is based on the optimality of $\vy^*(\vx)$.
    
    Consequently, we have
    \begin{align*}
      f_\delta(\vx,\vy_\delta^*(\vx))\leq f(\vx,\vy_\delta^*(\vx))+\delta L \overset{\eqref{eq:diff-fd-elta}}< f(\vx,\vy^*(\vx))-\delta L\le f_\delta(\vx,\vy^*(\vx)),  
    \end{align*}
    where the first and the last inequalities are based on Proposition \ref{prop:smoothing}(a).
    Thus, we achieve 
    \begin{align*}
        f_\delta(\vx,\vy^*(\vx))>f_\delta(\vx,\vy_\delta^*(\vx)),
    \end{align*}
    which contradicts the definition  $\vy_\delta^*(\vx)=\argmax_{\vy\in\fY} f_\delta(\vx,\vy)$.    
    Therefore, equation \eqref{eq:contradiction-ystar} does not hold and we conclude $\|\vy_\delta^*(\vx)-\vy^*(\vx)\|\leq 2\sqrt{\delta L/\mu}$ for all $\vx\in\BR^{d_x}$.    

\end{proof}

\subsection{Proof of Lemma \ref{distance between Phi Phi_delta}}
\begin{proof}    
For all given $\vx\in\BR^{d_x}$, $\vy\in\fY$, and  $(\vu,\vv)\sim {\rm Unif}(\BB^d(\vzero,1))$,
it holds
\begin{align}\label{eq:diff-f-delta-f}
    f(\vx+\delta\vu,\vy+\delta\vv) 
    \leq f(\vx,\vy)  + L(\norm{\delta\vu}^2+\norm{\delta\vv}^2)^{1/2}
    \leq f(\vx,\vy) + \delta L,
\end{align}
where the first inequality is based on Assumption \ref{asm:Lipschitz} and the last inequality is based on 
Proposition~\ref{prop:smoothing}(a) and the fact $\norm{\vu}^2+\norm{\vv}^2\leq 1$.
Taking the expectation with respect to $\vu$ and $\vv$ on equation~\eqref{eq:diff-f-delta-f}, we have
\begin{align}\label{eq:f_delta-f}
    f_\delta(\vx,\vy) = \BE_{(\vu,\vv)\sim {\rm Unif}(\BB^d(\vzero,1))}[f(\vx+\delta\vu,\vy+\delta\vv)] 
    \leq f(\vx,\vy) + \delta L.
\end{align}

Consequently, we have
\begin{align*}
    \Psi_\delta(\vx) = f_\delta(\vx,\vy_\delta^*(\vx))
    \overset{(\ref{eq:f_delta-f})}{\leq}  f(\vx,\vy_\delta^*(\vx)) + \delta L
     \overset{(\ref{eq:y-star})}{\leq} f(\vx,\vy^*(\vx)) + \delta L
     = \Phi(\vx) + \delta L
\end{align*}
for all $\vx\in\BR^{d_x}$.
Following the derivation of equation (\ref{eq:f_delta-f}), we can also achieve
\begin{align}\label{eq:f_delta-f-2}
    f(\vx,\vy)  \leq f_\delta(\vx,\vy) + \delta L,
\end{align}
which implies
\begin{align*}
    \Phi(\vx) = f(\vx,\vy^*(\vx))
    \overset{(\ref{eq:f_delta-f-2})}{\leq}  f_\delta(\vx,\vy^*(\vx)) + \delta L
     \overset{(\ref{eq:y-star})}{\leq} f_\delta(\vx,\vy_\delta^*(\vx)) + \delta L
     = \Psi_\delta(\vx) + \delta L.
\end{align*}
Therefore, we have $|\Phi(\vx)-\Psi_\delta(\vx)|\leq \delta L$ for all $\vx\in\BR^{d_x}$.
\end{proof}

\section{Technical Lemmas for Convergence Analysis}\label{appendix:lemmas}
Recall that Propositions \ref{prop:smoothing} and \ref{prop: concave of f_delta} implies the surrogate function $f_\delta(\vx,\vy)$ is $M_{f_\delta}$-smooth with $M_{f_\delta}=cd^{1/2}L\delta^{-1}$ and $\mu$-strongly concave in $\vy$ under Assumptions \ref{asm:Lipschitz} and \ref{asm:strongly-concave}.
Hence, we can apply Danskin's theorem \citep{bertsekas1971control} and Lemma A.6 of \citet{chen2024efficient} to achieve the following results.

\begin{lemma}[{\citet[Theorem A.22]{bertsekas1971control}}]\label{danskin}
Under Assumptions \ref{asm:Lipschitz}, \ref{asm:convex-closed} and \ref{asm:strongly-concave}, 
the mapping $\vy_\delta^*(\vx) = \argmax_{\vy\in\fY} f_\delta(\vx,\vy)$ is uniquely defined
and the function $\Psi_\delta(\vx) = f_\delta(\vx,\vy_\delta^*(\vx))$ is differentiable with 
\begin{align*}
    \nabla \Psi_\delta(\vx)=\nabla_x f_\delta(\vx,\vy_\delta^*(\vx)).
\end{align*}
\end{lemma}

\begin{lemma}\label{lemma: A1}
Under the setting of Lemma \ref{danskin}, it holds 
\begin{align*}
\| \nabla \Psi_\delta(\vx) - \nabla_x f_\delta(\vx, \vy) \| \leq \frac{2M_{f_\delta}}{\mu} \| G_y(\vx, \vy,\nabla_y f_\delta(\vx,\vy),\eta) \|,   
\end{align*}
for all $\vx\in\BR^{d_x}$, $\vy\in\fY$, and $\eta\in(0, 1/M_{f_\delta}]$.
\end{lemma}

Following the generalized definitions for mappings $G_x$ and $G_y$ shown in equation (\ref{eq:general gradient-mapping-xy}), we have following results.

\begin{lemma}\label{lemma: g&G_1}
For all $\vx,\vg_{x,1},\vg_{x,2}\in \mathbb{R}^{d_x}$, $\vy_1,\vy_2\in\BR^{d_y}$, and $\eta_x>0$, it holds
\begin{align*}
&\| G_x(\vx,\vy_1, \vg_{x,1}, \eta_x) - G_x(\vx,\vy_2, \vg_{x,2}, \eta_x) \| \leq\|\vg_{x,1}-\vg_{x,2}\|
\end{align*}
Additionally, for all $\vx_1,\vx_2\in\BR^{d_x}$, $\vy,\vg_{y,1},\vg_{y,2}\in \mathbb{R}^{d_y}$, and $\eta_y>0$, it holds
\begin{align*}
 \| G_y(\vx_1,\vy, \vg_{y,1}, \eta_y) - G_y(\vx_2,\vy, \vg_{y,2}, \eta_y) \| \leq \|\vg_{y,1}-\vg_{y,2}\|.
\end{align*}
\end{lemma}

\begin{proof}
Following definition  (\ref{eq:general gradient-mapping-xy}), it holds
\[
\begin{aligned}
& \bigl\|G_x(\vx,\vy_1,\vg_{x,1},\eta_x)-G_x(\vx,\vy_2,\vg_{x,2},\eta_x)\bigr\|
= \frac{1}{\eta_x}\bigl\|\Pi_{\mathcal{X}}(\vx-\eta_x\vg_{x,1})-\Pi_{\mathcal{X}}(\vx-\eta_x\vg_{x,2})\bigr\|  \\
&\le \frac{1}{\eta_x}\big\|\bigl(\vx-\eta_x\vg_{x,1}\bigr)-\bigl(\vx-\eta_x \vg_{x,2}\bigr)\big\| 
= \|\vg_{x,1}-\vg_{x,2}\|.
\end{aligned}
\]
where the inequality is based on the expansiveness of the projection operator. Similarly, we can also achieve the desired result for $G_y$.
\end{proof}

\begin{lemma}\label{lemma: g&G_2}
For all $\vx\in\fX$, $\vy\in\fY$, $\vg_x\in\BR^{d_x}$, $\vg_y\in\BR^{d_y}$, and $\eta>0$, it holds
\[
\langle \vg_x, G_x(\vx,\vy, \vg_x, \eta) \rangle \geq \|G_x(\vx,\vy, \vg_x, \eta)\|^2
\quad\text{and}\quad
\langle \vg_y, G_y(\vx,\vy, \vg_y, \eta) \rangle \geq \|G_y(\vx,\vy, \vg_y, \eta)\|^2.
\]
Applying Cauchy--Schwarz inequality, we also achieve
\[
\|\vg_x\| \geq \|G_x(\vx,\vy, \vg_x, \eta)\|
\quad\text{and}\quad
\|\vg_y\| \geq \|G_y(\vx,\vy, \vg_y, \eta)\|.
\]
\end{lemma}

\begin{proof}
We denote $\vx^+ := {\Pi}_\fX\left(\vx - \eta \vg_x\right)$. 
The convexity and closedness of set $\fX$ implies for all $\vu \in \fX$, it holds
\begin{equation*}
    \left\langle \vx^+ - (\vx - \eta \vg_x), \vu - \vx^+ \right\rangle \geq 0
\end{equation*}
Dividing above equation by $\eta^2$ and taking $\vu = \vx$, we achieve
\begin{equation*}
    \frac{1}{\eta^2}\left\langle \vx^+ - (\vx - \eta \vg_x), \vx - \vx^+ \right\rangle \geq 0.
\end{equation*}
By rearranging above inequality, the definition of $G_x(\vx,\vy, \vg_x, \eta)$ leads to
\begin{equation*}
\left\langle \vg_x, G_x(\vx,\vy, \vg_x, \eta) \right\rangle  
= \left\langle \vg_x, \frac{\vx - \vx^+}{\eta} \right\rangle  
\geq \frac{1}{\eta^2}\left\langle \vx- \vx^+  , \vx - \vx^+\right\rangle  
= \|G_x(\vx,\vy, \vg_x, \eta)\|^2.
\end{equation*}
Following the above derivation, we also achieve
\[
\langle \vg_y, G_y(\vx,\vy, \vg_y, \eta) \rangle \geq \|G_y(\vx,\vy, \vg_y, \eta)\|^2.
\]
Applying Cauchy--Schwarz inequality on above two results, we finish the proof.
\end{proof}

\begin{lemma}\label{Function value optimality controls the gradient mapping}
Under Assumptions \ref{asm:Lipschitz}, \ref{asm:convex-closed}, and \ref{asm:concave}, it holds
\[
\|G_y(\vx,\vy,\nabla_y f_\delta(\vx,\vy),\eta)\|^2
\;\le\;
\frac{2}{\eta}\bigl(f_\delta(\vx,\vy_\delta^*(\vx))-f_\delta(\vx,\vy)\bigr).
\]
for all $\vx\in\BR^{d_x}$, $\vy\in\mathcal Y$, and $\eta\in(0,1/M_{f_\delta}]$, where $\vy_\delta^*(\vx)\in\argmax_{\vy\in\fY} f_\delta(\vx,\vy)$ and $M_{f_\delta}>0$ follows Proposition \ref{prop:smoothing}.
\end{lemma}

\begin{proof}
We denote $\vy^+ := \Pi_{\mathcal Y}\bigl(\vy+\eta\nabla_y f_\delta(\vx,\vy)\bigr)$.
According to Proposition \ref{prop:smoothing}(c), it holds
\begin{align}\label{eq:smooth-proj}
f_\delta(\vx,\vy^+)\ge f_\delta(\vx,\vy)+\langle \nabla f_\delta(\vx,\vy),\vy^+-\vy\rangle-\frac{M_{f_\delta}}{2}\|\vy^+-\vy\|^2.    
\end{align}
The convexity and closedness of set $\fY$ implies for all $\vv\in\fY$, it holds
\begin{equation*}
\bigl\langle \vy^+ -(\vy+\eta\nabla_y f_\delta(\vx,\vy)),\vv-\vy^+\bigr\rangle \ge 0,
\end{equation*}
Taking $\vu=\vy$ for above inequality, we achieve 
\begin{align}\label{eq: projection ineq}
\eta\langle \nabla_y f_\delta(\vx,\vy),\,\vy^+-\vy\rangle \ge \|\vy-\vy^+\|^2.
\end{align}
Combining equations \eqref{eq:smooth-proj} and \eqref{eq: projection ineq}, we achieve
\begin{align}\label{eq:smooth-proj2}
\begin{split}    
& f_\delta(\vx,\vy^+)-f_\delta(\vx,\vy)\ge \langle\nabla_y f_\delta(\vx,\vy),\vy^+-\vy\rangle-\frac{M_{f_\delta}}{2}\|\vy^+-\vy\|^2 \\
\ge & \left(\frac{1}{\eta}-\frac{M_{f_\delta}}{2}\right)\|\vy-\vy^+\|^2
\ge \frac{1}{2\eta}\|\vy-\vy^+\|^2 
= \frac{\eta}{2}\|G_y(\vx,\vy,\nabla_y f_\delta(\vx,\vy),\eta)\|^2,
\end{split}
\end{align}
where the second inequality is based on the assumption $\eta<1/M_{f_\delta}$ and the last step is based on the definition of $G_y$.
According to the fact $f_\delta(\vx,\vy_\delta^*(\vx))\ge f_\delta(\vx,\vy^+)$, it holds
\begin{align*}
f_\delta(\vx,\vy_\delta^*(\vx))-f_\delta(\vx,\vy) \ge f_\delta(\vx,\vy^+)-f_\delta(\vx,\vy) \overset{\eqref{eq:smooth-proj2}}\ge \frac{\eta}{2}\|G_y(\vx,\vy,\nabla_y f_\delta(\vx,\vy),\eta)\|^2,   
\end{align*}
which completes the proof.
\end{proof}

\section{Proofs for Results in Section \ref{sec:convergence-PGFDA}}\label{Proof of Section 4}

This section provides detailed proofs for finding approximate Goldstein saddle stationary point (GSSPs) of the objective $f(\vx,\vy)$.

\subsection{Proof of Lemma \ref{lemma: gradient mapping between Psi_delta and f_delta}}
\begin{proof}
For all $\vx\in\fX$ and $\vy\in\fY$, the triangle inequality implies
\begin{align}\label{eq:triangle ineq for gradient mapping}
\begin{split}
    &\norm{G_x(\vx,\vy,\nabla_x f_\delta(\vx,\vy),\eta_x)}\\
\leq&  \norm{G(\vx,\nabla_x \Psi_\delta(\vx),\eta_x)} + \norm{G(\vx,\nabla_x \Psi_\delta(\vx),\eta_x)
-G_x(\vx,\vy,\nabla_x f_\delta(\vx,\vy),\eta_x)}.
\end{split}
\end{align}
Note that the term $G(\vx,\nabla_x \Psi_\delta(\vx),\eta_x)$ can be written as
\begin{align}\label{eq:GxPsi-delta}
\begin{split}    
   G(\vx,\nabla_x \Psi_\delta(\vx),\eta_x) 
= & \frac{\vx-\Pi_\mathcal{X}(\vx-\eta_x \nabla \Psi_\delta(\vx))}{\eta_x} \\
= & \frac{\vx-\Pi_\fX(\vx-\eta_x \nabla_x f_\delta(\vx,\vy_\delta^*(\vx)))}{\eta_x} \\
= & G_x(\vx,\vy,\nabla_x f_\delta(\vx,\vy_\delta^*(\vx)),\eta_x),
\end{split}
\end{align}
where second step is based on Lemma \ref{danskin}. We then bound the second term in the last line of equation~\eqref{eq:triangle ineq for gradient mapping} as follows
\begin{align}\label{eq:triangle-term-2}
\begin{split}    
   & \norm{G(\vx,\nabla_x \Psi_\delta(\vx),\eta_x)
-G_x(\vx,\vy,\nabla_x f_\delta(\vx,\vy),\eta_x)} \\
\overset{\eqref{eq:GxPsi-delta}}= & \norm{G_x(\vx,\vy,\nabla_x f_\delta(\vx, \vy_\delta^*(\vx)),\eta_x)
-G_x(\vx,\vy,\nabla_x f_\delta(\vx,\vy),\eta_x)}    \\
~\le~& \norm{\nabla_x f_\delta(\vx, \vy_\delta^*(\vx)) - \nabla_x f_\delta(\vx,\vy)} 
\leq \frac{2M_{f_\delta}}{\mu}\| G_y(\vx, \vy,\nabla_y f_\delta(\vx,\vy),\tilde\eta_y) \|,
\end{split}
\end{align}
where the first inequality is based on Lemma \ref{lemma: g&G_1} and the second inequality is based on Lemma \ref{lemma: A1}.
Combining equations \eqref{eq:triangle ineq for gradient mapping} and  \eqref{eq:triangle-term-2}, we finish the proof.
\end{proof}

\subsection{Proof of Lemma \ref{thm:zeroth-order minimization}}

\begin{proof}
We denote 
$\tilde{\vy}_{k+1}=\vy_k-\eta_k\vv_k$. For all $\hat\vy\in\fY$, it holds
\begin{align}\label{ineq: v_k, y_k-y*}
\begin{split}
& \E[\langle \vv_k, \vy_k - \hat\vy \rangle]  
=  \E\left[\frac{1}{\eta_k} \langle \vy_k - \tilde{\vy}_{k+1}, \vy_k -  \hat\vy \rangle \right]\\
=& \E\left[\frac{1}{2\eta_k} \Big( \|\vy_k - \tilde{\vy}_{k+1}\|^2 + \|\vy_k -  \hat\vy\|^2 - \|\tilde{\vy}_{k+1} -  \hat\vy\|^2 \Big)\right]\\
=& \E\left[\frac{1}{2\eta_k} \Big( \eta_k^2 \|\vv_k\|^2 + \|\vy_k -  \hat\vy\|^2 - \|\tilde{\vy}_{k+1} -  \hat\vy\|^2 \Big)\right] \\
\le& \E\left[\frac{1}{2\eta_k} \Big(  16\sqrt{2\pi }d_y\eta_k^2L^2 + \|\vy_k -  \hat\vy\|^2 - \|\tilde{\vy}_{k+1} -  \hat\vy\|^2 \Big)\right] \\
\le& \E\left[\frac{1}{2\eta_k} \Big( \|\vy_k - \hat\vy\|^2 - \|\vy_{k+1} -  \hat\vy\|^2 \Big) + 8\sqrt{2\pi }d_y\eta_k L^2\right],
\end{split}
\end{align}
where first inequality is based on Lemma D.1 of \citet{lin2022gradient} and 
the second inequality is based on the fact that
$\|\vy_{k+1}-\hat\vy\|\le\|\tilde\vy_{k+1}-\hat\vy\|$
due to the convexity and the closedness of the feasible set $\fY$. 

According to Proposition \ref{prop: concave of f_delta}, the function $h_\nu$ is $\mu$-strongly convex on $\fY$. This implies
\begin{align}\label{ineq: h(y_k)-h(y*)}
\begin{split}    
   h_\nu(\vy_k) - h_\nu(\hat\vy) 
 \le &\langle \nabla h_\nu(\vy_k), \vy_k - \hat\vy \rangle - \frac{\mu}{2}\|\vy_k - \hat\vy\|^2 \\
 = &  \E\left[ \langle \vv_k, \vy_k - \hat\vy \rangle - \frac{\mu}{2}\|\vy_k - \hat\vy\|^2 \right],
\end{split}
\end{align}
where the equality is based on $\BE[\vv_k]=\nabla h_\nu(\vy_k)$ from Lemma D.1 of \citet{lin2022gradient}.

Combining equations (\ref{ineq: v_k, y_k-y*})
and (\ref{ineq: h(y_k)-h(y*)}), for given $\vy_k$, it holds 
\begin{align*}
& h_\nu(\vy_k) - h_\nu(\hat\vy) \\
\le& \E\left[\frac{1}{2\eta_k} \Big( \|\vy_k - \hat\vy\|^2 - \|\vy_{k+1} - \hat\vy\|^2 \Big)
    + 8\sqrt{2\pi}d\eta_k L^2 - \frac{\mu}{2}\|\vy_k - \hat\vy\|^2\right] \\
=& \E\left[\frac{\mu(k+1)}{4} \Big( \|\vy_k - \hat\vy\|^2 - \|\vy_{k+1} - \hat\vy\|^2 \Big)
    + \frac{16\sqrt{2\pi}d L^2}{\mu(k+1)} - \frac{\mu}{2}\|\vy_k - \hat\vy\|^2 \right]\\
=& \E\left[\frac{\mu(k-1)}{4}\|\vy_k - \hat\vy\|^2
    - \frac{\mu(k+1)}{4}\|\vy_{k+1} - \hat\vy\|^2
    + \frac{16\sqrt{2\pi}\, d_yL^2}{\mu(k+1)}\right],
\end{align*}
where the first equality holds by setting 
$\eta_k=2/(\mu(k+1))$. This leads to
\begin{align}\label{ineq: iterate of h_nu}
\begin{split}
&k\big(h_\nu(\vy_k) - h_\nu( \hat\vy)\big)\\
\le& \E\bigg[\frac{\mu(k-1)k}{4}\|\vy_k -  \hat\vy\|^2
   - \frac{\mu k (k+1)}{4}\|\vy_{k+1} -  \hat\vy\|^2
   + \frac{16\sqrt{2\pi}\, d_yL^2}{\mu}\bigg].
\end{split}
\end{align}
Summing the above inequality over $k = 0, \ldots, K-1$, and taking expectation of both sides over $\vy_1,\cdots,\vy_{K-1}$, we obtain
\[
\sum_{t=0}^{K-1} k\big(\E[h_\nu(\vy_k)] - h_\nu( \hat\vy)\big)
\le \frac{16\sqrt{2\pi}\, d_yL^2K}{\mu},
\]
which leads to
\begin{align}\label{eq:sum-hk}
\begin{split}    
\sum_{k=0}^{K-1} \frac{2k}{K(K-1)} \E[h_\nu(\vy_k)]
\le h_\nu( \hat\vy) + \frac{32\sqrt{2\pi}\, d_yL^2}{\mu(K-1)}.
\end{split}
\end{align}
Therefore, taking
\begin{align}\label{y_out and K}
\vy_{\text{out}}=\sum_{k=0}^{K-1}\frac{2k\vy_k}{K(K-1)}
\quad \text{and} \quad
K=\left\lceil\frac{64\sqrt{2\pi}d_yL^2}{\mu\tilde\epsilon}\right\rceil, 
\end{align}
{and taking $\hat{\vy}=\vy_\nu^*\triangleq\argmin_{\vy\in\fY} h_\nu(\vy)$, it holds}
\[
\E[h_\nu(\vy_{\text{out}})]\le\sum_{k=0}^{K-1} \frac{2k}{K(K-1)} \E[h_\nu(\vy_k)]
\le h_\nu( \vy_\nu^*) + \tilde\epsilon,
\]
where the first inequality is based on the convexity of $h_\nu$ and the last step is based on equation~\eqref{eq:sum-hk} and the setting of $K$.

Now we prove the result for the objective $h(\cdot)$.
By Proposition 2.2 of \citet{lin2022gradient}, it holds \(|h-h_\nu|\le\nu L\). 
{By taking $\hat{\vy}=\vy^*\triangleq\argmin_{\vy\in\fY} h(\vy)$, equation (\ref{ineq: iterate of h_nu}) implies}
\begin{align*}
\begin{split}
&k\big(h(\vy_k) - h( \vy^*)\big)
\le k\big(h_\nu(\vy_k) - h_\nu( \vy^*)\big)+2k\nu L\\
\le& \E\bigg[\frac{\mu(k-1)k}{4}\|\vy_k -  \vy^*\|^2
   - \frac{\mu k (k+1)}{4}\|\vy_{k+1} -  \vy^*\|^2
   + \frac{16\sqrt{2\pi}\, d_yL^2}{\mu}+2k\nu L\bigg].
\end{split}
\end{align*}
Summing the above inequality over $k = 0, \ldots, K-1$, and taking expectation of both sides, we obtain
\[
\sum_{t=0}^{K-1} k\big(\E[h(\vy_k)] - h( \vy^*)\big)
\le \frac{16\sqrt{2\pi}\, d_yL^2K}{\mu}
+K(K-1)\nu L,
\]
which leads to
\begin{align}\label{eq:sum-h}
\begin{split}  
\sum_{k=0}^{K-1} \frac{2k}{K(K-1)} \E[h(\vy_k)]
\le h( \vy^*) + \frac{32\sqrt{2\pi}\, d_yL^2}{\mu(K-1)} +2\nu L.
\end{split}
\end{align}
We follow the settings of $\vy_{\text{out}}$ and $K$ in equation (\ref{y_out and K}) and further let $\nu=\tilde\epsilon/(4L)$, which leads to
\[
\E[h(\vy_{\text{out}})]\le\sum_{k=0}^{K-1} \frac{2k}{K(K-1)} \E[h(\vy_k)]
\le h( \vy^*) + \tilde\epsilon.
\]
where the first inequality is based on the convexity of $h$ and the last step is based on equation~\eqref{eq:sum-h} and the setting of $K$ and $\nu$.
\end{proof}

\subsection{Proof of Lemma \ref{update}}\label{Proof of Proposition of f}

Following Propositions \ref{prop:smoothing} and \ref{prop:smooth-Psi-delta}, we denote the smoothness parameters of $f_\delta$ and $\Psi_\delta$ as
\begin{align}\label{Notation: smoothness}
    M_{f_\delta} =\frac{c \sqrt{d}L}{\delta}
    \qquad \text{and} \qquad 
    M_{\Psi_\delta}=(1+\kappa)M_{f_\delta}
\end{align}
for some constant $c>0$, where  $\kappa={M_{f_\delta}}/{\mu}$ is the condition number of $f_\delta$.

We also denote
\begin{align}\label{notation: simplify of the gradient mapping}
    \tilde{\vu}_t=G_x(\vx_t, \vy_t,\vu_t,\eta_x) = \frac{\vx_t-\vx_{t+1}}{\eta_x},
\quad
\tilde{\vv}_t=G_y(\vx_t, \vy_t,\vv_t,\eta_y) = \frac{\vy_{t+1} - \vy_t}{\eta_y}.
\end{align}
and
\begin{align}\label{notation: full z and g}
    \vz_t=
\begin{bmatrix}
    \vx_t\\
    \vy_t
\end{bmatrix},
\quad
\vg_t=
\begin{bmatrix}
    \vu_t\\
    \vv_t
\end{bmatrix}.
\end{align}
We then provide the following lemmas.

\begin{lemma}\label{lemma:C.1}
    Following the settings of Lemma \ref{update} and notations (\ref{Notation: smoothness})--(\ref{notation: full z and g}),  it holds
\begin{align}\label{eq:varphi}
\begin{split}
\E[\Psi_\delta(\vx_{t+1})]
\le &
\mathbb{E}  \bigg[\Psi_{\delta}(  \vx_t) -\frac{\eta_x}{8}\Vert G(\vx_t,\nabla\Psi_\delta(\vx_t),\eta_x) \Vert^2-\eta_x\langle \vu_t, \tilde{\vu}_t\rangle \\
& \quad + \left(\frac{\eta_x}{2}+\frac{M_{\Psi_\delta}  \eta_x^2 }{2}\right) \Vert   \tilde{\vu}_t\Vert^2
+\frac{3\eta_x}{2} \|\vu_t-\nabla_x f_\delta(\vx_t,\vy_t)\|^2 \\
& \quad + 12\kappa^2\eta_x\|\tilde{\vv}_t\|^2 + 12\kappa^2\eta_x\|\nabla_y f_\delta(\vx_t,\vy_t)-\vv_t\|^2\bigg].
\end{split}
\end{align}
\end{lemma}

\begin{proof}
According to Proposition \ref{prop:smooth-Psi-delta}, 
the $M_{\Psi_\delta}$-smoothness of $\Psi_\delta$ implies
\begin{align}\label{inter00}
\begin{split}
\quad \mathbb{E}[\Psi_{\delta}(  \vx_{t+1})] 
\le&
\mathbb{E} \left [ \Psi_{\delta}(  \vx_t) + \left\langle \nabla \Psi_{\delta}(  \vx_t),\,  \vx_{t+1} -  \vx_t \right\rangle + \frac{M_{\Psi_\delta} }{2} \Vert   \vx_{t+1} -   \vx_t \Vert^2 \right] \\
=& \mathbb{E} \left [\Psi_{\delta}(\vx_t) - \eta_x \left\langle \nabla \Psi_{\delta}(\vx_t), \tilde{\vu}_t \right\rangle + \frac{M_{\Psi_\delta} \eta_x^2 }{2} \Vert   \tilde{\vu}_t \Vert^2 \right],
\end{split}
\end{align}
where the last step is based on the definition of ${\tilde\vu}_t$ defined in equation (\ref{notation: simplify of the gradient mapping}). 

We split the expectation for inner product $-\left\langle \nabla \Psi_{\delta}(  \vx_t), \tilde{\vu}_t \right\rangle$ in equation \eqref{inter0} as follows
\begin{align}\label{inter0}
\begin{split}
\!\!\!& \E\left[-\left\langle \nabla \Psi_{\delta}(  \vx_t), \tilde{\vu}_t \right\rangle \right]\\
\!\!\!=& \E\left[ -  \left\langle \vu_t,\tilde{\vu}_t \right\rangle
 +\left\langle \vu_t-\nabla_x f_\delta(\vx_t,\vy_t), \tilde{\vu}_t \right\rangle  
+  \left\langle \nabla_x f_\delta(\vx_t,\vy_t)-\nabla\Psi_\delta(\vx_t), \tilde{\vu}_t \right\rangle
\right]\\
\!\!\!=&
\E\Big[ - \left\langle \vu_t,\tilde{\vu}_t \right\rangle
 + \underbrace{\left\langle \vu_t-\nabla_x f_\delta(\vx_t,\vy_t), \tilde{\vu}_t -G_x(\vx_t,\vy_t,\nabla_x f_\delta(\vx_t,\vy_t),\eta_x)\right\rangle}_{A_1}  \\
\!\!\!&\quad\! +  \underbrace{\left\langle \vu_t-\nabla_x f_\delta(\vx_t,\vy_t), G_x(\vx_t,\vy_t,\nabla_x f_\delta(\vx_t,\vy_t),\eta_x) \right\rangle}_{A_2} 
 + \underbrace{\left\langle \nabla_x f_\delta(\vx_t,\vy_t)-\nabla\Psi_\delta(\vx_t), \tilde{\vu}_t \right\rangle}_{A_3}\!\Big].
\end{split}
\end{align}
For the term $A_1$, it holds
\begin{align}\label{ineq: A_1}
\begin{split}
    A_1
=& \left\langle \vu_t-\nabla_x f_\delta(\vx_t,\vy_t), \tilde{\vu}_t -G_x(\vx_t,\vy_t,\nabla_x f_\delta(\vx_t,\vy_t),\eta_x)
\right\rangle\\
=&\left\langle \vu_t-\nabla_x f_\delta(\vx_t,\vy_t), G_x(\vx_t,\vy_t,{\vu}_t,\eta_x) -G_x(\vx_t,\vy_t,\nabla_x f_\delta(\vx_t,\vy_t),\eta_x)
\right\rangle \\
\leq & \|\vu_t-\nabla_x f_\delta(\vx_t,\vy_t)\|\cdot
\|G_x(\vx_t,\vy_t,{\vu}_t,\eta_x) -G_x(\vx_t,\vy_t,\nabla_x f_\delta(\vx_t,\vy_t),\eta_x)\|\\
\le & \|\vu_t-\nabla_x f_\delta(\vx_t,\vy_t)\|^2,
\end{split}
\end{align}
where the first inequality is obtained by Cauchy--Schwarz inequality and the second inequality is based on Lemma $\ref{lemma: g&G_1}$.

For the term $A_2$, the definition of $\vu_t$ and Proposition \ref{prop:g-unbiased-variance} implies
\begin{align}\label{eq: A_2}
    \E[A_2]=\langle\mathbb{E}[\vu_t-\nabla_x f_\delta(\vx_t,\vy_t)],G_x(\vx_t,\vy_t,\nabla_x f_\delta(\vx_t,\vy_t),\eta_x)\rangle =0.
\end{align}

For the term $A_3$, it holds
\begin{align}\label{ineq: A_3}
\begin{split}
A_3 ={} & \langle\nabla_x f_\delta(\vx_t,\vy_t)-\nabla\Psi_\delta(\vx_t),   \tilde{\vu}_t\rangle\\
\le{} & \|\nabla_x f_\delta(\vx_t,\vy_t) - \nabla \Psi_\delta(\vx_t)\|^2+\frac{1}{4}\|\tilde{\vu}_t\|^2\\
\le{} & 4\kappa^2 \|G_y(\vx_t,\vy_t, \nabla_y    f_\delta(\vx_t,\vy_t),\eta_y)\|^2
    +\frac{1}{4}\|\tilde{\vu}_t\|^2\\
\le{} & 8\kappa^2 \|G_y(\vx_t,\vy_t,\nabla_y f_\delta(\vx_t,\vy_t),\eta_y)-G_y(\vx_t,\vy_t,\vv_t,\eta_y)\|^2 \\
    & +8\kappa^2 \|G_y(\vx_t,\vy_t,\vv_t,\eta_y)\|^2
    +\frac{1}{4}\|\tilde{\vu}_t\|^2\\
\le{} & 8\kappa^2 \|\nabla_y f_\delta(\vx_t,\vy_t)-\vv_t\|^2 +8\kappa^2 \|\tilde{\vv}_t\|^2
    +\frac{1}{4}\|\tilde{\vu}_t\|^2,
\end{split}
\end{align}
where the first inequality is based on the fact $ \langle \va,\vb\rangle \leq \norm{\va}^2 + \frac{1}{4}\norm{\vb}^2$ for all $\va,\vb\in\BR^{d_x}$; the second inequality is based on Lemma \ref{lemma: A1} and definition $\kappa=M_{f_\delta}/\mu$; the third inequality is based on Young's inequality; and the last step is based on Lemma \ref{lemma: g&G_2} and equation \eqref{notation: simplify of the gradient mapping}. 

Combining equations (\ref{inter0}), (\ref{ineq: A_1}), (\ref{eq: A_2}) and (\ref{ineq: A_3}), we obtain
\begin{align}\label{inter1}
\begin{split}
\!\!\!& \E\left[-\left\langle \nabla \Psi_{\delta}(  \vx_t), \tilde{\vu}_t \right\rangle \right]\\
\!\!\!=& \E\!\left[-\left\langle \vu_t,\tilde{\vu}_t \right\rangle
+\|\vu_t-\nabla_x f_\delta(\vx_t,\vy_t)\|^2\!+8\kappa^2 \|\nabla_y f_\delta(\vx_t,\vy_t)-\vv_t\|^2\!+8\kappa^2 \|\tilde{\vv}_t\|^2\!+\frac{1}{4}\|\tilde{\vu}_t\|^2\right]\!.\!
\end{split}
\end{align}
Moreover, we can write the last term in equation \eqref{inter1} as
\begin{align}\label{eq:u4121}
\frac{1}{4}\|\tilde{\vu}_t\|^2=\frac{1}{2}\|\tilde{\vu}_t\|^2 -\frac{1}{4}\|\tilde{\vu}_t\|^2   
\end{align}
and provide an upper bound for $-\frac{1}{4}\|\tilde{\vu}_t\|^2$ as follows
\begin{align}\label{the internal ineq}
\begin{split}
    \!\! &-\frac{1}{4}\|\tilde{\vu}_t\|^2=-\frac{1}{4}\Vert   G_x(\vx_t,\vy_t,\vu_t,\eta_x) \Vert^2\\
    \!\! \le&
    \frac{1}{4}\Vert   G_x(\vx_t,\vy_t,\vu_t,\eta_x)-G(\vx_t,\nabla\Psi_\delta(\vx_t),\eta_x) \Vert^2
    -\frac{1}{8}\Vert G(\vx_t,\nabla\Psi_\delta(\vx_t),\eta_x) \Vert^2
    \\
   \!\! \le&
   \frac{1}{4} \Vert   \vu_t-\nabla\Psi_\delta(\vx_t) \Vert^2
    -\frac{1}{8}\Vert G(\vx_t,\nabla\Psi_\delta(\vx_t),\eta_x) \Vert^2
    \\
    \!\! \le&
    \frac{1}{2}\Vert   \vu_t-\nabla_x f_\delta(\vx_t,\vy_t) \Vert^2
    +\frac{1}{2}\Vert \nabla\Psi_\delta(\vx_t)-\nabla_x f_\delta(\vx_t,\vy_t) \Vert^2
    -\frac{1}{8}\Vert G(\vx_t,\nabla\Psi_\delta(\vx_t),\eta_x) \Vert^2
    \\
    \!\!\le
    &
    \frac{1}{2}\Vert   \vu_t-\nabla_x f_\delta(\vx_t,\vy_t) \Vert^2
    +2\kappa^2 \| G_y(\vx_t, \vy_t,\nabla_y f_\delta(\vx_t,\vy_t),\eta_y) \|^2
    -\frac{1}{8}\Vert G(\vx_t,\nabla\Psi_\delta(\vx_t),\eta_x) \Vert^2\\
    \!\! \le
    &
    \frac{1}{2}\Vert   \vu_t-\nabla_x f_\delta(\vx_t,\vy_t) \Vert^2
    +4\kappa^2 \| G_y(\vx_t, \vy_t,\nabla_y f_\delta(\vx_t,\vy_t),\eta_y)-\tilde \vv_t \|^2 
    +4\kappa^2 \| \tilde \vv_t \|^2\\
    \!\! & -\frac{1}{8}\Vert G(\vx_t,\nabla\Psi_\delta(\vx_t),\eta_x) \Vert^2\\
    \!\! \le &
    \frac{1}{2}\Vert   \vu_t-\nabla_x f_\delta(\vx_t,\vy_t) \Vert^2 \!
    +4\kappa^2 \| \nabla_y f_\delta(\vx_t,\vy_t)- \vv_t \|^2\! 
    +4\kappa^2 \| \tilde \vv_t \|^2 \!
     -\frac{1}{8}\Vert G(\vx_t,\nabla\Psi_\delta(\vx_t),\eta_x) \Vert^2,\!\!\!\\
    \end{split}
\end{align}
where the second and the last inequality is based on Lemma \ref{lemma: g&G_2}; 
the fourth inequality is based on Lemma~\ref{lemma: A1}; 
and the other inequalities are based on Young's inequality. 

Combining equations (\ref{inter1}), (\ref{eq:u4121}) and (\ref{the internal ineq}), we obtain
\begin{align*}
& \E\left[-\left\langle \nabla \Psi_{\delta}(  \vx_t), \tilde{\vu}_t \right\rangle \right]\\
\le& \E \bigg[-\left\langle \vu_t,\tilde{\vu}_t \right\rangle +\|\vu_t-\nabla_x f_\delta(\vx_t,\vy_t)\|^2\!+8\kappa^2 \|\nabla_y f_\delta(\vx_t,\vy_t)-\vv_t\|^2\!+8\kappa^2 \|\tilde{\vv}_t\|^2 + \frac{1}{2}\|\tilde{\vu}_t\|^2  \\
&+\frac{1}{2}\Vert   \vu_t-\nabla_x f_\delta(\vx_t,\vy_t) \Vert^2
    +4\kappa^2 \| \nabla_y f_\delta(\vx,\vy)- \vv_t \|^2\! 
    +4\kappa^2 \| \tilde \vv_t \|^2 \!
     -\frac{1}{8}\Vert G(\vx_t,\nabla\Psi_\delta(\vx_t),\eta_x) \Vert^2
\bigg]\\
=& \E \bigg[-\left\langle \vu_t,\tilde{\vu}_t \right\rangle 
+\frac{3}{2} \|\vu_t-\nabla_x f_\delta(\vx_t,\vy_t)\|^2+12\kappa^2 \|\nabla_y f_\delta(\vx_t,\vy_t)-\vv_t\|^2  \\
& + \frac{1}{2}\|\tilde{\vu}_t\|^2 + 12\kappa^2 \|\tilde{\vv}_t\|^2
-\frac{1}{8}\Vert G(\vx_t,\nabla\Psi_\delta(\vx_t),\eta_x) \Vert^2
\bigg].
\end{align*}
By substituting above result into equation \eqref{inter00}, we finish the proof.

\end{proof}

\begin{lemma}\label{lemma:C.2}
Following the settings of Lemma \ref{update} and notations (\ref{Notation: smoothness})--(\ref{notation: full z and g}),  it holds
    \begin{align}\label{fxy}
    \begin{split}
        &\E[-f_\delta(\vx_{t+1},\vy_{t+1})]
        \le \mathbb{E} \bigg[ - f_{\delta}(  \vx_t,  \vy_t) +\eta_x\langle\vu_t,\tilde{\vu}_t\rangle     \\
&\qquad  \qquad
+\left( \frac{\eta_x^2 M_{f_\delta}}{2} + \eta_x^2 \eta_y M_{f_\delta}^2 + \frac{ \eta_x^2}{2\eta_y} \right) \Vert   \tilde{\vu}_t \Vert^2  +  \frac{\eta_y}{2} \Vert \nabla_x f_{\delta}(\vx_t,\mathbf{y}_t ) -   \vu_t \Vert^2 \\
&\qquad \qquad
- \left(\frac{\eta_y}{4} - \frac{\eta_y^2 M_{f_\delta}}{2}\right)\Vert \tilde\vv_{t}\Vert^2
+ \frac{\eta_y}{2} \Vert\nabla_y f_{\delta}(  \vx_t, \mathbf{y}_t )- \mathbf{v}_t  \Vert^2 \bigg].
    \end{split}
    \end{align}
\end{lemma}

\begin{proof}
According to Proposition \ref{prop:smoothing}, the $M_{f_\delta}$-smoothness of $f_{\delta}(\vx,\vy)$ implies
\begin{align}\label{inter5}
\begin{split}
& \BE \left[- f_{\delta}(  \vx_{t+1}, \mathbf{y}_{t+1} )\right]\\
\leq &
\mathbb{E}\left[ -f_{\delta}(  \vx_{t+1}, \mathbf{y}_t ) 
- \left\langle \nabla_y f_{\delta}(  \vx_{t+1},\mathbf{y}_t ),\, \mathbf{y}_{t+1}  - \mathbf{y}_t \right\rangle
+ \frac{M_{f_\delta}}{2} \Vert \mathbf{y}_{t+1}  -  \mathbf{y}_t  \Vert^2 \right] \\
\leq& 
\mathbb{E} \bigg[ - f_{\delta}(  \vx_t,\mathbf{y}_t ) 
-  \left\langle \nabla_x f_{\delta}(  \vx_t,\mathbf{y}_t ),\,  \vx_{t+1} -   \vx_t \right\rangle
+ \frac{M_{f_\delta}}{2} \Vert   \vx_{t+1} -   \vx_t \Vert^2 \\
&\quad 
- \left\langle \nabla_y f_{\delta}(  \vx_{t+1},\mathbf{y}_t ),\, \mathbf{y}_{t+1}  - \mathbf{y}_t \right\rangle
+ \frac{M_{f_\delta}}{2} \Vert \mathbf{y}_{t+1}  - \mathbf{y}_t  \Vert^2\bigg] \\
=& \mathbb{E} \bigg[- f_{\delta}(  \vx_t,  \vy_t)
+  \eta_x\left\langle \nabla_x f_{\delta}(  \vx_t,\mathbf{y}_t ),\,  \tilde{\vu}_t \right\rangle
+ \frac{\eta_x^2 M_{f_\delta}}{2} \Vert    \tilde{\vu}_t  \Vert^2  \\&\quad\quad
-\eta_y \left\langle \nabla_y f_{\delta}(  \vx_t, \mathbf{y}_t ),  \tilde{\vv}_t  \right\rangle
+ \frac{\eta_y^2 M_{f_\delta}}{2} \Vert  \tilde{\vv}_t   \Vert^2
\bigg],
\end{split}
\end{align}
where the last step follows equation \eqref{notation: simplify of the gradient mapping}.
For the term $\eta_x\left\langle \nabla_x f_{\delta}(  \vx_t,\mathbf{y}_t ),\,  \tilde{\vu}_t \right\rangle$ in equation~\eqref{inter5}, we provide its upper bound as follows
\begin{align}
\label{inter6}
\begin{split}
     &\eta_x\langle\nabla_x f_{\delta}(  \vx_t,\mathbf{y}_t ), \tilde{\vu}_t\rangle \\
     =&\eta_x\langle\nabla_x f_{\delta}(  \vx_t,\mathbf{y}_t)-\vu_t,\tilde{\vu}_t\rangle+\eta_x\langle\vu_t,\tilde{\vu}_t\rangle\\
     \le&\frac{\eta_y}{2}\|\nabla_x f_{\delta}(  \vx_t,\mathbf{y}_t )-\vu_t\|^2+\frac{\eta_x^2}{2\eta_y}\|\tilde{\vu}_t\|^2+\eta_x\langle\vu_t,\tilde{\vu}_t\rangle,
\end{split}
\end{align}
where the inequality is based on $ \eta_x\langle \va,\vb\rangle \leq \eta_y\norm{\va}^2/2 + \eta_x^2\norm{\vb}^2/(2\eta_y)$ for all $\va,\vb\in\BR^{d_x}$. 

For term $-\eta_y \left\langle \nabla_y f_{\delta}(  \vx_t, \mathbf{y}_t ),  \tilde{\vv}_t  \right\rangle$ in equation (\ref{inter5}), we split it into
\begin{align}\label{inter7}
\begin{split}
&-\eta_y \left\langle \nabla_y f_{\delta}(  \vx_t, \mathbf{y}_t ),  \tilde{\vv}_t  \right\rangle\\
=&
\underbrace{-\eta_y \left\langle \nabla_y f_{\delta}(  \vx_t, \mathbf{y}_t )- \mathbf{v}_t ,\,   \tilde{\vv}_t  \right\rangle
-\eta_y \left\langle \mathbf{v}_t,\,  \tilde{\vv}_t  \right\rangle}_{B_1}
+   \underbrace{\eta_y \left\langle \nabla_y f_{\delta}(  \vx_t, \mathbf{y}_t )- \nabla_y f_{\delta}(  \vx_{t+1},\mathbf{y}_t ),
 \tilde{\vv}_t \right\rangle}_{B_2}.
 \end{split}
\end{align}
For the term $B_1$, it holds
\begin{align}\label{inter8}
\begin{split}
B_1 & = -\eta_y \left\langle \nabla_y f_{\delta}(  \vx_t, \mathbf{y}_t )- \mathbf{v}_t ,  \tilde{\vv}_t  \right\rangle
-\eta_y \left\langle \mathbf{v}_t,\,  \tilde{\vv}_t  \right\rangle\\
& \le \frac{\eta_y}{2}\|\nabla_y f_{\delta}(  \vx_t, \mathbf{y}_t )- \mathbf{v}_t\|^2+
\frac{\eta_y}{2}\|\tilde{\vv}_t\|^2 -\eta_y\|\tilde{\vv}_t\|^2\\
& = \frac{\eta_y}{2}\|\nabla_y f_{\delta}(  \vx_t, \mathbf{y}_t )- \mathbf{v}_t\|^2-
\frac{\eta_y}{2}\|\tilde{\vv}_t\|^2, 
\end{split}
\end{align}
where the first inequality is based on the fact $ -\langle \va,\vb\rangle \leq \norm{\va}^2/2 + \norm{\vb}^2/2$ for all $\va,\vb\in\BR^{d_x}$, Lemma~\ref{lemma: g&G_2}, and the definition of $\tilde{\vv}_t$ in equation \eqref{notation: simplify of the gradient mapping}.
For the term $B_2$, it holds
\begin{align}\label{inter9}
\begin{split}
 B_2 & =  \eta_y\langle\nabla_y f_{\delta}( \vx_t,  \vy_t )- \nabla_y f_{\delta}( \vx_{t+1}, \vy_t )),  \tilde{\vv}_t\rangle\\
 & \le 
   \eta_y \Vert \nabla_y f_{\delta}( \vx_{t+1},  \vy_t ) - \nabla_y f_{\delta}( \vx_t, \vy_t ) \Vert^2 +  \frac{\eta_y}{4} \Vert  \tilde{\vv}_t  \Vert^2 \\
& \le    \eta_x^2 \eta_y M_{f_\delta}^2 \Vert  \tilde{\vu}_t \Vert^2 +  \frac{\eta_y}{4} \Vert \tilde{\vv}_t  \Vert^2,
\end{split}
\end{align}
where the last inequality is based on the $M_{f_\delta}$-smoothness of $f_\delta$ and the definition of $\tilde{\vu}_t$ in equation~\eqref{notation: simplify of the gradient mapping}.

Combining equations (\ref{inter6}), (\ref{inter7}) and (\ref{inter8}), it holds
\begin{align}\label{inter10}
\begin{split}
&-\eta_y \left\langle \nabla_y f_{\delta}(  \vx_t, \mathbf{y}_t ),  \tilde{\vv}_t  \right\rangle\le
\frac{\eta_y}{2}\|\nabla_y f_{\delta}(  \vx_t, \mathbf{y}_t )- \mathbf{v}_t\|^2-
\frac{\eta_y}{4} \Vert \tilde{\vv}_t  \Vert^2 + \eta_x^2 \eta_y M_{f_\delta}^2 \Vert  \tilde{\vu}_t \Vert^2.
\end{split}
\end{align}
By substituting equations (\ref{inter6}) and (\ref{inter10}) back into equation (\ref{inter5}),  we finish the proof.
\end{proof}

\begin{lemma}\label{lemma: C.4 SVRG accelerate}
Following the settings of Lemma \ref{update} and notations (\ref{Notation: smoothness})--(\ref{notation: full z and g}),  it holds
\begin{align}\label{ineq: variance reduction}
\!\!\mathbb{E} \left\| \vg_{t+1} - \nabla f_\delta(\vz_{t+1}) \right\|^2 
\leq \frac{ 16\sqrt{2\pi}pdL^2}{\tilde{b}} + (1-p) \mathbb{E} \| \vg_t - \nabla f_\delta(\vz_t) \|^2 +\frac{d^2L^2}{b\delta^2} \| \vz_{t+1}-\vz_t\|^2.
\end{align}
\end{lemma}
\begin{proof}
The update rule of $\vg_t=[\vu_t; \vv_t]$ in Algorithm \ref{alg:f} implies
\begin{align}\label{eq: vr-split with p}
\begin{split}
& \mathbb{E} \left[\left\| \vg_{t+1} - \nabla f_\delta(\vz_{t+1}) \right\|^2\right] \\
=& p\,\mathbb{E} \left[\left\| \vg(\vx_{t+1},\vy_{t+1};\tilde\fS_{t+1}) - \nabla f_\delta(\vz_{t+1}) \right\|^2\right] \\
&+  (1-p)\,\mathbb{E} \left[\left\| \vg_t + \vg(\vx_{t+1},\vy_{t+1};\fS_{t+1})-\vg(\vx_{t},\vy_{t};\fS_{t+1}) - \nabla f_\delta(\vz_{t+1}) \right\|^2\right].
\end{split}
\end{align}
According to Corollary \ref{Variance of the estimator}, the first term in the right side of (\ref{eq: vr-split with p}) could be bounded as
\begin{align}\label{inter p}
    p\,\mathbb{E} \left\| \vg(\vx_{t+1},\vy_{t+1};\tilde\fS_{t+1}) - \nabla f_\delta(\vx_{t+1},\vy_{t+1}) \right\|^2 \le\frac{ 16\sqrt{2\pi}pdL^2}{\tilde{b}}.
\end{align}
We then bound the second term in the right side of (\ref{eq: vr-split with p}) as
\begin{align}
&  \mathbb{E} \left[\left\| \vg_t -\nabla f_\delta(\vz_{t})+ \vg(\vx_{t+1},\vy_{t+1};\fS_{t+1})-\vg(\vx_{t},\vy_{t};\fS_{t+1}) - \nabla f_\delta(\vz_{t+1})+\nabla f_\delta(\vz_{t}) \right\|^2\right] \nonumber\\
=& \mathbb{E} \left[\| \vg(\vx_{t+1},\vy_{t+1};\fS_{t+1})-\vg(\vx_{t},\vy_{t};\fS_{t+1}) - \nabla f_\delta(\vz_{t+1})+\nabla f_\delta(\vz_{t}) \|^2\right] + \mathbb{E}\left[\| \vg_t -\nabla f_\delta(\vz_{t}) \|^2\right] \nonumber \\
 \le& \frac{1}{b} \mathbb{E} \left[\| \vg(\vx_{t+1},\vy_{t+1};\vw_{x,t+1},\vw_{y,t+1},\vxi_{t+1})-\vg(\vx_{t},\vy_{t};\vw_{x,t+1},\vw_{y,t+1},\vxi_{t+1}) \|^2\right]\nonumber\\
 &+ \mathbb{E} \left[\| \vg_t -\nabla f_\delta(\vz_{t}) \|^2\right]  \nonumber \\
\le& \frac{d^2L^2}{b\delta^2} \| \vz_{t+1}-\vz_t\|^2 +\mathbb{E} \| \vg_t - \nabla f_\delta(\vz_t) \|^2, \label{inter 1-p}
\end{align}
where the first equality is based on the martingale property according to Proposition 1 of \cite{fang2018spider};
the first inequality is based on Corollary \ref{Variance of the estimator}; the second inequality is based on applying Young's inequality and Assumption \ref{asm:Lipschitz} to obtain
\begin{align}\label{gx_1-gx_2}
\begin{split}
    &\mathbb{E}_{\vxi}\left[\|\vg(\vx_{t},\vy_{t};\vw_x,\vw_y,\vxi)-\vg(\vx_{t+1},\vy_{t+1};\vw_x,\vw_y,\vxi) \|^2\right]\\
    \le& \frac{d^2}{2\delta^2}\mathbb{E}_{\vxi}\left[\|F(\vx_t+\delta \vw_x,\vy_t+\delta \vw_y;\vxi)-F(\vx_{t+1}+\delta \vw_x,\vy_{t+1}+\delta \vw_y;\vxi)\|^2\right]\\
    & + \frac{d^2}{2\delta^2}\mathbb{E}_{\vxi}\left[\|F(\vx_t-\delta \vw_x,\vy_t-\delta \vw_y;\vxi)-F(\vx_{t+1}-\delta \vw_x,\vy_{t+1}-\delta \vw_y;\vxi)\|^2\right] \\
    \le & \frac{d^2 L^2}{\delta^2}\|\vz_t-\vz_{t+1}\|^2.
\end{split}
\end{align}
Combining equation (\ref{eq: vr-split with p}), (\ref{inter p}) and (\ref{inter 1-p}), we finish the proof.
\end{proof} 

Now we provide the proof of Lemma \ref{update} as follows.

\begin{proof}
According to Lemmas \ref{lemma:C.1} and \ref{lemma:C.2}, we add (\ref{eq:varphi}) multiplying $(1+\alpha)$ along with (\ref{fxy}) multiplying $\alpha$ to obtain
\begin{align}
& \BE[\Psi_{\delta}(\vx_{t+1})  + \alpha \left( \Psi_{\delta}(\vx_{t+1}) - f_{\delta}(\vx_{t+1}, \vy_{t+1})\right)] \nonumber\\
 \le&  \E[\Psi_{\delta}(\vx_t)  + \alpha \left( \Psi_{\delta}(\vx_{t}) - f_{\delta}(\vx_{t}, \vy_{t})\right)]
+ (1+\alpha) \mathbb{E}\left[  -\frac{\eta_x}{8}\Vert   G(\vx_t,\nabla\Psi_\delta(\vx_t),\eta_x) \Vert^2-\eta_x\vu_t^\top\tilde{\vu}_t\right] \nonumber\\
&+
(1+\alpha)\E\left[\left(\frac{\eta_x}{2}+\frac{M_{\Psi_\delta} \eta_x^2 }{2}\right) \Vert   \tilde{\vu}_t\Vert^2
 +\frac{3\eta_x}{2} \|\vu_t-\nabla_x f_\delta(\vx_t,\vy_t)\|^2
+12\kappa^2\eta_x\|\tilde{\vv}_t\|^2\right] \nonumber\\
 &+(1+\alpha)\E\left[12\kappa^2\eta_x\|\nabla_y f_\delta(\vx_t,\vy_t)-\vv_t\|^2\right] 
 +\alpha \E\left[ - \left(\frac{\eta_y}{4} - \frac{\eta_y^2 M_{f_\delta}}{2}\right) \Vert \tilde{\vv}_t   \Vert^2 +\eta_x \vu_t^\top\tilde{\vu}_t\right]  \nonumber\\
& + \alpha \E\left[ \left(\frac{\eta_x^2 M_{f_\delta}}{2} + \eta_x^2 \eta_y M_{f_\delta}^2 +\frac{ \eta_x^2}{2\eta_y} \right) \Vert   \tilde{\vu}_t \Vert^2 
+ \frac{\eta_y}{2} \Vert \nabla_x f_{\delta}(  \vx_t,\mathbf{y}_t ) -   \vu_t \Vert^2 + \frac{\eta_y}{2} \Vert\nabla_y f_{\delta}(  \vx_t, \mathbf{y}_t )- \mathbf{v}_t ) \Vert^2
\right] \nonumber\\
\le & \mathbb{E} \left[ \Psi_{\delta}(\vx_t)  + \alpha \left( \Psi_{\delta}(\vx_{t}) - f_{\delta}(\vx_{t}, \vy_{t})\right) - \frac{(1+\alpha)\eta_x}{8} \Vert G(\vx_t,\nabla\Psi_\delta(\vx_t),\eta_x)  \Vert^2\right] \nonumber\\
& + \left(\frac{3\eta_x(1+\alpha)}{2\eta_y} + \frac{\alpha}{2}\right) \eta_y \mathbb{E}\left[ \Vert   \vu_t  - \nabla_x f_{\delta}(  \vx_t, \mathbf{y}_t  ) \Vert^2\right] \nonumber\\
& + \left( \frac{12 \kappa^2 \eta_x(1+\alpha)}{\eta_y} + \frac{\alpha}{2} \right) \eta_y\mathbb{E}\left[ \Vert \mathbf{v}_t  - \nabla_y f_{\delta}(  \vx_t, \mathbf{y}_t  ) \Vert^2\right] \nonumber\\
& - \left(\eta_x- (1+\alpha) \left(\frac{\eta_x}{2} + \frac{M_{\Psi_\delta} \eta_x^2}{2} \right) - \alpha \left( \frac{M_{f_\delta}\eta_x^2 }{2} +\eta_x^2 \eta_y M_{f_\delta}   +\frac{ \eta_x^2}{2\eta_y} \right) \right) \mathbb{E}\left[  \Vert   \tilde{\vu}_t \Vert^2 \right] \nonumber\\
& - \left( \alpha \left( \frac{\eta_y}{4} - \frac{M_{f_\delta} \eta_y^2}{2} \right) -(1+\alpha) 12 \kappa^2 \eta_x\right) \mathbb{E}\left[   \Vert \tilde{\vv}_t  \Vert^2 \right], \label{eq: (1+alpha) Psi - f}
\end{align}
where the second inequality is obtained by the fact $\|\tilde{\vu}_t\|^2
\le\vu_t^\top\tilde{\vu}_t$ according to Lemma \ref{lemma: g&G_2} and rearranging the terms.
We take  
\[
\eta_x=\frac{\alpha}{768(1+\alpha)\kappa^2 M_{f_\delta}}
=\frac{\alpha\mu^2\delta^3}{768(1+\alpha)c^3d^{3/2}L^3}
\quad \text{and}\quad
\eta_y=\frac{1}{4 M_{f_\delta}}= \frac{\delta}{4c\sqrt{d}L},
\]
then above inequality 
and the definition of Lyapunov function $\mathcal{L}_t$ with $\alpha \in (0,1/4]$ implies
\begin{align*}
    & \frac{12(1+\alpha) \kappa^2 \eta_x}{\eta_y} + \frac{\alpha}{2} ,~~~
     \frac{3(1+\alpha) \eta_x}{2\eta_y}  + \frac{\alpha}{2} \le \alpha,~~~
    \alpha \left( \frac{\eta_y}{4} - \frac{M_{f_\delta} \eta_y^2}{2} \right) -(1+\alpha) 12\kappa^2 \eta_x =\frac{\alpha \eta_y}{16}, \\
    & \text{and} \quad \eta_x- (1+\alpha) \left(\frac{\eta_x}{2} + \frac{M_{\Psi_\delta} \eta_x^2}{2} \right) - \alpha \left( \frac{M_{f_\delta}\eta_x^2 }{2} +\eta_x^2 \eta_y M_{f_\delta}   +\frac{ \eta_x^2}{2\eta_y}\right) 
    \geq \frac{\eta_x}{4}.
\end{align*}

Combining above results with equation (\ref{eq: (1+alpha) Psi - f}), we achive
\begin{align}\label{eq: final (1+alpha) Psi - f}
\begin{split}
&\BE[\Psi_{\delta}(\vx_{t+1})  + \alpha \left( \Psi_{\delta}(\vx_{t+1}) - f_{\delta}(\vx_{t+1}, \vy_{t+1})\right)] 
 \\
\le& \E\Big[ \Psi_{\delta}(\vx_t)  + \alpha \left( \Psi_{\delta}(\vx_{t}) - f_{\delta}(\vx_{t}, \vy_{t})\right)  - \frac{\eta_x}{8} \Vert G(\vx_t,\nabla\Psi_\delta(\vx_t),\eta_x)  \Vert^2   \\
&\quad - \frac{\eta_x}{4} \Vert  \tilde{\vu}_t \Vert^2 - \frac{\alpha \eta_y}{16} \Vert \tilde{\vv}_t  \Vert^2 + \alpha \eta_y \Vert   \vu_t - \nabla_x f_{\delta}( \vx_t,  \vy_t  )\Vert^2 
+ \alpha \eta_y \Vert \vv_t  - \nabla_y f_{\delta}(\vx_t,  \vy_t ) \Vert^2  \Big]
\\
=& 
\BE \Big[ \Psi_{\delta}(\vx_t)+ \alpha \left( \Psi_{\delta}(\vx_{t}) - f_{\delta}(\vx_{t}, \vy_{t})\right) - \frac{\eta_x}{8} \Vert G(\vx_t,\nabla\Psi_\delta(\vx_t),\eta_x)  \Vert^2 \\
&\quad- \frac{1}{4\eta_x} \Vert  \vx_{t+1}-\vx_t \Vert^2 - \frac{\alpha }{16\eta_y} \Vert \vy_{t+1}-\vy_t  \Vert^2 
+\alpha \eta_y \Vert   \vg_t - \nabla f_{\delta}( \vx_t,  \vy_t  ) \Big]\\
\le& 
\BE \Big[ \Psi_{\delta}(\vx_t)+ \alpha \left( \Psi_{\delta}(\vx_{t}) - f_{\delta}(\vx_{t}, \vy_{t})\right) - \frac{\eta_x}{8} \Vert G(\vx_t,\nabla\Psi_\delta(\vx_t),\eta_x)  \Vert^2 \\  
&\quad -\frac{\alpha}{16\eta_y} \Vert \vz_{t+1}-\vz_t  \Vert^2 +\alpha \eta_y \Vert   \vg_t - \nabla f_{\delta}( \vx_t,  \vy_t  )\Vert^2  \Big],
\end{split}
\end{align}
where the last step is based on the settings $\eta_x/\eta_y=\alpha/(192(1+\alpha)\kappa^2)$ and $\alpha\in(0,1/4]$ that implies
$1/(4\eta_x) \geq \alpha/(16\eta_y)$.

According to Lemma \ref{lemma: C.4 SVRG accelerate}, we add equation (\ref{eq: final (1+alpha) Psi - f}) along with equation (\ref{ineq: variance reduction}) multiplying~$\alpha\eta_y p^{-1}$ to obtain
\begin{align*}
&\mathbb{E}\left[\mathcal{L}_{t+1}\right]\\
 \leq&  \mathbb{E}\left[ \mathcal{L}_t - \frac{\eta_x}{8} \Vert G(\vx_t,\nabla\Psi_\delta(\vx_t),\eta_x)  \Vert^2   
+\frac{ 16\sqrt{2\pi}\alpha \eta_y dL^2}{\tilde{b}} + \left(\frac{\alpha \eta_y d^2L^2}{pb\delta^2}-\frac{\alpha}{16\eta_y}\right) \| \vz_{t+1}-\vz_t\|^2  \right] \\
\leq&   \mathbb{E}\left[ \mathcal{L}_t - \frac{\eta_x}{8} \Vert G(\vx_t,\nabla\Psi_\delta(\vx_t),\eta_x)  \Vert^2 
+\frac{  16\sqrt{2\pi}\alpha \eta_y dL^2}{\tilde{b}}\right],
\end{align*}
where the last step is based on the settings $\eta_y={\delta}/{(4c\sqrt{d}L)}$ and $pb \ge d/c^2$ which guarantees 
\begin{align*}
\frac{\alpha \eta_y d^2L^2}{pb\delta^2}-\frac{\alpha}{16\eta_y} \leq 0.    
\end{align*}
Hence, we finish the proof.
\end{proof}

\subsection{Proof of Theorem \ref{zero-thm}}
\begin{proof}
According to Lemma \ref{update}, it holds
\begin{align*}
\BE \left[ \frac{\eta_x}{8} \Vert G(\vx_t,\nabla\Psi_\delta(\vx_t),\eta_x)  \Vert^2   \right]
\le & \BE \left[ \mathcal{L}_t \right]-\BE[\mathcal{L}_{t+1}]  +  \frac{16\sqrt{2\pi }\alpha \eta_y dL^{2}}{\tilde{b}}
\end{align*}  
Taking the average on above inequality over $t = 0,1\cdots,T-1$, we obtain
\begin{align} \label{ieq:res-bound}
\frac{1}{T} \sum_{t=0}^{T-1} \BE[\Vert G(\vx_t,\nabla\Psi_\delta(\vx_t),\eta_x)  \Vert^2] \leq \frac{8\E[\mathcal{L}_0-\mathcal{L}_T] }{\eta_x T}  +
  \frac{128\alpha\sqrt{2\pi }dL^{2}\eta_y}{\tilde{b}\eta_x}.
\end{align}

We bound the term $\mathcal{L}_0-\mathcal{L}_T$ as follows
\begin{align}\label{L0-LT}
\begin{split}
\mathcal{L}_0-\mathcal{L}_T 
= & \Psi_{\delta}(\vx_0) + \alpha \left( \Psi_{\delta}(\vx_{0}) - f_{\delta}(\vx_{0}, \vy_{0}) \right)+\frac{\alpha\eta_y}{p}\norm{\mathbf{g}_0-\nabla f_\delta(\mathbf{z}_0)}^2 \\
& -\left(\Psi_{\delta}(\vx_T) + \alpha \left( \Psi_{\delta}(\vx_{T}) - f_{\delta}(\vx_{T}, \vy_{T}) \right)+\frac{\alpha\eta_y}{p}\norm{\mathbf{g}_T-\nabla f_\delta(\mathbf{z}_T)}^2\right) \\
\leq & 
\underbrace{\Psi_{\delta}(\vx_0) - \Psi_{\delta}(\vx_T)}_{C_1} + 
\underbrace{\alpha \left( \Psi_{\delta}(\vx_{0}) - f_{\delta}(\vx_{0}, \vy_{0}) \right)}_{C_2}+
\underbrace{\frac{\alpha\eta_y}{p}\norm{\mathbf{g}_0-\nabla f_\delta(\mathbf{z}_0)}^2}_{C_3},
\end{split}
\end{align}
where the last step is based on the fact $\Psi_\delta(\vx_T) = \max_{\vy \in \mathcal{Y}} f_\delta(\vx_T, \vy) \geq f_\delta(\vx_T, \vy_T)$.

For the term $C_1$, the definition of $\Psi_\delta(\vx)$, it holds
\begin{align}\label{C1}
\!C_1=\Psi_\delta(\vx_0)-\Psi_\delta(\vx_T)\le\Psi_\delta(\vx_0)-\inf _{\vx\in\fX} \Psi_\delta(\vx)\le \Phi(\vx_0)-\inf_{\vx\in\fX} \Phi(\vx)+2\delta L=\Delta+2\delta L,
\end{align}
where the second inequality is based on Lemma \ref{distance between Phi Phi_delta}.

For the term $C_2$,
according to Lemma~\ref{thm:zeroth-order minimization} with
$h(\vy)=-f(\vx_0,\vy)$, $H(\vy;\xi)=-F(\vx_0,\vy;\xi)$, $\nu=\delta$ and $\tilde\epsilon=\alpha^{-1}$ and the setting
\begin{align*}
K_{\mathrm{in}}=\left\lceil\frac{64\sqrt{2\pi} d_y L^2}{\alpha^{-1}\mu}\right\rceil
\end{align*}
implies
\begin{align}\label{C2}
\BE[C_2]=\alpha\BE[\Psi_{\delta}(\vx_0) - f_{\delta}(\vx_0, \vy_0)] \le 1.
\end{align}
For the term $C_3$, parameter settings of $\eta_y$, $p$, and $ b$ implies
\begin{align}\label{C_3}
\E[C_3] = \E\left[\frac{\alpha\eta_y}{p}\norm{\mathbf{g}_0-\nabla f_\delta(\mathbf{z}_0)}^2\right]
\le \frac{16\sqrt{2\pi}\alpha\eta_y dL^2}{pb}
= \frac{4c\sqrt{2\pi}\alpha\delta L}{d^{\frac{1}{2}}}
\le c\sqrt{2\pi}\delta L.
\end{align}
Combining equations (\ref{L0-LT}), (\ref{C1}), (\ref{C2}) and (\ref{C_3}), it holds  
\begin{align}\label{ieq:bound-diff-H}    
\E[\mathcal{L}_0-\mathcal{L}_T]\le  \Delta+ \left(2+c\sqrt{2\pi}\right)\delta L+1.
\end{align}
According to equations \eqref{ieq:res-bound} and \eqref{ieq:bound-diff-H} and parameter settings of $\eta_x$ and $\eta_y$ with $\kappa=M_{f_\delta}\mu^{-1}$ and $M_{f_\delta} =c \sqrt{d}L\delta^{-1}$,
it holds
\begin{align}\label{eq:avg-G_x} 
\begin{split}
&\frac{1}{T} \sum_{t=0}^{T-1} \BE[\Vert G(\vx_t,\nabla\Psi_\delta(\vx_t),\eta_x)  \Vert^2]\\
\leq &\frac{6144(1+\alpha)c^3d^\frac{3}{2}L^3 \left(\Delta+(2+c\sqrt{2\pi})\delta L+1\right)}{\alpha \mu^2\delta^3 T} 
 +
  \frac{24576(1+\alpha)c^2\sqrt{2\pi }d^2L^4}{\tilde{b}\mu^2\delta^2}.
  \end{split}
\end{align}
We set
\[
T=\frac{49152c^3d^{3/2}L^3(1+\alpha)\left(\Delta+(2+c\sqrt{2\pi})\delta L+1\right)}{\alpha \mu^2\delta^3 \epsilon^2},~~ \tilde{b}=\frac{196608(1+\alpha)c^2\sqrt{2\pi }d^2L^4}{\mu^2\delta^2\epsilon^2},
\]
then equation \eqref{eq:avg-G_x} implies
\begin{align}\label{gradient control 1}
\BE\left[\Vert G(\vx_{\mathrm{out}},\nabla\Psi_\delta(\vx_{\mathrm{out}}),\eta_x)\|\right]  \le \frac{\epsilon}{2}.    
\end{align}
Recall that line \ref{line:haty} call Algorithm \ref{alg:minimization} with
$h(\vy)=-f(\vx_{\mathrm{out}},\vy)$, $H(\vy;\vxi)=-F(\vx_{\mathrm{out}},\vy;\vxi)$, $\nu=\delta$, and 
$K_{\mathrm{out}}=\lceil{512\sqrt{2\pi }c^3  d^{3/2}d_y L^5}/{(\mu^3\delta^3  \epsilon^2)}\rceil$  
to achieve point $\vy_{\mathrm{out}}$ for given $\vx_{\mathrm{out}}$.
According to Lemma \ref{thm:zeroth-order minimization} 
it holds
\begin{align}\label{eq:f_delta_out}
   \E[ f_\delta(\vx_{\mathrm{out}},\vy_\delta^*(\vx_{\mathrm{out}}))-f_\delta(\vx_{\mathrm{out}},\vy)]\le \frac{ \epsilon^2\mu^2}{8M_{f_\delta}^3}.
\end{align}
We set $\tilde\eta_y= M_{f_\delta}^{-1}=\delta/({c\sqrt{d}L})$, then Lemma \ref{Function value optimality controls the gradient mapping} implies
\begin{align}\label{gradient control 2}
\begin{split}
&\frac{M_{f_\delta}}{\mu}\E[\| G_y(\vx_{\mathrm{out}}, \vy_{\mathrm{out}},\nabla_y f_\delta(\vx_{\mathrm{out}},\vy_{\mathrm{out}}),\tilde\eta_y) 
\|]\\
\le &
\frac{M_{f_\delta}}{\mu}\sqrt{\frac{2}{\tilde\eta_y}\E\left[f_\delta(\vx_{\mathrm{out}},\vy_\delta^*(\vx_{\mathrm{out}}))-f_\delta(\vx_{\mathrm{out}},\vy_{\mathrm{out}})\right]}
\leq
\frac{\epsilon}{2},
\end{split}
\end{align}
where the last step is based on equation \eqref{eq:f_delta_out} and the parameter settings.
According to Lemma \ref{lemma: gradient mapping between Psi_delta and f_delta}, we obtain
\begin{align}\label{gradient control 3}
\begin{split}
&\E[\norm{G_x(\vx_{\mathrm{out}},\vy_{\mathrm{out}},\nabla_x f_\delta(\vx_{\mathrm{out}},\vy_{\mathrm{out}}),\eta_x)}]\\
\leq&
\E[\norm{G(\vx_{\mathrm{out}},\nabla \Psi_\delta(\vx_{\mathrm{out}}),\eta_x)}] + \frac{M_{f_\delta}}{\mu}\E[\| G_y(\vx_{\mathrm{out}}, \vy_{\mathrm{out}},\nabla_y f_\delta(\vx_{\mathrm{out}},\vy_{\mathrm{out}}),\tilde\eta_y)] \|\le\epsilon,
\end{split}
\end{align}
where the last step is based on equations (\ref{gradient control 1}) and (\ref{gradient control 2}). 
Based on equations \eqref{gradient control 2} and \eqref{gradient control 3} and Proposition \ref{prop:smoothing}(d), we know that the output $(\vx_{\mathrm{out}}, \vy_{\mathrm{out}})$ is a $(\eta_x,\tilde{\eta}_y,\delta,\epsilon)$-GSSP of $f(\vx,\vy)$ in expectation.

Now we consider the complexity for finding above GSSP.
Based on the setting $pb= d/c^2$, the expected number of zeroth-order oracle calls in each iteration is 
\begin{align*}
p\tilde{b}+(1-p)b\le 
\frac{196608(1+\alpha)\sqrt{2\pi }c^2d^2L^4p}{\mu^2\delta^2\epsilon^2}+\frac{d}{c^2p} 
\le 
\frac{800^2c^2d^2L^4p}{\mu^2\delta^2\epsilon^2}+\frac{d}{c^2p}
= \frac{800d^{{3}/{2}}L^2}{\delta\mu\epsilon},
\end{align*}
where the last step is based on the settings
\[\quad b=\frac{800d^{\frac{3}{2}}L^2}{\delta\mu\epsilon}
\qquad \text{and}\qquad p=\frac{\delta\mu\epsilon}{800c^2\sqrt{d}L^2},\]
which minimize the term ${800^2c^2d^2L^4p}/{(\mu^2\delta^2\epsilon^2)}+{d}/{(c^2p)}$ under the constraint $pb=d/c^2$.
Therefore, the overall zeroth-order oracle calls to achieve a $(\eta_x,\tilde{\eta}_y,\delta,\epsilon)$-GSSP is
\begin{align*}
K_{\mathrm{in}} + K_{\mathrm{out}}+ (p\tilde{b}+(1-p)b)T
= \fO\left(\frac{L^5 d^{3}(
 \Delta+\delta L+1)}{\mu^3 \delta^4 \epsilon^3}\right)
\end{align*}
in expectation, which finished the proof.

\end{proof}

\subsection{Proof of Proposition \ref{prop: the smoothed function of the regularized}}
\begin{proof}
The definition of $f_\delta(\vx,\vy)$ implies
\begin{align}\label{eq: compute the smoothed function of the regularized}
\begin{split}
\tilde f_\delta(\vx,\vy) 
& =  \BE\left[f(\vx+\delta\vu, \vy+\delta\vv)-\frac{\epsilon}{2D_y}\|\vy + \delta\vv -\vy_0\|^2\right] \\    
& = f_\delta(\vx,\vy)
-\frac{\epsilon}{2D_y}\mathbb{E}\big[\|\vy+\delta \vv-\vy_0\|^2\big]\\
& = f_\delta(\vx,\vy) - \frac{\epsilon}{2D_y}\|\vy-\vy_0\|^2
- \frac{\epsilon}{D_y}\delta \left\langle \vy-\vy_0,\mathbb{E}[\vv]\right\rangle - \frac{\delta^2\epsilon}{2D_y}\mathbb{E}\left[\|\vv\|^2\right] \\
&= f_\delta(\vx,\vy)-
\frac{\epsilon}{2D_y}\|\vy-\vy_0\|^2 - \frac{\delta^2\epsilon}{2D_y}\mathbb{E}\left[\|\vv\|^2\right],
    \end{split}
\end{align}
where the last step is based on the fact $\mathbb{E}[\vv]=\mathbf{0}$ since the distribution
$(\vu,\vv)\sim \mathrm{Unif}(\mathbb{B}^d(\vzero,1))$
means the random vector $\vv$ is symmetric about the origin. 
Here, we denote $d=d_x+d_y$.

Now it remains to compute $\mathbb{E}\left[\|\vv\|^2\right]$. 
We denote 
\begin{align*}
\vq\triangleq \begin{bmatrix}
\vu \\ \vv    
\end{bmatrix}
= [q_1, \dots, q_d]^\top, \text{~~~~~i.e.,~~~}
\vu=[q_1, \dots, q_{d_x}]^\top
\text{~~~and~~~}
\vv=[q_{d_x}+1, \dots, q_{d}]^\top.
\end{align*}
Since the distribution $\vq\sim \mathrm{Unif}(\mathbb{B}^d(\vzero,1))$ is rotational invariance, it holds
\begin{align}\label{eq:sym-q}
\mathbb{E}[q_1^2]=\mathbb{E}[q_2^2]=\cdots=\mathbb{E}[q_d^2].
\end{align}
Additionally, the distribution $\vq\sim \mathrm{Unif}(\mathbb{B}^d(\vzero,1))$ implies for all $r\in[0,1]$, it holds
\[
\mathbb{P}(\|\vq\|^2\le r^2)
= \mathbb{P}(\|\vq\|\le r) =
\frac{\mathrm{Vol}(\mathbb{B}^d(\vzero,r))}{\mathrm{Vol}(\mathbb{B}^d(\vzero,1))}
= r^d.
\]
This implies for all $t\in[0,1]$, it holds
\begin{align}\label{eq:prob-q}
\mathbb{P}(\|\vq\|^2\ge t)
= 1- t^{d/2},
\end{align}
Therefore, we have
\begin{align}\label{eq:Eq}
\mathbb{E}\left[\|\vq\|^2\right]
= \int_0^{+\infty} \mathbb{P}(\|\vq\|^2\ge t)\,\mathrm{d}t
\overset{\eqref{eq:prob-q}}= \int_0^1 (1-t^{{d}/{2}})\,\mathrm{d}t
= \frac{d}{d+2}.
\end{align}
Combing above results, we achieve
\begin{align}\label{expectation of square of v}
\mathbb{E}[\|\vv\|]^2
= \sum_{i=d_x+1}^{d}\mathbb{E}[q_i^2]
\overset{\eqref{eq:sym-q}}{=} \frac{d_y}{d} \mathbb{E}\left[\|\vq\|^2\right]
\overset{\eqref{eq:Eq}}{=} \frac{d_y}{d+2}.
\end{align}
Combining equations \eqref{eq: compute the smoothed function of the regularized} and \eqref{expectation of square of v}, we conclude that
\[
\tilde f_\delta(\vx,\vy)
=
f_\delta(\vx,\vy)
-\frac{\epsilon}{2D_y}\|\vy-\vy_0\|^2
-\frac{d_y\delta^2\epsilon}{2(d+2)D_y}.
\]
\end{proof}

\subsection{Proof of Corollary \ref{thm: complexity of f: concave}}

\begin{proof}
Recall the regularized problem (\ref{prob:regularized}) can be written as
\begin{align*}
    \min_{\vx\in\fX}\max_{\vy\in\fY}\tilde{f}(\vx,\vy) \triangleq 
    \E\left[\tilde F(\vx,\vy;\vxi)\right]
\end{align*}
where 
\begin{align*}    
\tilde F(\vx,\vy;\vxi) = F(\vx,\vy;\vxi)-\frac{\epsilon}{2D_y}\|\vy-\vy_0\|^2.
\end{align*}
We then verify the mean-square Lipschitz continuity of the stochastic component $\tilde F(\vx,\vy;\vxi)$.
In particular, for all $(\vx_1,\vy_1),(\vx_2,\vy_2)\in \fX\times \fY$, it holds
\[
\tilde{F}(\vx_1,\vy_1;\vxi)-\tilde{F}(\vx_2,\vy_2;\vxi)
=
F(\vx_1,\vy_1;\vxi)-F(\vx_2,\vy_2;\vxi)
-\frac{\epsilon}{2D_y}
\Big(\|\vy_1-\vy_0\|^2-\|\vy_2-\vy_0\|^2\Big),
\]
which implies
\begin{align}\label{eq:tilde-F-L}
\begin{split}
&\left(
\mathbb{E}\left[
\big|\tilde{F}(\vx_1,\vy_1;\vxi)-\tilde{F}(\vx_2,\vy_2;\vxi)\big|^2
\right]
\right)^{1/2} \\
\le&
\left(
\mathbb{E}\left[
\big|F(\vx_1,\vy_1;\vxi)-F(\vx_2,\vy_2;\vxi)\big|^2
\right]
\right)^{1/2}
+
\frac{\epsilon}{2D_y}
\left |
\|\vy_1-\vy_0\|^2-\|\vy_2-\vy_0\|^2
\right |\\
\le&
L\sqrt{\|\vx_1-\vx_2\|^2+\|\vy_1-\vy_2\|^2}
+
\frac{\epsilon}{2D_y}
\big(\|\vy_1-\vy_0\|+\|\vy_2-\vy_0\|\big)\|\vy_1-\vy_2\|\\
\le&
L\sqrt{\|\vx_1-\vx_2\|^2+\|\vy_1-\vy_2\|^2}
+\epsilon \|\vy_1-\vy_2\|\\
\le& (L+\epsilon)\sqrt{\|\vx_1-\vx_2\|^2+\|\vy_1-\vy_2\|^2}.
\end{split}
\end{align}
where the first inequality is based on Minkowski's inequality;
the second inequality is based on the mean-squared
Lipschitz continuity of $F(\vx,\vy;\vxi)$ and triangle inequality; 
the third inequality is based on the fact that $D_y$ is the diameter of $\fY$.

Thus, equation \eqref{eq:tilde-F-L} indicates that the stochastic component $\tilde{F}(\vx,\vy,\vxi)$ is mean-squared Lipschitz continuous with parameter
\begin{align*}
\tilde{L}=L+\epsilon.    
\end{align*}
In addition, the concavity of $f(\vx,\vy)$ with respect to $\vy$ implies that the function $\tilde{f}(\vx,\vy)$ is $\tilde \mu$-strongly concave in $\vy$ with parameter
\begin{align*}
\tilde{\mu}=\frac{\epsilon}{D_y}.    
\end{align*}
Therefore, we can apply PGFDA (Algorithm \ref{alg:f}) with the stochastic zeroth-order oracle $\tilde F(\vx,\vy;\vxi)$ to solve the regularized minimax problem \eqref{prob:regularized}.
We further follow the proof of Theorem~\ref{zero-thm} by replacing 
$F$, $L$, $\mu$, and $\epsilon$ with $\tilde F$, $\tilde L$, $\tilde \mu$ and $\epsilon/2$, then results of \eqref{gradient control 2} and \eqref{gradient control 3} 
indicates we can achieve $(\vx_{\rm out}, \vy_{\rm out})\in\fX\times\fY$ such that
\begin{align}\label{gradient control 2: regularize}
\begin{split}
&\frac{M_{f_\delta}}{\mu}\E[\| G_y(\vx_{\mathrm{out}}, \vy_{\mathrm{out}},\nabla_y \tilde f_\delta(\vx_{\mathrm{out}},\vy_{\mathrm{out}}),\tilde\eta_y) 
\|]
\leq
\frac{\epsilon}{2},
\end{split}
\end{align}
and
\begin{align}\label{gradient control 3: regularize}
\begin{split}
&\E[\norm{G_x(\vx_{\mathrm{out}},\vy_{\mathrm{out}},\nabla_x \tilde f_\delta(\vx_{\mathrm{out}},\vy_{\mathrm{out}}),\eta_x)}]\le\epsilon
\end{split}
\end{align}
with the overall stochastic zeroth-order oracle complexity of
\begin{align*}
\fO\left(\frac{\tilde L^5 d^{3}(
\Delta+\delta L+ 1)}{\tilde \mu^3 \delta^4 \epsilon^3}\right)=\fO\left(\frac{(L+\epsilon)^5 d^{3}D_y^3 (
\Delta+\delta L+ 1)}{ \delta^4 \epsilon^6}\right),   
\end{align*}
where
\begin{align*}    \eta_x=\Theta\left(\frac{\tilde\mu^2\delta^3}{d^{3/2}\tilde L^3}\right)=\Theta\left(\frac{\epsilon^2\delta^3}{d^{3/2}(L+\epsilon)^3D_y^2}\right),
\quad\tilde\eta_y=\Theta\left(\frac{\delta}{\sqrt{d}\tilde L}\right)=\Theta\left(\frac{\delta}{\sqrt{d} (L+\epsilon)}\right).
\end{align*}

Finally, we show that $(\vx_{\rm out}, \vy_{\rm out})\in\fX\times\fY$ is a desired $(\eta_x,\tilde\eta_y,\delta,\epsilon)$-GSSP of Problem \eqref{prob:constrained main}.
According to Lemma \ref{lemma: g&G_1}, it holds
\begin{align}\label{eq:diff-delta-tilde}
\!\!\! \|G_y(\vx,\vy, \nabla_y f_\delta(\vx,\vy), \eta_y) - G_y(\vx,\vy, \nabla_y \tilde{f_\delta}(\vx,\vy), \eta_y) \|  
\leq  \|\nabla_y f_\delta(\vx,\vy)-\nabla_y \tilde{f_\delta}(\vx,\vy)\|
\leq \frac{\epsilon}{2}
\end{align}
for all $\vx\in\fX$ and $\vy\in\fY$,
where the last step is based on taking gradient with respect to $\vy$ on 
$\tilde{f_\delta}(\vx,\vy)=f_\delta(\vx,\vy)-\epsilon\|\vy-\vy_0\|^2/({2D_y})-{\epsilon\delta^2 d_y}/({2(d+2)D_y})$ that leads to
\begin{align*}
    \|\nabla_y f_\delta(\vx,\vy)-\nabla_y \tilde{f_\delta}(\vx,\vy)\|
= \frac{\epsilon}{D_y}\|\vy-\vy_0\|    
\leq\frac{\epsilon}{2}.
\end{align*}
This further implies
\begin{align}\label{eq:tilde-f-y}
\begin{split}    
&\BE\left[\|G_y(\vx_{\mathrm{out}}, \vy_{\mathrm{out}}, \nabla_y f_\delta(\vx_{\mathrm{out}},\vy_{\mathrm{out}}), \eta_y)\|\right]\\
\le & \BE\left[\| G_y(\vx_{\mathrm{out}},\vy_{\mathrm{out}}, \nabla_y f_\delta(\vx_{\mathrm{out}},\vy_{\mathrm{out}}), \eta_y) - G_y(\vx_{\mathrm{out}},\vy_{\mathrm{out}}, \nabla_y \tilde{f_\delta}(\vx_{\mathrm{out}},\vy_{\mathrm{out}}), \eta_y) \|\right] \\
& + \BE\left[\|G_y(\vx_{\mathrm{out}},\vy_{\mathrm{out}}, \nabla_y \tilde{f_\delta}(\vx_{\mathrm{out}},\vy_{\mathrm{out}}), \eta_y) \|\right] \\
\le & \frac{\epsilon}{2} + \frac{\epsilon}{2} = \epsilon,
\end{split}
\end{align}
where the second inequality is based on equations \eqref{eq:diff-delta-tilde} and \eqref{gradient control 2: regularize}.
Additionally, the definition of $\tilde f_\delta(\vx,\vy)$ means $\nabla_x f_\delta(\vx_{\mathrm{out}},\vy_{\mathrm{out}})=\nabla \tilde f_\delta(\vx_{\mathrm{out}},\vy_{\mathrm{out}})$, which leads to
\begin{align}\label{eq:tilde-f-x}
\|G_x(\vx_{\mathrm{out}},\vy_{\mathrm{out}}, \nabla_x f_\delta(\vx_{\mathrm{out}},\vy_{\mathrm{out}}), \eta_x)\| = \|G_x(\vx_{\mathrm{out}},\vy_{\mathrm{out}}, \nabla_x \tilde{f_\delta}(\vx_{\mathrm{out}},\vy_{\mathrm{out}}), \eta_x)\| \overset{\eqref{gradient control 3: regularize}}\le \epsilon.    
\end{align}
Hence, combining results \eqref{eq:tilde-f-y} and \eqref{eq:tilde-f-x} 
with Theorem 3.1 of \citet{lin2022gradient} implies that
$(\vx_{\mathrm{out}},\vy_{\mathrm{out}})$ is a $(\eta_x,\eta_y,\delta,\epsilon)$-GSSP of Problem \eqref{prob:constrained main}, which finishes the proof.
\end{proof}

\section{Proof for Results in Section \ref{sec:convergence NL-PGFDA}}\label{appendix:Phi}

This section provides detailed proofs for finding approximate generalized Goldstein stationary points (GGSPs) of the primal function $\Phi(\vx)=\max_{\vy\in\fY}f(\vx,\vy)$.

\subsection{Proof for Theorem \ref{Theorem of Phi}}

\begin{proof} 
We first define the zeroth-order estimator of $\nabla \Phi_\delta(\vx_t)$ as follows
\begin{align*}
\vg^*(\vx_t;\vw_{t,i},\vxi_{t,i})\!=\!\frac{d_x}{2\delta}(F(\vx_t\!+\!\delta\vw_{t,i},\vy^*(\vx_t\!+\!\delta\vw_{t,i});\vxi_{t,i})\!-\!F(\vx_t\!-\!\delta\vw_{t,i},\vy^*(\vx\!-\!\delta\vw_{t,i});\vxi_{t,i}))\cdot\vw_{t,i}.
\end{align*}
where $\vw_{t,i}\sim{\rm Unif}(\mathbb{S}^{d_x-1})$ and $\vxi_{t,i}\sim\fD$.
The definition of $\vy^*(\cdot)$ in equation \eqref{eq:y-star} leads to
\begin{align*}
    \vy^*(\vx_t+\delta\vw_{t,i})=\argmax_{\vy\in\mathcal Y} f(\vx_t+\delta\vw_{t,i},\vy)
    \quad \text{and} \quad 
    \vy^*(\vx_t-\delta\vw_{t,i})=\argmax_{\vy\in\mathcal Y} f(\vx_t -\delta\vw_{t,i},\vy).
\end{align*}
This implies that the vector $\vg^*(\vx_t;\vw_{t,i},\vxi_{t,i})$ is an unbiased estimator of $\nabla \Phi_\delta(\vx_t)$, i.e.,
\begin{align*}
& \BE[\vg^*(\vx_t;\vw_{t,i},\vxi_{t,i})] \\
=& \E_{\vw_{t,i},\vxi_{t,i}}\left[\frac{d_x}{2\delta}(F(\vx_t+\delta\vw_{t,i},\vy^*(\vx_t+\delta\vw_{t,i});\vxi_{t,i})-F(\vx_t-\delta\vw_{t,i},\vy^*(\vx_t-\delta\vw_{t,i});\vxi_{t,i}))\cdot\vw_{t,i}\right]
\\
=& \E_{\vw_{t,i}}\left[\frac{d_x}{2\delta}(f(\vx_t+\delta\vw_{t,i},\vy^*(\vx_t+\delta\vw_{t,i}))-f(\vx_t-\delta\vw_{t,i},\vy^*(\vx_t-\delta\vw_{t,i})))\cdot\vw_{t,i}\right]\\
= & \E_{\vw_{t,i}}\left[\frac{d_x}{2\delta}(\Phi(\vx_t+\delta\vw_{t,i})-\Phi(\vx_t-\delta\vw_{t,i}))\cdot\vw_{t,i}\right]
=\nabla\Phi_\delta(\vx_t),
\end{align*}
where the last step follows Lemma D.1 of \cite{lin2022gradient} since Proposition \ref{Lipshitz of Phi} indicates the function $\Phi(\vx)$ is Lipschitz continuous.
Then lines \ref{line:nl-pgfda-w} and \ref{line:nl-pgfda-xi} of Algorithm~\ref{alg:Phi} result the mini-batch unbiased estimator of $\nabla\Phi_\delta(\vx_t)$ as follows
\[
    \vu^*_t \triangleq \frac{1}{b}\sum_{i=1}^b\vg^*(\vx_t;\vw_{t,i},\vxi_{t,i}).
\]
According to Corollary 1 of \cite{chen2023faster}, 
the unbiasedness of $\vg^*(\vx_t;\vw_{t,i},\vxi_{t,i})$ leads to
\begin{align}\label{Variance of u^*}
    \E\left[\|\vu_t^*-\nabla \Phi_\delta(\vx_t)\|^2\right]\le\frac{16\sqrt{2\pi}d_xL^2}{b}.
\end{align}

Since the exact $\vy^*(\cdot)$ is unavailable in general, we apply the subroutine Gradient-Free-Descent (Algorithm \ref{alg:minimization}) to establish $\vy_{t,i}^+\approx\vy^*(\vx_t+\delta\vw_{t,i})$ and $\vy_{t,i}^-\approx\vy^*(\vx_t-\delta\vw_{t,i})$.
This results the estimator of $\vu^*_t$ as follows
\[
    \hat\vu_t\triangleq\frac{1}{b}\sum_{i=1}^b\hat{\vg}(\vx_t,\vy_{t,i}^+,\vy_{t,i}^-;\vw_{t,i},\vxi_{t,i}),
\]
where 
\[
    \hat{\vg}(\vx_t,\vy_{t,i}^+,\vy_{t,i}^-;\vw_{t,i},\vxi_{t,i})=  \frac{d_x}{2\delta}\big(F(\vx_t+\delta\vw_{t,i},\vy_{t,i}^+;\vxi_{t,i}) - F(\vx_t-\delta\vw_{t,i},\vy_{t,i}^-;\vxi_{t,i})\big)\cdot\vw_{t,i}.
\]

We then consider the difference between $\vg^*(\vx_t;\vw_{t,i},\vxi_{t,i})$ and $\hat{\vg}(\vx_t,\vy_{t,i}^+,\vy_{t,i}^-;\vw_{t,i},\vxi_{t,i})$.
In particular, the triangle inequality leads to
\begin{align*}
&\|\vg^*(\vx_t;\vw_{t,i},\vxi_{t,i})-\hat{\vg}(\vx_t,\vy_{t,i}^+,\vy_{t,i}^-;\vw_{t,i},\vxi_{t,i})\| \\
\le & \frac{d_x}{2\delta}
\Bigg(
\underbrace{\left |F(\vx_t+\delta\vw_{t,i},\vy^*(\vx_t+\delta\vw_{t,i});\vxi_{t,i})-F(\vx_t+\delta\vw_{t,i},\vy_{t,i}^+;\vxi_{t,i})\right |}_{\Delta_t^+} 
\\
& ~\quad\quad + \underbrace{\left |(F(\vx_t-\delta\vw_{t,i},\vy^*(\vx_t-\delta\vw_{t,i});\vxi_{t,i}) -  F(\vx_t -\delta\vw_{t,i},\vy_{t,i}^-;\vxi_{t,i})\right|}_{\Delta_t^-}\Bigg).
\end{align*}
Consequently, it holds
\begin{align}\label{eq:diff-g-star-g-hat}
\begin{split}    
&\E_{\vxi_{t,i}} \left[\|\vg^*(\vx_t,\vw_{t,i};\vxi_{t,i}) - \hat{\vg}(\vx_t,\vy_{t,i}^+,\vy_{t,i}^-;\vw_{t,i},\vxi_{t,i})\|^2\right]\\
\le & \frac{d_x^2}{4\delta^2}
\E_{\vxi}\left[(\Delta_t^+ +\Delta_t^-)^2\right]
\le \frac{d_x^2}{2\delta^2}
\E_{\vxi_{t,i}}\left[(\Delta_t^+)^2 +(\Delta_t^-)^2\right] \\
\le & \frac{d_x^2L^2}{2\delta^2}\Big(\norm{\vy_{t,i}^+ - \vy^*(\vx_t+\delta\vw_{t,i})}^2+\norm{\vy_{t,i}^- - \vy^*(\vx_t-\delta\vw_{t,i})}^2 \Big) \\
\le & \frac{d_x^2L^2}{\mu \delta^2}\Big( f(\vx_t+\delta\vw_{t,i},\vy^*(\vx_t+\delta\vw_{t,i})) - f(\vx_t+\delta\vw_{t,i},\vy_{t,i}^+) \\
& \qquad\quad + f(\vx_t - \delta\vw_{t,i}, \vy^*(\vx_t - \delta\vw_{t,i})) - f(\vx_t - \delta\vw_{t,i}, \vy_{t,i}^-) \Big), 
\end{split}
\end{align}
where the third inequality is based on Assumption \ref{asm:Lipschitz} and the last inequality is based on the fact that the function $f(\vx,\cdot)$ is $\mu$-strongly concave on $\fY$ which leads to
\[
f(\vx,\vy^*(\vx)) - f(\vx,\vy) \ge \frac{\mu}{2}\|\vy-\vy^*(\vx)\|^2 
\]
for all $\vx\in\fX$ and $\vy\in\fY$.
According to lines \ref{line:y+} and \ref{line:y-} of Algorithm \ref{alg:Phi} and Lemma~\ref{thm:zeroth-order minimization} with
$h(\vy)=-f(\vx_t,\vy)$, $H(\vy;\xi)=-F(\vx_t,\vy;\xi)$, {$\nu=\tilde\delta\triangleq\tilde\epsilon/(4L)$} and the setting $K=\left\lceil{64\sqrt{2\pi}d_yL^2}/(\mu\tilde\epsilon)\right\rceil$ for some $\tilde\epsilon>0$,
it holds
\begin{align}\label{eq:diff-f-star-+}
f(\vx_t+\delta\vw_{t,i},\vy^*(\vx_t+\delta\vw_{t,i})) - f(\vx_t+\delta\vw_{t,i},\vy_{t,i}^+) \leq \tilde\epsilon
\end{align}
and
\begin{align}\label{eq:diff-f-star--}
f(\vx_t - \delta\vw_{t,i}, \vy^*(\vx_t - \delta\vw_{t,i})) - f(\vx_t - \delta\vw_{t,i}, \vy_{t,i}^-)  \leq \tilde\epsilon.
\end{align}
Combining equations \eqref{eq:diff-g-star-g-hat}, \eqref{eq:diff-f-star-+}, and \eqref{eq:diff-f-star--}, we achieve
\begin{align*}
\E_{\vxi_{t,i}}\left[\|\vg^*(\vx_t,\vw_{t,i};\vxi_{t,i})-\hat{\vg}(\vx_t,\vy_{t,i}^+,\vy_{t,i}^-;\vw_{t,i},\vxi_{t,i})\|^2\right]
        \le&  \frac{2d_x^2L^2 \tilde\epsilon}{\mu \delta^2}.
\end{align*}
Furthermore, the above inequality implies
\begin{align}\label{control of the difference between u, hat u}
\begin{split}    
\E\left[\|\vu_t^*-\hat\vu_t\|^2\right]
= & \BE\left[\left\|\frac{1}{b}\sum_{i=1}^b(\vg^*(\vx_t;\vw_{t,i},\vxi_{t,i})-\hat{\vg}(\vx_t,\vy_{t,i}^+,\vy_{t,i}^-;\vw_{t,i},\vxi_{t,i}))\right\|^2\right]  \\
\le & \E\left[\|\vg^*(\vx_t,\vw_{t,i};\vxi_{t,i})-\hat{\vg}(\vx_t,\vy_{t,i}^+,\vy_{t,i}^-;\vw_{t,i},\vxi_{t,i})\|^2\right] \\
\le & \frac{2d_x^2L^2 \tilde\epsilon}{\mu \delta^2}.
\end{split}
\end{align}

Now we analyze the convergence for Algorithm \ref{alg:Phi}. 
Following the notation in equation~\eqref{eq:mapping-smooth}, we denote
\begin{align*}
    G(\vx_t, \hat\vu_t,\eta) = \frac{\vx_t - \Pi_\mathcal{X}\left(\vx_t - \eta \hat\vu_t\right) }{\eta}.    
\end{align*}
For the smoothed surrogate function 
\begin{align*}
 \Phi_\delta(\vx) \triangleq \BE_{\vw\sim\BB^{d_x}(\vzero,1)}\left[\max_{\vy \in \BR^{d_y}} f(\vx+\delta\vw, \vy)\right],
\end{align*}
the update $\vx_{t+1}=\Pi_\mathcal{X} (\vx_t-\eta\hat\vu_t)$ implies
\begin{align}\label{inter eq 1}
\begin{split}
\E\left[\Phi_\delta(\vx_{t+1})\right]
\le& \E\left[\Phi_\delta(\vx_t) + \langle \nabla \Phi_\delta(\vx_t), \vx_{t+1}-\vx_t \rangle + \frac{cL\sqrt{d_x}}{2\delta}\|\vx_{t+1}-\vx_t\|^2\right] \\
=& \E\left[\Phi_\delta(\vx_t) - \eta \langle \nabla \Phi_\delta(\vx_t), G(\vx_t,\hat\vu_t,\eta)\rangle + \frac{c\eta^2 L\sqrt{d_x}}{2\delta}\|G(\vx_t,\hat\vu_t,\eta)\|^2 \right]\\
=& \E\bigg[\Phi_\delta(\vx_t) - \eta \langle \hat\vu_t, G(\vx_t,\hat\vu_t,\eta)\rangle+\eta \langle \hat\vu_t-\vu^*_t, G(\vx_t,\hat\vu_t,\eta)\rangle \\
& \quad+\eta\langle \vu^*_t-\nabla \Phi_\delta(\vx_t), G(\vx_t,\hat\vu_t,\eta)\rangle  + \frac{c\eta^2 L\sqrt{d_x}}{2\delta}\|G(\vx_t,\hat\vu_t,\eta)\|^2
\Bigg]\\
\le& \E\bigg[\Phi_\delta(\vx_t) - \eta \|G(\vx_t,\hat\vu_t,\eta)\|^2+\frac{\eta }{2}\| \hat\vu_t-\vu^*_t\|^2+\frac{\eta}{2}\|G(\vx_t,\hat\vu_t,\eta)\|^2  \\
&\quad+\eta\langle \vu^*_t-\nabla \Phi_\delta(\vx_t), G(\vx_t,\hat\vu_t,\eta)\rangle + \frac{c\eta^2 L\sqrt{d_x}}{2\delta}\|G(\vx_t,\hat\vu_t,\eta)\|^2 \bigg]\\
=&\E\bigg[\Phi_\delta(\vx_t) - \Big(\frac{\eta}{2} - \frac{c\eta^2 L\sqrt{d_x}}{2\delta}\Big)\|G(\vx_t,\hat\vu_t,\eta)\|^2 +\frac{\eta }{2}\| \hat\vu_t-\vu^*_t\|^2\\
& \quad+ \eta \langle \vu^*_t-\nabla \Phi_\delta(\vx_t), G(\vx_t,\hat\vu_t,\eta)\rangle\bigg],
\end{split}
\end{align}
where the first inequality is based on the $cL\sqrt{d_x}\delta^{-1}$-smoothness of function $\Phi_\delta(\vx)$ according to Proposition 2.2 of \cite{lin2022gradient}; the second inequality is based on Lemma \ref{lemma: g&G_2} and Young's inequality.

We then bound the term $ \E[\langle \vu^*_t-\nabla \Phi_\delta(\vx_t), G(\vx_t,\hat\vu_t,\eta)\rangle]$ as follows
\begin{align}\label{internal eq 2}
\begin{split}
&\E\left[ \langle \vu^*_t-\nabla \Phi_\delta(\vx_t), G(\vx_t,\hat\vu_t,\eta)\rangle \right]
\\
=& \E\big [ \langle \vu^*_t-\nabla \Phi_\delta(\vx_t), G(\vx_t,\nabla \Phi_\delta(\vx_t),\eta)\rangle  \\
&\quad + \langle\vu^*_t-\nabla\Phi_\delta(\vx_t), G(\vx_t,\vu^*_t,\eta)-G(\vx_t,\nabla \Phi_\delta(\vx_t),\eta)\rangle \\
&\quad + \langle\vu^*_t-\nabla\Phi_\delta(\vx_t),
G(\vx_t,\hat\vu_t,\eta)-G(\vx_t,\vu^*_t,\eta)
\rangle \big]
\\
\le& 
\E\left[\|\vu^*_t-\nabla\Phi_\delta(\vx_t)\|\|G(\vx_t,\vu^*_t,\eta)-G(\vx_t,\nabla \Phi_\delta(\vx_t),\eta)\| 
+\|\vu^*_t-\nabla\Phi_\delta(\vx_t)\|\|\vu^*_t-\hat\vu_t\|\right]
\\
\le& 
\E\left[ \|\vu^*_t-\nabla\Phi_\delta(\vx_t)\|^2
+\frac{1}{2}\|\vu^*_t-\nabla\Phi_\delta(\vx_t)\|^2+\frac{1}{2}\|\vu^*_t-\hat\vu_t\|^2\right].
\end{split}
\end{align}
where the first inequality is based on the fact~$\mathbb{E}[\vu^*_t]=\nabla \Phi_\delta(\vx_t)$, Cauchy--Schwarz inequality, and Lemma \ref{lemma: g&G_1}; the second inequality is based on Lemma \ref{lemma: g&G_1} and AM–GM inequality.

By substituting equation (\ref{internal eq 2}) into equation (\ref{inter eq 1}) and rearranging the terms, it holds
\begin{align*}
&\left(\frac{\eta}{2} - \frac{c\eta^2 L\sqrt{d_x}}{2\delta}\right)\,
\mathbb{E}\bigg[\|G(\vx_t,\hat\vu_t,\eta)\|^2\bigg]
\\
\le&\mathbb{E}\bigg[\Phi_\delta(\vx_t)-\Phi_\delta(\vx_{t+1}) + \frac{3\eta}{2} \|\vu_t^*-\nabla \Phi_\delta(\vx_t)\|^2+\eta\|\vu^*_t-\hat\vu_t\|^2\bigg].
\end{align*}

Taking the average on above inequality over $t=0$ to $T-1$, we obtain
\begin{align*}
&\frac{1}{T}\sum_{t=0}^{T-1}\left(\frac{\eta}{2} - \frac{c\eta^2 L\sqrt{d_x}}{2\delta}\right)\,
\mathbb{E}\bigg[\|G(\vx_t,\hat\vu_t,\eta)\|^2\bigg]\\
\le& \mathbb{E}\bigg[\frac{\Phi_\delta(\vx_0)-\Phi_\delta(\vx_T)}{T} + \frac{3\eta}{2T}\sum_{t=0}^{T-1}\|\vu^*_t-\nabla \Phi_\delta(\vx_t)\|^2+\eta\|\vu^*_t-\hat\vu_t\|^2\bigg].
\end{align*}
Taking $\eta=\delta/(2cL\sqrt{d_x}\,)$,   
$\vx_{\mathrm{out}}=\vx_j$, and $\hat\vu_{\mathrm{out}}=\vu_j$
with $j\sim{\rm Unif}(\{0,\dots,T-1\})$, it holds
\begin{align}\label{control of the gradient mapping with hat u}
\begin{split}
&\mathbb{E}\|G(\vx_{\mathrm{out}},\hat\vu_{\mathrm{out}},\eta)\|^2\\
\le& \mathbb{E}\left[\frac{8cL\sqrt{d_x}\big(\Phi_\delta(\vx_0)-\Phi_\delta(\vx_T)\big)}{T\delta} + \frac{6}{T}\sum_{t=0}^{T-1}\|\vu^*_t-\nabla \Phi_\delta(\vx_t)\|^2+4\|\vu^*_t-\hat\vu_t\|^2\right] \\
\le& \frac{8cL\sqrt{d_x}(\Delta+2\delta L)}{T\delta} + \frac{96\sqrt{2\pi }d_xL^2}{b}+ \frac{8d_x^2L^2 \tilde\epsilon}{\mu \delta^2},
\end{split}
\end{align}
where the second inequality is based on 
equation \eqref{C1}, \eqref{Variance of u^*} and \eqref{control of the difference between u, hat u}.

Let $\vu_{\mathrm{out}}^*=\vu_j^*$ with $j\sim{\rm Unif}(\{0,\dots,T-1\})$.
We apply Young's inequality to achieve
\begin{align}\label{eq:G_out_eta}
\begin{split}    
&\mathbb{E}\left[\|G(\vx_{\mathrm{out}},\nabla \Phi_\delta(\vx_{\mathrm{out}}),\eta)\|^2\right]\\
\le & \mathbb{E}\left[2\| G(\vx_{\mathrm{out}},\vu^*_{\mathrm{out}},\eta)-G(\vx_{\mathrm{out}},\nabla \Phi_\delta(\vx_{\mathrm{out}}),\eta)\|^2 + 2\|G(\vx_{\mathrm{out}},\vu^*_{\mathrm{out}},\eta)\|^2\right]\\
\le &
\mathbb{E}\,\big[2\| G(\vx_{\mathrm{out}},\vu^*_{\mathrm{out}},\eta)-G(\vx_{\mathrm{out}},\nabla \Phi_\delta(\vx_{\mathrm{out}}),\eta)\|^2 \\
&~~~ + 4\|G(\vx_{\mathrm{out}},\vu^*_{\mathrm{out}},\eta)-G(\vx_{\mathrm{out}},\hat\vu_{\mathrm{out}},\eta)\|^2
 +4\|G(\vx_{\mathrm{out}},\hat\vu_{\mathrm{out}},\eta)\|^2\big]\\
\le&
\mathbb{E}\left[2\| \vu^*_{\mathrm{out}}-\nabla \Phi_\delta(\vx_{\mathrm{out}})\|^2 + 4\| \vu^*_{\mathrm{out}}-\hat\vu_{\mathrm{out}}\|^2
+4\|G(\vx_{\mathrm{out}},\hat\vu_{\mathrm{out}},\eta)\|^2\right]\\
\le& \frac{32cL\sqrt{d_x}(\Delta+2\delta L)}{T\delta} + \frac{416\sqrt{2\pi }d_xL^2}{b}+\frac{40 d_x^2L^2 \tilde\epsilon}{\mu \delta^2},
\end{split}
\end{align}
where the third inequality is obtained by Lemma \ref{lemma: g&G_1}
and the last step is based on equations ~\eqref{Variance of u^*}, ~\eqref{control of the difference between u, hat u} and~\eqref{control of the gradient mapping with hat u}.

We further take
\begin{align}\label{eq:prameter-Phi}    
\tilde\epsilon=\frac{\mu\delta^2\epsilon^2}{120d_x^2 L^2}, \qquad
T=\left\lceil\frac{96c\sqrt{d_x}L(\Delta+2\delta L)}{\delta\epsilon^2}\right\rceil, \qquad \text{and} \qquad b=\frac{1248\sqrt{2\pi}d_xL^2}{\epsilon^2},
\end{align}
which implies
\begin{align}\label{eq:G-Phi-epsilon}    
\BE\left[\|G(\vx_{\mathrm{out}},\nabla \Phi_\delta(\vx_{\mathrm{out}}),\eta)\|\right]
\leq \left(\mathbb{E}\left[\|G(\vx_{\mathrm{out}},\nabla \Phi_\delta(\vx_{\mathrm{out}}),\eta)\|^2\right]\right)^{1/2} \overset{\eqref{eq:G_out_eta},\,\eqref{eq:prameter-Phi}}\le \epsilon
\end{align}
and the corresponds $K$ is
\begin{align*}
K=\left\lceil\frac{7680\sqrt{2\pi} d_x^2d_y L^4}{\mu^2\delta^2\epsilon^2}\right\rceil.
\end{align*}
Combining with Theorem 3.1 of \citet{lin2022gradient}, we know that $\vx_{\mathrm{out}}$ is a $(\eta,\delta,\epsilon)$-GGSP of the primal function $\Phi(\vx)$ and the total number of stochastic zeroth-order oracle calls is
\[
T\cdot b\cdot K
= \fO\left(\frac{ d_x^{7/2} d_y L^7(\Delta+\delta L)}{\mu^2\delta^3 \epsilon^6}\right).
\]
\end{proof}

\subsection{Proof of Proposition \ref{gradient distance of the strongly concave approximation}}
\begin{proof}
For the ease of presentation, we denote 
\begin{align}\label{tilde phi_mu}
\tilde{\mu}_\Phi \triangleq \frac{\delta\epsilon}{d_x D_y^2},
\end{align}
then we can write
\begin{align*}
\tilde\Phi(\vx)=\max_{\vy\in\fY} \Big(f(\vx,\vy)-\frac{\tilde{\mu}_\Phi}{2}\|\vy-\vy_0\|^2\Big).
\end{align*}
We first bound the difference between $\Phi(\vx)$ and $\tilde\Phi(\vx)$.
The definitions of $\tilde\Phi(\vx)$ and $\Phi(\vx)$ implies
that for all $\vx\in\fX$ and $\hat\vy\in\operatorname{argmax}_{\vy\in\fY} f(\vx,\vy)$, it holds
\begin{align}\label{eq: lower}
\begin{split}
\tilde\Phi(\vx) 
& \ge f(\vx,\hat\vy)-\frac{\tilde\mu_\Phi}{2}\|\hat\vy-\vy_0\|^2
 \ge f(\vx,\hat\vy)-\frac{\tilde\mu_\Phi D_y^2}{2} 
= \Phi (\vx)-\frac{\tilde\mu_\Phi D_y^2}{2}.
\end{split}
\end{align}
Additionally, for all $\vx\in\fX$ and $\tilde\vy=\operatorname{argmax}_{\vy\in\fY} \big(f(\vx,\vy)-\frac{\tilde\mu_\Phi}{2}\|\hat\vy-\vy_0\|^2\big)$, it holds
\begin{align}\label{eq: upper}
   \tilde\Phi(\vx)
=f(\vx,\tilde\vy)-\frac{\tilde\mu_\Phi}{2}\|\tilde\vy-\vy_0\|^2
\le f(\vx,\tilde\vy)\le \Phi(\vx).
\end{align}
Combining equations \eqref{eq: lower} and \eqref{eq: upper}, we achieve
\begin{align}\label{difference between Phi and Phi_mu}
    |\Phi(\vx)-\tilde\Phi(\vx)|\le \frac{\tilde\mu_\Phi D_y^2}{2} .
\end{align}

We then bound the difference between $\nabla\Phi(\vx)$ and $\nabla\tilde\Phi(\vx)$.
According to Lemma 2.1 of \cite{flaxman2005online}, it holds  
\begin{align}\label{eq:Phi-delta-epsilon}
    \nabla \Phi_\delta (\vx)=\frac{d_x}{\delta} \E_{\vw}\left[\Phi(\vx+\delta\vw)\vw\right]
    \quad\text{and}\quad
    \nabla \tilde\Phi_\delta (\vx)=\frac{d_x}{\delta} \E_{\vw}\left[\tilde\Phi(\vx+\delta\vw)\vw\right].
\end{align}
where $\vw\sim{\rm Unif}(\BS^{d-1})$.
This implies for all $\vx\in\fX$, we have
\begin{align*}
\big\|\nabla\Phi_\delta(\vx) - \nabla\tilde\Phi_\delta(\vx)\big\|
= & \left\|\frac{d_x}{\delta}\,\mathbb{E}_{\vw}\left[\left(\Phi(\vx + \delta \vw)- \tilde\Phi(\vx + \delta \vw)\right)\vw\right]\right\|\\
\le & \frac{d_x}{\delta}\,\sup_{\hat\vx\in\fX}\left|\Phi(\hat\vx) - \tilde\Phi(\hat\vx)\right|
\overset{\eqref{difference between Phi and Phi_mu}}{\le} \frac{\tilde\mu_\Phi d_x D_y^2}{2\delta}.
\end{align*}
According to Lemma \ref{lemma: g&G_1}, it holds
\begin{align*}
    \left\|G(\vx,\nabla\Phi_\delta(\vx),\eta) - G(\vx,\nabla\tilde\Phi_\delta(\vx),\eta)\right\|\le
    \left\|\nabla\Phi_\delta(\vx) - \nabla\tilde\Phi_\delta(\vx)\right\|
    \overset{\eqref{eq:Phi-delta-epsilon}}{\le} \frac{\tilde\mu_\Phi d_x D_y^2}{\delta}\overset{\eqref{tilde phi_mu}}{\le}\frac{\epsilon}{2}.
\end{align*}

\end{proof}

\subsection{Proof of Corollary \ref{thm: complexity of phi: concave}}
\begin{proof}
Following the derivation of equation \eqref{eq:tilde-F-L}, we know that
\[ \tilde F_\Phi(\vx,\vy;\vxi) \triangleq F(\vx,\vy;\vxi) - \frac{\delta\epsilon}{2 d_x D_y^2}\|\vy-\vy_0\|^2 \]
is $\tilde{L}_\Phi$ mean-squared Lipschitz continuous with respect to $(\vx,\vy)$, where 
\[\tilde{L}_\Phi=L+\frac{\delta\epsilon}{d_x D_y}.\]
In addition, the concavity of $f(\vx,\vy)$ on $\vy$ implies the function
\begin{align*}
f(\vx,\vy)-\frac{\delta\epsilon}{2 d_x D_y^2}\|\vy-\vy_0\|^2    
\end{align*}
is $\tilde \mu_\Phi$-strongly concave in $\vy$ with parameter
\begin{align*}
\tilde{\mu}_\Phi=\frac{\delta\epsilon}{d_xD_y^2}.    
\end{align*}
Therefore, we can apply NL-PGFDA (Algorithm \ref{alg:Phi}) with the stochastic zeroth-order oracle $\tilde F_\Phi(\vx,\vy;\vxi)$ to solve the regularized minimax problem \eqref{prob:regularized2}.
Following the proof of Theorem~\ref{Theorem of Phi} by replacing 
$F$, $L$, $\mu$, and $\epsilon$ with $\tilde F_\Phi$, $\tilde L_\Phi$, $\tilde \mu_\Phi$ and $\epsilon/2$, then the result \eqref{eq:G-Phi-epsilon} indicates 
\begin{align}\label{eq:G-tilde-epsilon} 
\BE\left[\|G(\vx_{\mathrm{out}},\nabla \tilde\Phi_\delta(\vx_{\mathrm{out}}),\eta)\|\right]
\leq  \frac{\epsilon}{2}.
\end{align}
where
\begin{align*}    
\eta=\Theta\left(\frac{\delta}{d_x^{1/2}\tilde L_\Phi}\right)
=\Theta\left(\frac{\delta d_xD_y}{d_x^{3/2}D_y L+d_x^{1/2}\delta\epsilon}\right).
\end{align*}
Hence, it holds
\begin{align*}
& \left\|G(\vx_{\rm out},\nabla\Phi_\delta(\vx_{\rm out}),\eta)\right\| \\
\le & \left\|G(\vx_{\rm out},\nabla\Phi_\delta(\vx_{\rm out}),\eta) - G(\vx_{\rm out},\nabla\tilde\Phi_\delta(\vx_{\rm out}),\eta)\right\|
    +\left\|G(\vx_{\rm out},\nabla\tilde\Phi_\delta(\vx_{\rm out}),\eta)\right\|
\le \frac{\epsilon}{2}+\frac{\epsilon}{2}=\epsilon,
\end{align*}
where the last inequality is based on 
Proposition \ref{gradient distance of the strongly concave approximation} and equation  \eqref{eq:G-tilde-epsilon}.
Combining with Theorem 3.1 of \citet{lin2022gradient}, we know that $\vx_{\rm out}$ is a $(\eta,\delta,\epsilon)$-GGSP of the function $\Phi(\vx)$.
Additionally, the required overall stochastic zeroth-order oracle complexity is
\begin{align*}
\fO\left(\frac{ d_x^{7/2} d_y \tilde L_\Phi^7(\Delta+\delta \tilde L_\Phi)}{\tilde\mu_\Phi^2\delta^3 \epsilon^6}\right)
=
\fO\left(\frac{ d_x^{11/2} d_y \big(L+\delta\epsilon/(d_x D_y)\big)^7D_y^4 \big(\Delta+\delta L+\delta^2\epsilon/(d_x D_y)\big)}{\delta^5 \epsilon^8}\right).
\end{align*}
\end{proof}

\end{document}